\documentclass[11pt]{article}

\usepackage{amstext, amsmath,latexsym,amsbsy,amssymb}
\usepackage{enumitem}
\newtheorem{definition}{Definition}[section]
\newtheorem{assumption}{Assumption}[section]
\newtheorem{proposition}{Proposition}[section]
\newtheorem{remark}{Remark}[section]
\newenvironment{proof}{{\bf Proof\ }}{\QED\\}
\newtheorem{lemma}{Lemma}[section]

\numberwithin{equation}{section}
\newtheorem{theorem}{Theorem}[section]

\newcommand{\QED}{\hspace*{\fill}\rule{2.5mm}{2.5mm}}

\allowdisplaybreaks
\begin{document}
\title{Nonlinear approximation theory for the homogeneous Boltzmann equation}
\author{\\ Minh-Binh Tran\\
 Basque Center for Applied Mathematics\\
 Mazarredo 14, 48009 Bilbao Spain\\
Email:   tbinh@bcamath.org}
\maketitle
\begin{abstract} A challenging problem in solving the Boltzmann equation numerically is that the velocity space is approximated by a finite region. Therefore, most methods are based on a truncation technique and the computational cost is then very high if the velocity domain is large. Moreover, sometimes, non-physical conditions have to be  imposed on the equation in order to keep the velocity domain bounded. In this paper, we introduce the first nonlinear approximation theory for the Boltzmann equation. Our nonlinear wavelet approximation is non-truncated and based on a nonlinear, adaptive spectral method associated with a new wavelet filtering technique and a new formulation of the equation. A complete and new theory to study the method is provided. The method is proved to converge and perfectly preserve most of the properties of the homogeneous Boltzmann equation. It could also be considered as a general frame work for approximating kinetic integral equations.
\end{abstract}
{\bf Keyword}
{Boltzmann equation, wavelet, adaptive spectral method, Maxwell lower bound, propagation of polynomial moments, propagation of exponential moments,  convergence to equilibrium, conservation laws, nonlinear approximation theory, numerical stability, convergence theory, wavelet filter. \\{\bf MSC:} {82C40, 65M70, 76P05, 41A46, 42C40.}
\tableofcontents
\section{Introduction}
Numerical resolution methods for the Boltzmann equation plays a very important role in the practical an theoretical study of the theory of rarefied gas. The main difficulty in the approximation of the Boltzmann equation is due to the multidimensional structure of the Boltzmann collision operator. The best known numerical method for Boltzmann equation is the Direct Simulation Monte Carlo technique by Bird \cite{Bird:1995:MGD}. The method is very efficient and preserves the main physical property of the equation; however, it is quite expensive.
\\ After the early work of Carleman (\cite{Carleman:1933:TEI}, \cite{Carleman:1957:PMC}),  Discrete Velocity Models - DVMs  has been developed as one of the main classes of deterministic algorithms to resolve the Boltzmann equation numerically (\cite{Tartar:1976:EGP}, \cite{Bony:1987:SGB}, \cite{Bony:1991:EGD}, \cite{BobylevCercignani:1999:DVM}, \cite{PatkowskiIllner:1988:DVM}, \cite{BobylevPalczewski:1995:OAB}, \cite{PalczewskiSchneider:1998:ESC}, \cite{Buet:1996:DVS},  \cite{KowalczykPalczewski:2008:NSB}). They are based on a Cartesian grid in velocity and a discrete collision operator, which is a nonlinear system of conservation laws
\begin{equation}\label{EqDVMs}
\frac{\partial f_i}{\partial t} +v_i.\nabla_x f_i=Q_i(f,f), (x,t)\in \Omega\times \mathbb{R}, v_i\in V,
\end{equation}
where $Q_i$ is the discrete collision operator. The velocity set $V$ is assumed to be a part of the regular grid
$$\mathcal{Z}_\Delta=\Delta\mathbb{Z}^3=\{\Delta (i_1,i_2,i_3)~~|~~(i_1,i_2,i_3)\in\mathbb{Z}^3\},$$
contained in a truncated set 
\begin{equation}\label{DVMsTruncation}
\mathcal{V}_\Delta^R=\{\Delta (i_1,i_2,i_3)~~|~~(i_1,i_2,i_3)\in\mathbb{Z}^3; |i_1|, |i_2|, |i_3|<R\}.
\end{equation}
 However, in order to guarantee the convergence, the mesh size $\Delta$ needs to be very small and the parameter $R$ needs to be large. DVMs are then very expensive, especially if we want to observe the behaviour of the solution for large velocities, which is a very important issue in the study of the Boltzmann equation.  Even for small velocities, the methods are quite expensive. The models were proved to be consistent (\cite{PalczewskSchneideriBobylev:1997:CRD}, \cite{FainsilberKurlbergWennberg:2006:LPC}), i.e. the discrete collision term could be seen as an approximation of the real collision operator. In \cite{Mischler:1997:CDV},  \cite{PanferovHeintz:2002:NCD}, \cite{DesvillettesMischler:1996:ASA} the approximate solutions are proved to converge weakly to the solution of the main equation  when $\Delta$ tends to $0$ and $R$ tends to infinity by DiPerna-Lions theory (\cite{DiPernaLions:1989:OCP}). However, because of this way of truncating the mesh from $\mathcal{Z}_\Delta$ to $\mathcal{V}_\Delta^R$, it is not easy to obtain an accuracy estimate of errors between the approximate solutions and the global solution on the entire non-truncated space.
\\ The second deterministic approximation is the Fourier Spectral Methods - FSMs, which were first introduced in \cite{PareschiPerthame:1996:FSM} inspired by spectral methods in fluid mechanics. The methods were later developed in several works, where a new way of accelerating the algorithms was also introduced  (\cite{PareschiRusso:2000:NSB}, \cite{MouhotPareschi:2004:FMB}, \cite{PareschiRusso:2005:INA}, \cite{FilbetMouhotPareschi:2006:SBE},  \cite{MouhotPareschi:2006:FAC}, \cite{FilbetMouhotPareschi:2006:SBE}, \cite{MouhotPareschi:2004:FMB}, \cite{PareschiRusso:2000:OSS}, \cite{PareschiToscaniVillani:2003:SMB}, \cite{Filbet:2012:ODA}, \cite{FilbetRusso:2004:ANM}, \cite{FilbetHuJin:2012:NSQ}, \cite{MarkowichPareschi:2005:FCE}). The analysis of the methods was provided in \cite{FilbetMouhot:2011:ASM}. The idea of the methods is to truncate the Boltzmann equation on the velocity space and periodize the solution on this new bounded domain. To illustrate this idea, we consider the equation
\begin{equation}\label{BoltzmannTruncated} 
\frac{\partial f}{\partial t}+v.\nabla_x f=Q^R(f,f), (x,v,t)\in \Omega\times (-R,R)^3\times \mathbb{R}, 
\end{equation}
where $Q^R$ is the truncated collision operator and $f$ is periodic on $(-R,R)^3$. Since $f$ is periodic on $(-R,R)^3$, we can write an approximation $f_N$ of $f$ in terms of Fourier series
$$f_N=\sum_{k=-(N,N,N)}^{(N,N,N)}\hat{f}_k\exp\left(-i\frac{\pi}{R}k.v\right),$$
which leads to a system of ODEs
$$\partial_t f_N=P_N Q^R(f_N,f_N),$$
or
$$\int_{(-R,R)^3}\left(\frac{\partial f_N}{\partial t}-Q^R(f_N,f_N)\right)\exp\left(-i\frac{\pi}{R}k.v\right)dv=0.$$
The approximation $(\ref{BoltzmannTruncated})$ was also used to derive DVMs for the Boltzmann equation (\cite{MouhotPareschiRey:2013:CDA}) through Carleman's representation of the equation. However, it is proved (\cite{FilbetMouhot:2011:ASM}) that $(\ref{BoltzmannTruncated})$ has a constant function equilibrium state, which is totally different from a normal equilibrium state of the Boltzmann equation. Therefore, solving $(\ref{BoltzmannTruncated})$ does not give us the real solution of the Boltzmann equation. 
\\ The major problem with deterministic methods like DVMs and FSMs that use a fixed discretization in the velocity space is that the velocity space is approximated by a finite region. Physically, the velocity space is $\mathbb{R}^3$ and even if the initial condition is compactly supported, the collision operator does not preserve this property. The collision operator indeed spreads out the supports by a factor $\sqrt2$ (see \cite{PulvirentiWennberg:1997:MLB}). Therefore in order to use both  DVMs and FSMs, we have to impose nonphysical conditions to keep the supports of the solutions in the velocity space uniformly compact. For DVMs, we have to remove binary collisions which spread outside the bounded velocity space. This truncation breaks down the convolution structure of the collision operators. For FSMs, the convolution structure is perfectly preserved however we need to add nonphysical binary collisions by a periodized process. In \cite{GambaTharkabhushanam:2009:SLM}, \cite{GambaTharkabhushanam:2010:SBS}, Gamba and Tharkabhushanam proposed another class of FSMs, called Spectral-Lagrangian Methods (SLMs), to preserve the conservation of mass, momentum and energy on the numerical schemes. However, since these are truncated methods, we need to remove the values of the initial conditions with velocities lying outside of the truncated domain. Moreover, similar as the other Fourier-based algorithms, the positivity of the solution could not be automatically preserved. Indeed, the authors proved that if the computation domain is large enough, the negative parts of the approximate solutions are very small in the energy norms. When the computation domains are large, since we need to keep the mesh sizes small, we need to put more grid points and the methods become very expensive. Indeed, for the three deterministic approximations that we mention here DVMs and both classical FSMs and SLMs, the meshes are non-adaptive and therefore they are expensive to carry out computations on large domains. Another drawback of SLMs is that according to the theory, the initial conditions need to be regular enough to guarantee the convergence.
\\ In order to be able to construct numerical schemes, it is natural that we require the computation domain to be bounded. To solve partial differential equations on unbounded domains, there is a famous method called Absorbing Boundary Conditions (ABCs) of Engquist and Majda (\cite{EngquistMajda:1977:ABC}, \cite{EngquistMajda:1979:RBC}), where the PDE  we need to solve is restricted onto a bounded domain and some artificial boundary conditions are introduced in order to guarantee that solving the equation on the bounded domain with the new boundary conditions and solving that equation on the whole space would give the same result. The construction of the ABCs is based on the partial differential structure of the PDEs and could not be used for integral equations like Boltzmann. Though it does not allow us to build ABCs, the integral structure of the Boltzmann equation gives us another advantage to build an equivalent strategy as ABCs for Boltzmann equation, which we will describe in the following: Consider the following change of variables from $\mathbb{R}^3$ to $(-1,1)^3$ $$v\to \bar{v}=\frac{v}{1+|v|}.$$
Apply this change of variables to the Boltzmann equation, we get a new formulation where the equation is considered on a bounded domain. Since there is no partial differential structure in the Boltzmann equation, the prize that we need to pay after using this change of variable is just the Jacobian of it, which is $\frac{1}{(1+|v|)^{4}}$. This means that we need to introduce a weight $\frac{1}{(1+|v|)^{4}}$ or equivalently $(\sqrt{1+|v|^2})^{-4}$ in all of the norms we consider for the solution of the new equation. Notice that $(\sqrt{1+|v|^2})^{-4}$ is just the momentum with order $-4$, which appears quite a lot in the theory of Boltzmann equation. {\it This new formulation of the Boltzmann equation is discussed in details in Section  $\ref{SectNewFormulationBoltzmann}$.}
\\ After having an equation on a bounded domain through a change of variables technique, we can construct a spectral algorithm similar as in \cite{{PareschiPerthame:1996:FSM}}. However, different from \cite{{PareschiPerthame:1996:FSM}}, we do not use Fourier basis.  Let us explain why. We recall some quantitative properties of the Boltzmann equation that we want to preserve on the numerical schemes. Notice that these properties could not be preserved with previous strategies.
\begin{itemize}
\item {\bf Maxwellian lower bounds} (Carleman \cite{Carleman:1933:TEI}, Pulvirenti and Wennberg\cite{PulvirentiWennberg:1997:MLB}):
if the initial condition $f_0$ satisfies 
$$\int_{\mathbb{R}^3}f_0(v)(1+|v|^2)dv<+\infty,$$
then
$$\forall t_0>0,\exists K_0>0,\exists A_0>0; t\geq t_0 \Longrightarrow \forall v\in\mathbb{R}^3, f(t,v)\geq K_0\exp(-A_0|v|^2),
$$
or
\begin{eqnarray}\label{Property1}
 & &\forall t_0>0,\exists K_0>0,\exists A_0>0; \\\nonumber
 & &t\geq t_0 \Longrightarrow \forall \bar{v}\in(-1,1)^3, f(t,\bar{v})\geq K_0\exp\left(-A_0\left|\frac{\bar{v}}{1-|\bar{v}|}\right|^2\right),
\end{eqnarray}
\item {\bf Production of polynomial moments} (Povzner \cite{Povzner:1962:BEK}, Desvillettes \cite{Desvillettes:1993:SAM}, Wennberg \cite{Wennberg:1997:EDM}, Mischler and Wennberg \cite{MischlerWennberg:1999:SHB}): if the initial condition $f_0$ satisfies 
$$\int_{\mathbb{R}^3}f_0(v)(1+|v|^2)dv<+\infty,$$
then
$$\forall s\geq 2, \forall t_0>0,\sup_{t\geq t_0}\int_{\mathbb{R}^3}f(t,v)(1+|v|^s)<+\infty,
$$
or
\begin{equation}\label{Property2}
\forall s\geq 2, \forall t_0>0,\sup_{t\geq t_0}\int_{(-1,1)^3}f(t,\bar{v})\left(1+\left|\frac{\bar{v}}{1-|\bar{v}|}\right|^s\right)<+\infty.
\end{equation}
\item{\bf  Propagation of exponential moments} (Bobylev, Gamba and Panferov \cite{BobylevGambaPanferov:2004:MIH}, Gamba, Panferov and Villani \cite{GambaPanferovVillani:2009:UMS}, Alonso, Ca{\~n}izo, Gamba and Mouhot\cite{ACGM:2013:NAC}):
 Assume that the initial data satisfies for some $s\in[\gamma,2]$
$$\int_{\mathbb{R}^3}f_0(v)\exp(a_0|v|^s)dv\leq C_0,
$$then there are some constants $C,a>0$ such that
$$\int_{\mathbb{R}^3}f(t,v)\exp(a|v|^s)dv<C,
$$
or
\begin{equation}\label{Property3}
\int_{(-1,1)^3}f(t,\bar{v})\exp\left(a\left|\frac{\bar{v}}{1-|\bar{v}|}\right|^s\right)d\bar{v}<C.
\end{equation}
\end{itemize}
Suppose that we approximate $f$ by its truncated Fourier series
$$f_N=\sum_{k_1,k_2,k_3=(-N,-N,-N)}^{(N,N,N)}\hat{f}_k \exp(i\pi k.\bar{v}),$$
with 
$$\hat{f}_k=\frac{1}{8}\int_{(-1,1)^3}f(\bar{v})\exp(-i\pi k.\bar{v})d\bar{v}.$$
We can see that the approximate solution $f_N$ will never satisfy the properties that we mention above no matter how good $f$ is. The reason is that all components of the Fourier basis, i.e. the $\sin$ and $\cos$ functions are globally and smoothly defined on the whole interval $[-1,1]$ and they encounter singular problems at the extremes $-1$ and $1$. This raises the need for a compactly supported wavelet basis and a new filtering technique. The idea of the technique is simple: we remove compactly supported wavelets which contain the singular points $-1$ and $1$. After having a good orthogonal basis based on this filtering technique, we can apply the normal spectral method to solve the equation. This filtering technique looks like a truncation technique, however it is more natural since we only  need to remove some spectral components and different from classical approximations, the support of our approximate solutions spread to the whole space $\mathbb{R}^3$ gradually after each approximate level $N$. Moreover, it is designed to preserve properties $(\ref{Property1})$, $(\ref{Property2})$ and $(\ref{Property3})$, which are crucial in resolving the Boltzmann equation numerically. We preserve the good properties of both DVMs and FSMs: we are able to keep the convolution structure of the collision operators and do not have to impose a periodic boundary condition on the equation. In addition, our algorithm solves the entire, non-trucated problem with the complexity $(N^{2d})$. {\it The wavelet basis, the filtering technique and the spectral method will be presented in section $\ref{SecWavelet}$. More precisely, our spectral equation is defined in $(\ref{SpectralEquation})$. A comparison between Zuazua's Fourier filtering technique (\cite{Zuazua:2005:POC} and \cite{Zuazua:CNA:2006}) used to preserve the propagation, observation and control of waves and our wavelet filtering technique used to preserve the properties of propagation of polynomial and exponential moments  will also be given in subsection $\ref{SubsecAssumptionPolynomial}$.}
\\ We have provided our first point of view based on Absorbing Boundary Contions. In order to understand better the mechanism of our nonlinear, adaptive spectral method, we now provide a different point of view based on Nonlinear Approximation Theory (\cite{DeVore:1998:NA}, \cite{DaubechiesDeVore:2003:ABF}, \cite{DeVore:2007:OC}). The fundamental problem of approximation theory is to resolve a  complicated function, by simpler, easier to compute functions called "the approximants". The main idea of nonlinear approximation is that the approximants do not come from linear spaces but rather from nonlinear manifolds. An important application of nonlinear approximation is  the adaptive finite element methods for elliptic equations originated in \cite{BabukaSuri:1994:PHP} and developed in  \cite{CohenDahmenDeVore:2001:AWM}, \cite{Cohen:2002:AMP}, \cite{CohenDeVoreNochetto:2012:AFEM}. These methods are based on the idea that fine meshes are put where the solutions are bad and coarse meshes are set where the solutions are good. Coming back to the Boltzmann equation, suppose that we use the Haar wavelet to solve the Botlzmann equation with the new variable $\bar{v}$ on $(-1,1)^3$. As we see later from  $(\ref{HaarFunction})$, $(\ref{HaarPeriodizedBasis})$ and $(\ref{HaarPeriodizedBasisRearrange})$, solving the Boltzmann equation with $\bar{v}$ on $(-1,1)^3$  means that we need construct a mesh by to dividing $(-1,1)^3$ into $2^{3N}$ small cubes. To explain clearer our idea, suppose that we are in one dimension and we need to approximate the solution in a space spanned by the following orthogonal  basis
\begin{eqnarray*}
\left\{ \begin{array}{ll}
\phi_{N,k}(\bar{v})=\chi_{(2^{-N}(2k-1),2^{-N}(2k+1))} \mbox{ for }k=0,\pm 1,\dots,\pm (2^{N-1}-1) ,\vspace{.1in}\\
\phi_{N,2^{N-1}}(\bar{v})=\chi_{(-1,-1+2^{-N})\cup(1-2^{-N},1)}. \end{array}\right.
\end{eqnarray*}
 Let us make the change of variable $\bar{v}\to v=\frac{\bar{v}}{1-|\bar{v}|}$. 
\begin{eqnarray*}
\left\{ \begin{array}{ll}
\phi_{N,k}(v)=\chi_{\left(\min\left\{\frac{2k-1}{2^N-|2k-1|},\frac{2k+1}{2^N-|2k+1|}\right\},\max\left\{\frac{2k-1}{2^N-|2k-1|},\frac{2k+1}{2^N-|2k+1|}\right\}\right)}\vspace{.1in}\\ ~~~~~~~~~~~~~~~~~~~~\mbox{ for }k=0,\pm 1,\dots,\pm (2^{N-1}-1) ,\vspace{.1in}\\
\phi_{N,2^{N-1}}(v)=\chi_{(-\infty,2^N-1)\cup(2^N-1,+\infty)}. \end{array}\right.
\end{eqnarray*}
We can see that solving the Boltzmann equation in $\bar{v}$ on a uniform mesh in $(-1,1)$ is equivalent with solving the Boltzmann equation in ${v}$ on a non-uniform mesh  in $\mathbb{R}$. In other words, the role of the change of variables $v \to\bar{v}$ is to construct a new non-uniform mesh to approximate the Boltzmann equation. The non-uniform mesh has the following interesting property: the larger $|v|$ is the coarser the mesh is, and the smaller $|v|$ is the finer the mesh is. This is crucial, since properties $(\ref{Property1})$, $(\ref{Property2})$ and $(\ref{Property3})$ play the role of a preconditioning analysis in our nonlinear approximation theory: the solution $f$ of the Boltzmann equation behaves like a Maxwellian as $|v|$ large, which means that if $|v|$ is large, we only need a coarse mesh to represent the value of $f$. This is also the main difference between our approximation and classical ones. We can see from the spectral equations $(\ref{SpectralEquation})$ and $(\ref{SpectralEquationSol})$ that the mapping $\varphi$ has a "support-stretching" effect: it maps the wavelet basis $\{\Phi_{N,k}\}$ supported in $(-1,1)^3$ to a new "nonlinear basis" $\{\Phi_{N,k}(\varphi)\}$ supported in the whole space, which are "the approximants" of our nonlinear approximation.   {\it Our method therefore gives a general frame work for solving kinetic integral equations (for example, the coagulation models \cite{Escobedo:2012:CNM}, the quantum Boltzmann equations \cite{EscobedoMischler:2001:OQB},...) numerically:} Suppose that we need to solve the following problem
$$\partial_tf(t,v)= \mathbb{Q}(f,f)(t,v),\mbox{ on  } (0,T)\times\mathbb{R}^3,$$
$$f(0,v)=f_0(v)\mbox{ on  } \mathbb{R}^3,$$
where $\mathbb{Q}$ is some bilinear form. We approximate $f$ as
$$f_N(v)=\sum_{k=(-N,-N,-N)}^{(N,N,N)}a_k{\Phi_{N,k}(\varphi(v))},$$
and get the approximate equation on the unknown $\left(a_k(t)\right)_{k=(-N,-N,-N)}^{(N,N,N)}$
\begin{eqnarray*}
{\frac{\partial a_k}{\partial t}}
&=&\sum_{i,j=(-N,-N,-N)}^{(N,N,N)}{a_ia_j}{<\mathbb{Q}\left({\Phi_{N,i}}(\varphi(v)),{\Phi_{N,j}}(\varphi(v))\right),\Phi_{N,k}(\varphi(v))>}.
\end{eqnarray*}
\\ Moreover, our approximation also provides a general view point for both DVMs and FSMs: FSMs and DVMs are special cases of our approximation using Fourier and Haar wavelet basis. If we take Haar wavelet basis as the spectral basis, our algorithm in this special case then gives an nonlinear, adaptive DVMs for Boltzmann equation, where no direct truncation as $(\ref{DVMsTruncation})$ is imposed and the convolution structure of the collision operator is perfectly preserved. Our new adaptive DVMs is then cheap and it has a spectral accuracy. Therefore, both classical DVMs and FSMs could be seen as special linear and non-adaptive  approximations in our theory. We will come back to this discussion at the end of subsection $\ref{SubSecSpectral}$.
\\ We also introduce a full new analysis to  study theoretically our algorithm. Different with the periodized case (\cite{FilbetMouhot:2011:ASM}) where the truncated Boltzmann collision operator is a bounded bilinear form and the projection of the collision operator onto the subspaces $P_NQ^R$ could be considered as a perturbation of  $Q^R$ with a small term $(Id-P_N)Q^R$, in our case, the analysis is much harder since the collision operator is unbounded. Since $P_NQ$ does not preserve the symmetry of $Q$, the first problem is how we could preserve the conservation laws with this approximation
$$\int_{\mathbb{R}^3}P_NQ^R(f,f)dv=\int_{\mathbb{R}^3}P_NQ^R(f,f)v_idv=\int_{\mathbb{R}^3}P_NQ^R(f,f)|v|^2dv=0.$$ 
Another problem is the preservation the "coercivity" property of the gain part of the collision operator
$$\int_{\mathbb{R}^3}P_NQ^R_+(f,f)fdv\geq \int_{\mathbb{R}^3}|v|^\gamma f^2dv.$$
Notice that this is one of the main advantages of our approximation: perfectly preserve the coercivity property of the gain part of the collision operator. Approximation strategies using Fourier basis could not preserve this coercivity structure because of the effect of the Gibbs phenomenon and the non-positivity of the projection $P_N$ using Fourier basis. We then introduce some new methods to resolve these problems. Since the methods are quite technical, we will leave this discussion for subsection $\ref{SubSecAssumption}$. Based on these new techniques, we can construct a new method for studying our algorithm theoretically.
\begin{itemize}
\item We approximate the projected operator $P_NQ$ by bounded operators $Q_{N,\lambda}$ and prove that the solutions $f_{N,\lambda}$ produced by the bounded operators are uniformly bounded in $L^1$ and $L^2$, moreover they are bounded from below by a Maxwellian. Notice that different from Fourier-based spectral algorithms, our approximate solutions are automatically positive because of the positivity of the wavelet projection.
\item We prove that $f_{N,\lambda}$ converges to $f_N$ as $\lambda$ tends to infinity. Moreover $f_N$ are uniformly bounded in $L^1$, $L^2$ and they are bounded from below by a Maxwellian.
\item We preform  a detailed analysis to prove that $f_N$ converges to $f$ which guarantees the convergence of the algorithm. Notice that different from \cite{Mischler:1997:CDV},  \cite{PanferovHeintz:2002:NCD} our convergence proof is not based on averaging lemma techniques and gives a strong convergence result. 
\end{itemize}
In order to preserve properties $(\ref{Property2})$ and $(\ref{Property3})$, we need to overcome further difficulties. One difficulty is to obtain an $L^1_s$ estimate of the collision operator: how we could perform a Povzner's inequality argument for the projected collision operator $P_NQ^R$ since the structure of the operator is totally different. Another difficulty should be: if we have 
$$\int_{\mathbb{R}^3}f_0(v)\exp(a_0|v|^s)dv\leq C_0,$$
how could we have a uniform bound with respect to $N$ for 
$$\int_{\mathbb{R}^3}P_N(f_0(v))\exp(a_0|v|^s)dv.$$
Due to the Gibbs phenomenon and the non-positivity of the projection $P_N$, Fourier basis is not a good choice to preserve these properties. This is then another advantage of spectral approximations using wavelet basis. We introduce some new methods to overcome these difficulties and we will leave these discussions for sections $\ref{secpolynomialmonemts}$ and $\ref{secexponentialmonemts}$.  
Our main results are: 
\begin{itemize}
\item We prove that the algorithm converges; the energy, mass and momentum of the approximate solution converge to that of the original equation; moreover the approximate solution is bounded from below by a Maxwellian. {\it These are the results of theorems $\ref{ThmBondh_N}$ and $\ref{ThmConvergence}$.}
\item We prove that the polynomial moments of arbitrary orders of the approximate solutions are uniformly bounded. {\it This is the result of theorem $\ref{ThmPropaPolyMoments}$.}
\item We also prove that the exponential moments of the approximate solutions are uniformly bounded. {\it This is the result of theorem $\ref{ThmUpperMaxwellBounds}$.}
\end{itemize}
In other words, the algorithm is proved to converge and preserve all of the properties $(\ref{Property1})$, $(\ref{Property2})$ and $(\ref{Property3})$ as well as the conservation laws. Moreover, we could prove that the approximate solutions belong to $L^2_s$ and converges to the main solution in $L^1_s$ for all $s>0$ (see remark $\ref{RemarkLs1Convergence}$). Since our nonlinear approximation preserves well the structure of the collision operator, we could expect that other properties of the solution could be reflected on the numerical scheme as well.
\\ We also want to mention another important property of the Boltzmann equation: In the paper \cite{Mouhot:2006:RCE}, Mouhot proved that the solution of the Boltzmann equation converges to its equilibrium with the rate $O(\exp(-ct))$. {\it In theorem $\ref{ThmConvergencetoEquilibrium}$, we  prove that this property is preserved by our approximation as well.}
\section{A new formulation of the Boltzmann equation}\label{SectNewFormulationBoltzmann}
\subsection{The Boltzmann equation}
The Boltzmann equation describes the behaviour of a dilute gas of particles when the binary elastic collisions are the only interactions taken into account. It reads
\begin{equation}\label{Boltzmann}
\frac{\partial f}{\partial t} +v.\nabla_x f=Q(f,f), ~~~ x\in \Omega, v\in\mathbb{R}^3,
\end{equation}
where $\Omega\in\mathbb{R}^3$ is the spacial domain and $f:=f(t,x,v)$ is the time-dependent particle distribution function for the phase space. The Boltzmann collision operator $Q$ is a quadratic operator defined as
\begin{equation}\label{CollisionOperator}
Q(f,f)(v)=\int_{\mathbb{R}^d}\int_{\mathbb{S}^{d-1}}B(|v-v_*|,\cos\theta)(f'_*f'-f_*f)d\sigma dv_*,
\end{equation}
where $f=f(v)$, $f_*=f(v_*)$, $f'=f(v')$, $f'_*=f(v'_*)$ and
\begin{eqnarray*}
\left\{ \begin{array}{ll} v'=v-\frac{1}{2}((v-v_*-|v-v_*|\sigma),\vspace{.1in}\\
v'_*=v-\frac{1}{2}((v-v_*+|v-v_*|\sigma),\end{array}\right.
\end{eqnarray*}
with $\sigma\in \mathbb{S}^2$.
\\ In this work, we only assume that $B$ is locally integrable and
\begin{equation}\label{AssumptionB}
B(|u|,\cos\theta)=|u|^\gamma b(\cos\theta),
\end{equation}
where $\gamma\in(0,1)$ and $b$ is a smooth function satisfying
\begin{equation}\label{Assupmtionb1}
\int_0^\pi b(\cos\theta)\sin\theta d\theta<+\infty,
\end{equation}
and assumptions $(2.1)$-$(2.2)$ in \cite{MouhotVillani:2004:RTS}
\begin{equation}\label{Assupmtionb2}
\exists \theta_b>0 \mbox{ such that } supp\{b(\cos\theta)\}\subset\{\theta~~~|~~~\theta_b\leq\theta\leq\pi-\theta_b\}.
\end{equation}
Under these assumptions, the collision operator could be split as
$$Q(f,f)=Q^+(f,f)-L(f)f,$$
with 
$$Q^+(f,f)=\int_{\mathbb{R}^3}\int_{\mathbb{S}^2}B(|v-v_*|,\cos\theta)f'_*f'd\sigma dv_*$$
and
$$L(f)=\int_{\mathbb{R}^3}\int_{\mathbb{S}^2}B(|v-v_*|,\cos\theta)f_*d\sigma dv_*.$$
Formally, Boltzmann collision operator has the properties of conserving mass, momentum and energy
$$\int_{\mathbb{R}^3}Q(f,f)dv=0,$$
$$\int_{\mathbb{R}^3}Q(f,f)vdv=0,$$
$$\int_{\mathbb{R}^3}Q(f,f)|v|^2dv=0,$$
and it satisfies the Boltzmann's H-theorem
$$-\frac{d}{dt}\int_{\mathbb{R}^3}f\log fdv=-\int_{\mathbb{R}^3}Q(f,f)log f dv\geq 0,$$
in which $-\int f \log f$ is defined as the {\it entropy} of the solution. A consequence of the Boltzmann's H-theorem is that any equilibrium distribution function has the form of a locally Maxwellian distribution
$$M(\rho,u,T)=\frac{\rho}{(2\pi T)^{3/2}}\exp\left(-\frac{|u-v|^2}{2T}\right),$$
where $\rho$, $u$, $T$ are the {\it density, macroscopic velocity} and {\it temperature} of the gas
$$\rho=\int_{\mathbb{R}^3}f(v)dv,$$
$$u=\frac{1}{\rho}\int_{\mathbb{R}^3}vf(v)dv,$$
$$T=\frac{1}{3\rho}\int_{\mathbb{R}^3}|u-v|^2f(v)dv.$$
We suppose that the initial datum $f_0$ satisfies $f_0(x,v)\geq 0$ on $\mathbb{R}^6$ and
$$\int_{\mathbb{R}^3}f_0(v)(1+|v|^2)dv<+\infty.$$
We refer to \cite{CercignaniIllnerPulvirenti:1994:TMT} and \cite{Villani:2002:RMT} for further details and discussions on the Boltzmann equation. In this work, we only consider the equation in $\mathbb{R}^3$ but the methodology would be exactly the same for other dimensions.
\subsection{The new formulation}
Different from \cite{PareschiPerthame:1996:FSM}, where a truncation technique is introduced in order to reduce the Boltzmann equation defined on the whole domain into an equation on a bounded domain, we introduce in this section a new formulation of the Boltzmann equation defined on $(-1,1)^3$ based on a change of variables technique. Let us define the following change of variables mapping
$$\varphi:\mathbb{R}^3\to(-1,1)^3,$$
\begin{equation}\label{ChangeVariableMapping}
\varphi(v)=(\varphi_1(v_1),\varphi_2(v_2),\varphi_3(v_3))=\left(\frac{v_1}{1+|v|},\frac{v_2}{1+|v|},\frac{v_3}{1+|v|}\right),
\end{equation}
where we restrict our attention to the norm $|v|=\max\{|v_1|,|v_2|,|v_3|\}$ with $v=(v_1,v_2,v_3)\in\mathbb{R}^3$. The inverse mapping $\varphi^{-1}$ of $\varphi$ reads
$$\varphi^{-1}:(-1,1)^3\to\mathbb{R}^3,$$
$$\varphi^{-1}(\bar{v})=(\varphi_1(\bar{v}_1),\varphi_2(\bar{v}_2),\varphi_3(\bar{v}_3))=\left(\frac{\bar{v}_1}{1-|\bar{v}|},\frac{\bar{v}_2}{1-|\bar{v}|},\frac{\bar{v}_3}{1-|\bar{v}|}\right).$$ 
The idea of our technique is to replace the variable $v$ in $\mathbb{R}^3$ by a new variable in $(-1,1)^3$ through the mapping $\varphi$. Based on this idea, we define the new density function $$g(t,\bar{v})=f(t,\varphi^{-1}(\bar{v})),$$
where $\bar{v}$ is the new variable in $(-1,1)^3$.
\\ With the notice that the Jacobian of the change of variable $\bar{v}\to v$ is $\frac{1}{(1+|v|)^4}$, we have
$$\int_{(-1,1)^3}|g(\bar{v})|^p(1-|\bar{v}|)^{-s-4}d\bar{v}=\int_{(-1,1)^3}|f(\varphi^{-1}(\bar{v}))|^p(1-|\bar{v}|)^{-s-4}d\bar{v}$$
$$=\int_{\mathbb{R}^3}|f(v)|^p(1+|v|)^{s+4}d(\varphi(v))=\int_{\mathbb{R}^3}|f(v)|^p(1+|v|)^sdv.$$
Therefore if $f(v)$ belongs to $L^1$ with the weight $(1+|v|)^s$, then $g(\bar{v})$ belongs to $L^1$ with the weight $(1-|\bar{v}|)^{-s-4}$. Notice that there are several one-to-one mappings that map $\mathbb{R}^3$ to $(-1,1)^3$ however the above property makes us choose to work on $\varphi$. 
\\ We now define 
$$L^{p}_s=\{f~~~|~~~ \int_{\mathbb{R}^3}|f(v)|^p(1+|v|)^{sp}dv<+\infty\},$$
and 
$$\mathcal{L}^{p}_s=\{f~~~|~~~ \int_{(-1,1)^3}|f(\bar{v})|^p(1-|\bar{v}|)^{-sp}d\bar{v}<+\infty\},$$
where $p$, $s$ are real numbers.
For further use, we also need
$$L^{p}(W)=\{f~~~|~~~ \int_{\mathbb{R}^3}|f(v)|^pW^p(v)dv<+\infty\},$$
$$\mathcal{L}^{p}(W')=\{f~~~|~~~ \int_{(-1,1)^3}|f(\bar{v})|^p(W'(\bar{v}))^pd\bar{v}<+\infty\},$$
where $W$, $W'$ are some positive weights.
\\ Moreover, we also need the notation $$<v>=\sqrt{1+|v|^2},~~~\forall v\in\mathbb{R}^3.$$
\\ The Boltzmann equation for $g$ is now
\begin{eqnarray}\label{BoltzmannFirstNewFormulation}
&&\partial_t g(t,x,\bar{v})+\frac{\bar{v}}{1-|\bar{v}|}\nabla_x g(t,x,\bar{v})=\int_{(-1,1)^3}\int_{\mathbb{S}^2}\frac{B(|\varphi^{-1}(\bar{v})-\varphi^{-1}(\bar{v}_*)|,\sigma)}{(1-|\bar{v}_*|)^{4}}\\\nonumber
&&\times\left[g\left(\varphi\left(\frac{\varphi^{-1}(\bar{v})+\varphi^{-1}(\bar{v}_*)}{2}-\sigma\frac{|\varphi^{-1}(\bar{v})-\varphi^{-1}(\bar{v}_*)|}{2}\right)\right)\right.\\\nonumber
&&\left.\times g\left(\varphi\left(\frac{\varphi^{-1}(\bar{v})+\varphi^{-1}(\bar{v}_*)}{2}+\sigma\frac{|\varphi^{-1}(\bar{v})-\varphi^{-1}(\bar{v}_*)|}{2}\right)\right)-g(\bar{v})g(\bar{v}_*)\right]d\sigma d\bar{v}_*,
\end{eqnarray}
which is our first new formulation of the Boltzmann equation.\\
Now define
$$h(t,\bar{v})=g(t,\bar{v})(1-|\bar{v}|)^{-4},$$
which implies 
$$\int_{(-1,1)^3}|h(\bar{v})|(1-|\bar{v}|)^{-s}d\bar{v}=\int_{\mathbb{R}^3}|f(v)|(1+|v|)^sdv.$$
This means if $f$ belongs to $L^1_s$  then $h$ belongs to $\mathcal{L}^1_s$. Notice that we define $h(t,\bar{v})=g(t,\bar{v})(1-|\bar{v}|)^{-4}$ to make our proof simpler, however the theoretical results remain the same if $h(t,\bar{v})=g(t,\bar{v})(1-|\bar{v}|)^{-n},$ with $n$ being any constant in $\mathbb{R}$, $n$ could be $0$. 
\\ The Boltzmann equation for $h$ then reads
\begin{eqnarray}\label{BoltzmannNewFormulationPre}\nonumber
&&\partial_t h(t,x,\bar{v})+\frac{\bar{v}}{1-|\bar{v}|}\nabla_x h(t,x,\bar{v})=\int_{(-1,1)^3}\int_{\mathbb{S}^2}B(|\varphi^{-1}(\bar{v})-\varphi^{-1}(\bar{v}_*)|,\sigma)\\
&&\times\left[\mathcal{C}(\bar{v},\bar{v}_*,\sigma)h\left(\varphi\left(\frac{\varphi^{-1}(\bar{v})+\varphi^{-1}(\bar{v}_*)}{2}-\sigma\frac{|\varphi^{-1}(\bar{v})-\varphi^{-1}(\bar{v}_*)|}{2}\right)\right)\right.\\\nonumber
&&\left.\times h\left(\varphi\left(\frac{\varphi^{-1}(\bar{v})+\varphi^{-1}(\bar{v}_*)}{2}+\sigma\frac{|\varphi^{-1}(\bar{v})-\varphi^{-1}(\bar{v}_*)|}{2}\right)\right)-h(\bar{v})h(\bar{v}_*)\right]d\sigma d\bar{v}_*,
\end{eqnarray}
where 
\begin{eqnarray}\label{DefinitionC}\nonumber
\mathcal{C}(\bar{v},\bar{v}_*,\sigma)&=&\left[1-\varphi\left(\frac{\varphi^{-1}(\bar{v})+\varphi^{-1}(\bar{v}_*)}{2}-\sigma\frac{|\varphi^{-1}(\bar{v})-\varphi^{-1}(\bar{v}_*)|}{2}\right)\right]^4\\\nonumber
& &\times\left[1-\varphi\left(\frac{\varphi^{-1}(\bar{v})+\varphi^{-1}(\bar{v}_*)}{2}+\sigma\frac{|\varphi^{-1}(\bar{v})-\varphi^{-1}(\bar{v}_*)|}{2}\right)\right]^4\\
& &\times(1-|\bar{v}|)^{-4}(1-|\bar{v}_*|)^{-4}.
\end{eqnarray}
Define
\begin{equation}\label{DefinitionB}
\mathcal{B}(\bar{v},\bar{v}_*,\sigma)=B(|\varphi^{-1}(\bar{v})-\varphi^{-1}(\bar{v}_*)|,\sigma),
\end{equation}
we get our second new formulation of the Boltzmann equation
\begin{eqnarray}\label{BoltzmannNewFormulation}\nonumber
&&\partial_t h(t,x,\bar{v})+\frac{\bar{v}}{1-|\bar{v}|}\nabla_x h(t,x,\bar{v})=\int_{(-1,1)^3}\int_{\mathbb{S}^2}\mathcal{B}(\bar{v},\bar{v}_*,\sigma)\\
&&\times\left[\mathcal{C}(\bar{v},\bar{v}_*,\sigma)h\left(\varphi\left(\frac{\varphi^{-1}(\bar{v})+\varphi^{-1}(\bar{v}_*)}{2}-\sigma\frac{|\varphi^{-1}(\bar{v})-\varphi^{-1}(\bar{v}_*)|}{2}\right)\right)\right.\\\nonumber
&&\left.\times h\left(\varphi\left(\frac{\varphi^{-1}(\bar{v})+\varphi^{-1}(\bar{v}_*)}{2}+\sigma\frac{|\varphi^{-1}(\bar{v})-\varphi^{-1}(\bar{v}_*)|}{2}\right)\right)-h(\bar{v})h(\bar{v}_*)\right]d\sigma d\bar{v}_*.
\end{eqnarray}
 The initial datum is now defined 
$$h_0(\bar{v})=(1-|\bar{v}|)^{-4}f_0(\varphi^{-1}(\bar{v})),$$
then
$$\int_{{(-1,1)}^3}h_0(\bar{v})\left(1+\frac{|\bar{v}|^2}{(1-|\bar{v}|)^2}\right)d\bar{v}<+\infty.$$
Let us mention that though the two new formulations seem to be complicated, we only use them for theoretical purposes. Our spectral equation $(\ref{SpectralEquation})$ is based on the former formulation of the equation. 
\section{Approximating the homogeneous Boltzmann equation: an adaptive spectral method}\label{SecWavelet}
In order to preserve the quantitative properties of the Boltzmann equation $(\ref{Property1})$, $(\ref{Property2})$ and $(\ref{Property3})$, as it is pointed out in the introduction, we cannot use the Fourier basis. We will construct a wavelet basis for  $L^2((-1,1)^3)$ in subsection $\ref{SubSecWaveletForL2}$. Our new spectral algorithm is defined in equation $(\ref{SpectralEquation})$ of subsection $\ref{SubSecSpectral}$. In subsection $\ref{SubSecAssumption}$ we discuss about the assumption that we need for the multiresolution analysis and the wavelet filtering technique.
\subsection{Wavelets for $L^2((-1,1)^3)$}\label{SubSecWaveletForL2}
{\it We first construct a wavelet multiresolution analysis for $L^2((-1,1))$.} Let $\phi$ be a {\it positive scaling function} which defines a {\it multiresolution analysis}, i.e., a ladder of embedded approximation subspaces of $L^2(\mathbb{R})$
$$\{0\} \to\dots V_1\subset V_0\subset V_{-1}\dots \to L^2(\mathbb{R})$$
such that $\phi_{j,k}=\{2^{-j/2}\phi(2^{-j}y-k)\}_{k\in\mathbb{Z}}$ constitutes an orthonormal basis for $V_j$. The 
{\it wavelet} $\psi$ is built to characterize the missing details between two adjacent levels of approximation. More concretely, $\{\psi_{j,k}\}_{k\in\mathbb{Z}}=\{2^{-j/2}\psi(2^{-j}y-k)\}_{k\in\mathbb{Z}}$ is an orthonormal basis of $W_j$ where
$$V_{j-1}=V_j\oplus W_j.$$
Multiresolution analysis is a frame work developed by Mallat \cite{Mallat:1989:MAW} and Meyer \cite{Meyer:1988:OFS}, we refer to these two pioneering works or the books \cite{Daubechies:1992:TLW}, \cite{Meyer:1993:WAA} for more details, examples and proofs.
\\ We now follow exactly the construction in \cite[Section 9.3]{Daubechies:1992:TLW} to build the same "periodized wavelets" for $L^2(-1,1)$. Notice that there are other ways besides this way (see \cite{Meyer:1991:OSI}, \cite{CohenDaubechiesJawerth:1993:MAW}). Suppose that the scaling function $\phi$ and the wavelet $\psi$ have reasonable decays, for example $|\phi(y)|, |\psi(y)|\leq C(1+|y|)^{-2-\epsilon},\epsilon>0$. Define
$$\phi_{j,k}^{per}(y)=\sum_{l\in\mathbb{Z}}\phi_{j,k}\left(\frac{y}{2}+l\right);~~~~~~\psi_{j.k}^{per}(y)=\sum_{l\in\mathbb{Z}}\psi_{j,k}\left(\frac{y}{2}+l\right);$$
and
$$V_j^{per}=\overline{Span\{\phi_{j,k}^{per},k\in\mathbb{Z}\}};~~~~~~W_j^{per}=\overline{Span\{\psi_{j,k}^{per},k\in\mathbb{Z}\}}.$$
Similar as \cite[Note 6, Chapter 9]{Daubechies:1992:TLW} we have 
$$\sum_{l\in\mathbb{Z}}\phi\left(\frac{x}{2}+l\right)=1,$$
which implies 
$$\phi_{j,k}^{per}=2^{-j/2}\sum_{l\in\mathbb{Z}}\phi(2^{-j-1}x-k+2^{-j}l)=2^{j/2} \mbox{ for } j\geq 0.$$
These facts mean $V_j^{per}$ for $j\geq 0$ are one dimensional spaces of constant functions. Moreover, similar as \cite[Note 7, Chapter 9]{Daubechies:1992:TLW} we have
$$\sum_{l\in\mathbb{Z}}\psi\left(\frac{x}{2}+l\right)=1,$$
and $W_j^{per}=\{0\}$ for $j\geq 0$. As a consequence, we only need to consider the spaces $V_j^{per}$ and $W_j^{per}$ with $j\leq 0$. According to the property of the multiresolution analysis $V_j$, $W_j$ $\subset$ $V_{j-1}$, then $V_j^{per}$, $W_j^{per}$ $\subset$ $V_{j-1}^{per}$. We also have that $W_j^{per}$ and $V_j^{per}$ are orthogonal 
\begin{eqnarray*}
& &\int_{-1}^1\psi_{j,k}^{per}(y)\phi^{per}_{j,k}(y)dy\\
&=&\sum_{l,l'\in\mathbb{Z}}2^{|j|}\int_{-1}^1\psi(2^{-j-1}y+2^{-j}l-k)\phi(2^{-j-1}y+2^{-j}l'-k')dy\\
&=&\sum_{l,l'\in\mathbb{Z}}2^{|j|}\int_{-1}^1\psi(2^{|j|-1}y+2^{|j|}(l-l')-k)\phi(2^{|j|-1}y-k')dy\\
&=&\sum_{r\in\mathbb{Z}}2^{|j|}\int_{-1}^1\psi(2^{|j|-1}y+2^{|j|}r-k)\phi(2^{|j|-1}y-k')dy=0.
\end{eqnarray*}
Similarly, in $W_j^{per}$, we have also that $\psi_{j,k}^{per}$ and $\psi_{j,k'}^{per}$ are orthogonal. Since $\phi^{per}_{j,k+m2^{|j|}}=\phi_{j,k}^{per}$ $\forall m\in\mathbb{Z}$, then the spaces $V_j^{per}$, $W_j^{per}$ are spanned by the $2^{|j|}$ functions obtained from $k=0,1,\dots,2^{|j|-1}$.  
\\ We therefore have a ladder of multiresolution spaces
$$V_0^{per}\subset V_{-1}^{per}\subset  V_{-2}^{per}\subset\dots\to L^2(-1,1)$$
with $$W_0^{per}\oplus V_0^{per}=V_{-1}^{per}\dots$$
and $\{\phi_{0,0}^{per}\}\cup\{\psi^{per}_{j,k};j\in -\mathbb{N}, k=0,\dots,2^{|j|}-1\}$ is an orthonormal basis of $L^2(-1,1)$.
\\ Define by $S_j\varkappa$ the orthogonal project of a function $\varkappa$ in $L^1(-1,1)$ onto $V_j$, similar as in \cite[Section 9.3]{Daubechies:1992:TLW} we then have the following remarkable property, which is not true with a Fourier basis 
$$\|S_j\varkappa\|_{L^1(-1,1)}\leq C_S \|\varkappa\|_{L^1(-1,1)},$$
and
$$\|S_j\varkappa\|_{L^\infty(-1,1)}\leq C_S \|\varkappa\|_{L^\infty(-1,1)},$$
where $C_S$ is a constant not depending on $j$ and $\varkappa$.   
\\ {\it We now construct a multiresolution analysis for $L^2((-1,1)^3)$.} Define
$$\Psi^{per}_{\bar{j},{k}}(\bar{y})=\psi^{per}_{j_1,k_1}(\bar{y}_1)\psi^{per}_{j_2,k_2}(\bar{y}_2)\psi^{per}_{j_3,k_3}(\bar{y}_3),$$
and
$$\Phi^{per}_{\bar{j},{k}}(\bar{y})=\phi^{per}_{j_1,k_1}(\bar{y}_1)\phi^{per}_{j_2,k_2}(\bar{y}_2)\phi^{per}_{j_3,k_3}(\bar{y}_3),$$
where $\bar{j}=(j_1,j_2,j_3)\in (-\mathbb{N})^3, {k}=(k_1,k_2,k_3)\in\{0,\dots, 2^{|j|}-1\}^3$,  $\bar{y}=(\bar{y}_1,\bar{y}_2,\bar{y}_3)\in (-1,1)^3$. Then $\{\Phi_{0,0}^{per}\}\cup\{\Psi^{per}_{\bar{j},k}\}$ is an orthonormal basis of $L^2((-1,1)^3)$. 
\\ Set $j\in-\mathbb{N}$ and put
$$\mathcal{V}_{|j|}=\{\Phi_{|j|,{k}}(\bar{y})=\Phi_{(j,j,j),{k}}^{per}(\bar{y}), {k}=(k_1,k_2,k_3)\in\{0,\dots,2^{|j|}-1\}^3\}.$$
then
$$\overline{\cup_{|j|\in\mathbb{N}}\mathcal{V}_{|j|}}=L^2((-1,1)^3),$$
which is the ladder of multiresolution spaces for $L^2((-1,1)^3)$ we need.
\\ Define by $P_{|j|}\varrho$ the orthogonal project of a function $\varrho$ in $L^1((-1,1)^3)$ onto $\mathcal{V}_{|j|}$, we also have the following properties 
\begin{equation}\label{L1Projection}
\|P_{|j|}\varrho\|_{L^1((-1,1)^3)}\leq C_P \|\varrho\|_{L^1((-1,1)^3)},
\end{equation}
and
\begin{equation}\label{LinftyProjection}
\|P_{|j|}\varrho\|_{L^\infty((-1,1)^3)}\leq C_P \|\varrho\|_{L^\infty((-1,1)^3)},
\end{equation}
where $C_P$ is a constant not depending on $j$ or $\varrho$. 
\\ Notice that since $\phi$ is a positive function, the following property is true
\begin{equation}\label{PositiveProjection}
\varrho\geq 0\Rightarrow P_{|j|}\varrho\geq 0.
\end{equation}
\subsection{The nonlinear approximation for the homogeneous Boltzmann equation}\label{SubSecSpectral}
First of all, we define the concept of a filter.
\begin{definition}\label{Filter}
Let $\varsigma$ be a function in $\mathcal{V}_N$, $N\in\mathbb{N}$ and 
$$\varsigma=\sum_{k=(0,0,0)}^{(2^N-1,2^N-1,2^N-1)}\varsigma_{N,k}\Phi_{N,k},$$
where $$\varsigma_{N,k}=\int_{(-1,1)^3}\varsigma\Phi_{N,k}d\bar{v}.$$
Set $\mathfrak{A}_N$ to be the set of indices $\{k=(k_1,k_2,k_3)~~~|~~~ 0\leq k_1,k_2,k_3\leq 2^N-1\}$, and suppose that $\mathfrak{B}_N$ is a subset of $\mathfrak{A}_N$. Define 
$$F_N\varsigma=\sum_{k\in\mathfrak{A}_N \backslash \mathfrak{B}_N}\varsigma_{N,k}\Phi_{N,k},$$ 
then $F_N$ is called a filter for $\varsigma$. In other words, a filter eliminates some components when we write $\varsigma$ as a linear combination of the basis $\{\Phi_{N,k}\}_{k\in \mathfrak{A}_N}$ of $\mathcal{V}_N$.
\\ Since our idea is to remove wavelets containing the extreme points of $(-1,1)^3$, we suppose that after the filtering process, $F_N\varsigma$ is supported in $(-\zeta_N,\zeta_N)^3$ with $0<\zeta_N<1$ and $F_N1$ is the characteristic function of $(-\zeta_N,\zeta_N)^3$. Notice that if $\bar{v}$ belongs to $(-\zeta_N,\zeta_N)^3$, then $v=\varphi^{-1}(\bar{v})$ belongs to $(-\frac{\zeta_N}{1-\zeta_N},\frac{\zeta_N}{1-\zeta_N})^3$.
For the sake of simplicity, we  denote 
\begin{equation}\label{NotationSum}
\sum_{k\in\mathfrak{A}_N \backslash \mathfrak{B}_N}=\sum_{k=0}^{2^N-1}.
\end{equation}
\end{definition}
We also suppose that there exist a positive constant $\epsilon^*$ and an open bounded set $\mathcal{D}\subset (-\frac{\zeta_N}{1-\zeta_N},\frac{\zeta_N}{1-\zeta_N})^3$ (for $N$ large enough) with non-zero measure such that 
\begin{equation}\label{InitialGreatThanEpsilon}
f_0>\epsilon^* \mbox{ in }  \mathcal{D}.
\end{equation}
We could assume that $\mathcal{D}$ is a ball. Similar as in \cite{FilbetMouhot:2011:ASM}, we only consider spectral methods for the homogeneous Boltzmann equation, which is written
\begin{equation}\label{BoltzmannHomogeneous}
\frac{\partial f}{\partial t} =Q(f,f), ~~~ v\in\mathbb{R}^3.
\end{equation}
After performing the change of variables, we get the new formulation of the homogeneous Boltzmann equation
\begin{eqnarray}\label{BoltzmannHomoNewFormulation}\nonumber
&&\partial_t h(t,\bar{v})=\int_{(-1,1)^3}\int_{\mathbb{S}^2}\mathcal{B}(\bar{v},\bar{v}_*,\sigma)\\
&&\times\left[\mathcal{C}(\bar{v},\bar{v}_*,\sigma)h\left(\varphi\left(\frac{\varphi^{-1}(\bar{v})+\varphi^{-1}(\bar{v}_*)}{2}-\sigma\frac{|\varphi^{-1}(\bar{v})-\varphi^{-1}(\bar{v}_*)|}{2}\right)\right)\right.\\\nonumber
&&\left.\times h\left(\varphi\left(\frac{\varphi^{-1}(\bar{v})+\varphi^{-1}(\bar{v}_*)}{2}+\sigma\frac{|\varphi^{-1}(\bar{v})-\varphi^{-1}(\bar{v}_*)|}{2}\right)\right)-h(\bar{v})h(\bar{v}_*)\right]d\sigma d\bar{v}_*,
\end{eqnarray}
where $\mathcal{B}$, $\mathcal{C}$ are defined in $(\ref{DefinitionC})$.
\\ Let $N$ be a positive integer and suppose that
$$h_N=\left(1+\frac{|\bar{v}|^2}{(1-|\bar{v}|)^2}\right)^{-1}F_NP_N\left(\left(1+\frac{|\bar{v}|^2}{(1-|\bar{v}|)^2}\right)h\right),$$
where $P_N$ is the orthogonal project onto the space $\mathcal{V}_N$ and $F_N$ is a filter. The reason that we multiply 
$$\left(1+\frac{|\bar{v}|^2}{(1-|\bar{v}|)^2}\right)$$
with $h$ before taking the projection $P_N$ is that 
\begin{eqnarray}\label{EnergyPreserveEquation}\nonumber
\int_{(-1,1)^3}P_N\left[\left(1+\frac{|\bar{v}|^2}{(1-|\bar{v}|)^2}\right)h\right]d\bar{v}&=&\int_{(-1,1)^3}\left(1+\frac{|\bar{v}|^2}{(1-|\bar{v}|)^2}\right)h(\bar{v})P_N (1)d\bar{v}\\
=\int_{(-1,1)^3}\left(1+\frac{|\bar{v}|^2}{(1-|\bar{v}|)^2}\right)h(\bar{v})d\bar{v}&=&\int_{\mathbb{R}^3}f(1+|v|^2)dv,
\end{eqnarray}
which means that we want to preserve the energy of the solution through the projection. We also denote
$$\tilde{h}_N=F_NP_N\left(\left(1+\frac{|\bar{v}|^2}{(1-|\bar{v}|)^2}\right)h\right),$$
 $$\mathcal{P}_N=F_NP_N,$$
 and
 $$\eta(\bar{v})= \left(1+\frac{|\bar{v}|^2}{(1-|\bar{v}|)^2}\right)^{-1}.$$
We therefore have
\begin{eqnarray}\label{SpectralMethod1}\nonumber
&&\partial_t \tilde{h}_N(t,\bar{v})\\\nonumber
&=&Q_N(\tilde{h}_N,\tilde{h}_N)=Q_N^+(\tilde{h}_N,\tilde{h}_N)-Q_N^-(\tilde{h}_N,\tilde{h}_N)\\\nonumber
&:=&\mathcal{P}_N\left\{\int_{(-1,1)^3}\int_{\mathbb{S}^2}\mathcal{B}(\bar{v},\bar{v}_*,\sigma)\right.\\\nonumber
&&\times\left[\eta(\bar{v})^{-1}\mathcal{C}(\bar{v},\bar{v}_*,\sigma)\tilde{h}_N\left(\varphi\left(\frac{\varphi^{-1}(\bar{v})+\varphi^{-1}(\bar{v}_*)}{2}-\sigma\frac{|\varphi^{-1}(\bar{v})-\varphi^{-1}(\bar{v}_*)|}{2}\right)\right)\right.\\
&&\times \tilde{h}_N\left(\varphi\left(\frac{\varphi^{-1}(\bar{v})+\varphi^{-1}(\bar{v}_*)}{2}+\sigma\frac{|\varphi^{-1}(\bar{v})-\varphi^{-1}(\bar{v}_*)|}{2}\right)\right)\\\nonumber
&&\times\eta\left(\varphi\left(\frac{\varphi^{-1}(\bar{v})+\varphi^{-1}(\bar{v}_*)}{2}-\sigma\frac{|\varphi^{-1}(\bar{v})-\varphi^{-1}(\bar{v}_*)|}{2}\right)\right)\\\nonumber
&&\left.\left.\times \eta\left(\varphi\left(\frac{\varphi^{-1}(\bar{v})+\varphi^{-1}(\bar{v}_*)}{2}+\sigma\frac{|\varphi^{-1}(\bar{v})-\varphi^{-1}(\bar{v}_*)|}{2}\right)\right)-\tilde{h}_N(\bar{v})\tilde{h}_N(\bar{v}_*)\eta(\bar{v}_*)\right]d\sigma d\bar{v}_*\right\},
\end{eqnarray}
or equivalently
\begin{eqnarray}\label{SpectralMethod2}\nonumber
&&\partial_t {h}_N(t,\bar{v})\\\nonumber
&=&Q_N(\tilde{h}_N,\tilde{h}_N)=Q_N^+(\tilde{h}_N,\tilde{h}_N)-Q_N^-(\tilde{h}_N,\tilde{h}_N)\\\nonumber
&:=&\mathbb{P}_N\left\{\int_{(-1,1)^3}\int_{\mathbb{S}^2}\mathcal{B}(\bar{v},\bar{v}_*,\sigma)\right.\\
&&\times\left[\mathcal{C}(\bar{v},\bar{v}_*,\sigma){h}_N\left(\varphi\left(\frac{\varphi^{-1}(\bar{v})+\varphi^{-1}(\bar{v}_*)}{2}-\sigma\frac{|\varphi^{-1}(\bar{v})-\varphi^{-1}(\bar{v}_*)|}{2}\right)\right)\right.\\\nonumber
&&\left.\left.\times {h}_N\left(\varphi\left(\frac{\varphi^{-1}(\bar{v})+\varphi^{-1}(\bar{v}_*)}{2}+\sigma\frac{|\varphi^{-1}(\bar{v})-\varphi^{-1}(\bar{v}_*)|}{2}\right)\right)-{h}_N(\bar{v}){h}_N(\bar{v}_*)\right]d\sigma d\bar{v}_*\right\},
\end{eqnarray}
where $\mathbb{P}_N$ is defined
$$\mathbb{P}_N(\varrho)=\eta\mathcal{P}_N(\eta^{-1}\varrho),$$
with some function $\varrho$, and the initial condition is
$$h_{0_{N}}=\mathbb{P}_N(h_0).$$
Suppose that 
$$\tilde{h}_N=\sum_{k=0}^{2^N-1}a_{N,k}\Phi_{N,k},$$
where $$a_{N,k}=\int_{(-1,1)^3}\tilde{h}_N\Phi_{N,k}d\bar{v}.$$
Then $(\ref{SpectralMethod1})$ and $(\ref{SpectralMethod2})$ are equivalent with the following system of ODEs for $k\in\mathcal{A}_N\backslash\mathcal{B}_N$
\begin{eqnarray}\label{ODESystem}\nonumber
&&\partial_t a_{N,k}=\int_{(-1,1)^3}\left\{\int_{(-1,1)^3}\int_{\mathbb{S}^2}\mathcal{B}(\bar{v},\bar{v}_*,\sigma)\right.\\\nonumber
&&\times\left[\mathcal{C}(\bar{v},\bar{v}_*,\sigma)\left(\sum_{l=0}^{2^N-1}a_
{N,l}\Phi_{N,l}\left(\varphi\left(\frac{\varphi^{-1}(\bar{v})+\varphi^{-1}(\bar{v}_*)}{2}-\sigma\frac{|\varphi^{-1}(\bar{v})-\varphi^{-1}(\bar{v}_*)|}{2}\right)\right)\right)\right.\\
&&\times \left(\sum_{l'=0}^{2^N-1}a_{N,l'}\Phi_{N,l'}\left(\varphi\left(\frac{\varphi^{-1}(\bar{v})+\varphi^{-1}(\bar{v}_*)}{2}+\sigma\frac{|\varphi^{-1}(\bar{v})-\varphi^{-1}(\bar{v}_*)|}{2}\right)\right)\right)\\\nonumber
&&\times\eta\left(\varphi\left(\frac{\varphi^{-1}(\bar{v})+\varphi^{-1}(\bar{v}_*)}{2}-\sigma\frac{|\varphi^{-1}(\bar{v})-\varphi^{-1}(\bar{v}_*)|}{2}\right)\right)\\\nonumber
&&\times \eta\left(\varphi\left(\frac{\varphi^{-1}(\bar{v})+\varphi^{-1}(\bar{v}_*)}{2}+\sigma\frac{|\varphi^{-1}(\bar{v})-\varphi^{-1}(\bar{v}_*)|}{2}\right)\right)\eta(\bar{v})^{-1}\\\nonumber
&&\left.\left.-\left(\sum_{l=0}^{2^N-1}a_{N,l}\Phi_{N,l}(\bar{v})\right)\left(\sum_{l'=0}^{2^N-1}a_{N,l'}\Phi_{N,l'}(\bar{v}_*)\right)\eta(\bar{v}_*)\right]d\sigma d\bar{v}_*\right\}\Phi_{N,k}d\bar{v}.
\end{eqnarray}
The resolution of this system gives an approximation of $h$. After solving the system $(\ref{ODESystem})$, we can get a full solution in $\mathbb{R}^3$ by the following mapping
\begin{equation}\label{FNfromhN}
f_N(v)=h_N(\varphi(v))(1+|v|)^{-4}.
\end{equation}
However, system $(\ref{ODESystem})$ is quite complicated and difficult to use in practical computations. We then introduce an equivalent form of it, which is easier to implement
\begin{eqnarray}\label{SpectralEquationPre}
&&\partial_t a_{N,k}\\\nonumber
& =&\sum_{l,l'=0}^{2^N-1}a_{N,l}a_{N,l'}\int_{\mathbb{R}^6\times\mathbb{S}^2}B(|v-v_*|,\sigma)\\\nonumber
& &\left[<v>^2\Phi_{N,l}(\varphi(v'_*))\frac{<v_*'>^{-2}}{(1+|v'_*|)^4}\Phi_{N,l'}(\varphi(v'))\frac{<v'>^{-2}}{(1+|v'|)^4}\right.\\\nonumber
& &-\left.\Phi_{N,l}(\varphi(v_*))\frac{<v_*>^{-2}}{(1+|v_*|)^4}\Phi_{N,l'}(\varphi(v))\frac{<v>^{-2}}{(1+|v|)^4}\right]\Phi_{N,k}(\varphi(v))d\sigma dv_*dv,
\end{eqnarray}
which gives an approximation of ${f(v)}{(1+|v|)^4}<v>^2$. As we mention above, the weight $(1+|v|)^4$ is put just to make the proof simpler, therefore, in practical computations, we can drop it to get the following equivalent system
\begin{eqnarray}\label{SpectralEquation}
&&\partial_t a_{N,k}\\\nonumber
& =&\sum_{l,l'=0}^{2^N-1}a_{N,l}a_{N,l'}\int_{\mathbb{R}^6\times\mathbb{S}^2}B(|v-v_*|,\sigma)\left[\frac{\Phi_{N,l}(\varphi(v'_*))}{<v_*'>^{2}}\frac{\Phi_{N,l'}(\varphi(v'))}{<v'>^{2}}<v>^2\right.\\\nonumber
& &-\left.\frac{\Phi_{N,l}(\varphi(v_*))}{<v_*>^{2}}\Phi_{N,l'}(\varphi(v))\right]\Phi_{N,k}(\varphi(v))d\sigma dv_*dv,~~~\forall k\in\mathcal{A}_N\backslash\mathcal{B}_N,
\end{eqnarray}
which is our {\bf spectral equation} and numerical simulations could be done with this system. The resolution of this system gives us a direct approximation 
\begin{equation}\label{SpectralEquationSol}
\sum_{k=0}^{2^N-1}a_{N,k}\Phi_{N,k}(\varphi(v)),
\end{equation}
of $f(v)<v>^2$. This formulation also gives us a clearer understanding about the mapping $\varphi$: its role is to stretch the support of $\Phi_{N,l}$ from $(-1,1)^3$ to $\mathbb{R}^3$ to get a new "nonlinear basis" on the whole space, which are "the approximants" of our nonlinear approximation. Notice that the weight $<v>^{-2}$ is put simply to preserve the energy of the solution due to $(\ref{EnergyPreserveEquation})$. In our opinion, the scheme
\begin{eqnarray*}
\partial_t a_{N,k}
& =&\sum_{l,l'=0}^{2^N-1}a_{N,l}a_{N,l'}\int_{\mathbb{R}^6\times\mathbb{S}^2}B(|v-v_*|,\sigma)\left[{\Phi_{N,l}(\varphi(v'_*))}{\Phi_{N,l'}(\varphi(v'))}\right.\\\nonumber
& &-\left.{\Phi_{N,l}(\varphi(v_*))}\Phi_{N,l'}(\varphi(v))\right]\Phi_{N,k}(\varphi(v))d\sigma dv_*dv,~~~\forall k\in\mathcal{A}_N\backslash\mathcal{B}_N,
\end{eqnarray*}
should give a good approximation for $f$ as well.
\\ As we mention in the introduction, if we choose $\phi$ to be the Haar scaling function, system $(\ref{SpectralEquation})$ becomes a Discrete Velocity Model, however, different from classical ones, $(\ref{SpectralEquation})$ has an adaptive mesh thanks to the mapping $\varphi$: the larger $|v|$ is, the coarser the mesh is, moreover it preserves the convolution structure of the collision operator. In other words, classical DVMs and Fourier-based spectral methods are in some sense non-adaptive cases of wavelet  spectral approximations. {\it Notice that in $(\ref{ODESystem})$, we take the basis created by $\phi$, but we can take the basis created by $\psi$ as well and the analysis would remain the same. In practice, we could choose the wavelets created by $\phi$ to be a function supported in $(-1/2,1/2)$ and bounded from above and below by $C_1$, $C_2>0$, a good example is the Haar function. }
\\ The existence and uniqueness of a solution of the equivalent systems $(\ref{ODESystem})$, $(\ref{SpectralEquationPre} )$ and $(\ref{SpectralEquation})$ is classical according to the theory of ODEs.  The numerical resolution of $(\ref{ODESystem})$ resolves the Boltzmann equation on the entire space with the same complexity with a normal spectral method. Notice that one of the main advantages of adaptive, nonlinear approximations is that they are much cheaper (\cite{DeVore:2007:OC},\cite{DeVore:1998:NA}).  Moreover, the simplicity of formulation $(\ref{SpectralEquation})$ gives us a lot of chance to construct fast nonlinear spectral algorithms for the equation, however this will be the topic of a different paper. 
\begin{proposition}\label{PropoExistenceSystemODE}
The system $(\ref{ODESystem})$  has a unique solution $\{a_{N,k}\}$ with  $a_{N,k}\in C^1(0,+\infty)$ $\forall k\in \mathcal{A}_N\backslash\mathcal{B}_N$. 
\end{proposition}
\subsection{Assumptions on the multiresolution analysis and the filter}\label{SubSecAssumption}
\subsubsection{Energy preserving property}
\begin{assumption}\label{Assumption1}
Define $\kappa=\eta(\bar{v})^{-1}\mathcal{P}_N\chi_{(-1,1)^3}$, where $\chi_{(-1,1)^3}$ is the characteristic function of $(-1,1)^3$. Set $\varkappa(v)=\kappa(\varphi(v)),$ where $\varphi$ is the change of variables mapping defined in $(\ref{ChangeVariableMapping})$. In order to preserve the energy of the approximate solution, we impose the following assumption on $\mathcal{P}_N$
\begin{equation}\label{AssumpConcret1}
\varkappa(v'_*)+\varkappa(v')-\varkappa(v)-\varkappa(v_*)\leq 0,~~\forall (v,v_*)\in\left(-\frac{\zeta_N}{1-\zeta_N},\frac{\zeta_N}{1-\zeta_N}\right)^6.
\end{equation}
\end{assumption}
{\it We now explain why this assumption is needed in order to preserve the energy of the the approximate solution.} Take $\eta(\bar{v})^{-1}$ as a test function for $(\ref{SpectralMethod1})$
\begin{eqnarray}\label{Assump1Eq1}
&&\int_{(-1,1)^3}\partial_t {h}_N(t,\bar{v})\eta(\bar{v})^{-1}\\\nonumber
&=&\int_{(-1,1)^3}\eta(\bar{v})^{-1}\mathbb{P}_N\left\{\int_{(-1,1)^3}\int_{\mathbb{S}^2}\mathcal{B}(\bar{v},\bar{v}_*,\sigma)\right.\\\nonumber
&&\times\left[\mathcal{C}(\bar{v},\bar{v}_*,\sigma){h}_N\left(\varphi\left(\frac{\varphi^{-1}(\bar{v})+\varphi^{-1}(\bar{v}_*)}{2}-\sigma\frac{|\varphi^{-1}(\bar{v})-\varphi^{-1}(\bar{v}_*)|}{2}\right)\right)\right.\\\nonumber
&&\left.\left.\times {h}_N\left(\varphi\left(\frac{\varphi^{-1}(\bar{v})+\varphi^{-1}(\bar{v}_*)}{2}+\sigma\frac{|\varphi^{-1}(\bar{v})-\varphi^{-1}(\bar{v}_*)|}{2}\right)\right)-{h}_N(\bar{v}){h}_N(\bar{v}_*)\right]d\sigma d\bar{v}_*\right\}d\bar{v}\\\nonumber
&=&\int_{(-1,1)^3}\mathcal{P}_N\left\{\int_{(-1,1)^3}\int_{\mathbb{S}^2}\eta(\bar{v})^{-1}\mathcal{B}(\bar{v},\bar{v}_*,\sigma)\right.\\\nonumber
&&\times\left[\mathcal{C}(\bar{v},\bar{v}_*,\sigma){h}_N\left(\varphi\left(\frac{\varphi^{-1}(\bar{v})+\varphi^{-1}(\bar{v}_*)}{2}-\sigma\frac{|\varphi^{-1}(\bar{v})-\varphi^{-1}(\bar{v}_*)|}{2}\right)\right)\right.\\\nonumber
&&\left.\left.\times {h}_N\left(\varphi\left(\frac{\varphi^{-1}(\bar{v})+\varphi^{-1}(\bar{v}_*)}{2}+\sigma\frac{|\varphi^{-1}(\bar{v})-\varphi^{-1}(\bar{v}_*)|}{2}\right)\right)-{h}_N(\bar{v}){h}_N(\bar{v}_*)\right]d\sigma d\bar{v}_*\right\}d\bar{v}\\\nonumber
&=&\int_{(-1,1)^3}\left\{\int_{(-1,1)^3}\int_{\mathbb{S}^2}\mathcal{B}(\bar{v},\bar{v}_*,\sigma)\right.\\\nonumber
&&\times\left[\mathcal{C}(\bar{v},\bar{v}_*,\sigma){h}_N\left(\varphi\left(\frac{\varphi^{-1}(\bar{v})+\varphi^{-1}(\bar{v}_*)}{2}-\sigma\frac{|\varphi^{-1}(\bar{v})-\varphi^{-1}(\bar{v}_*)|}{2}\right)\right)\right.\\\nonumber
&&\times {h}_N\left(\varphi\left(\frac{\varphi^{-1}(\bar{v})+\varphi^{-1}(\bar{v}_*)}{2}+\sigma\frac{|\varphi^{-1}(\bar{v})-\varphi^{-1}(\bar{v}_*)|}{2}\right)\right)\\\nonumber
&&\left.\left.-{h}_N(\bar{v}){h}_N(\bar{v}_*)\right]d\sigma d\bar{v}_*\right\}\eta(\bar{v})^{-1}\mathcal{P}_N\chi{(-1,1)^3}d\bar{v}\\\nonumber
&=&\int_{(-1,1)^3}\left\{\int_{(-1,1)^3}\int_{\mathbb{S}^2}\mathcal{B}(\bar{v},\bar{v}_*,\sigma)\right.\\\nonumber
&&\times\left[\mathcal{C}(\bar{v},\bar{v}_*,\sigma){h}_N\left(\varphi\left(\frac{\varphi^{-1}(\bar{v})+\varphi^{-1}(\bar{v}_*)}{2}-\sigma\frac{|\varphi^{-1}(\bar{v})-\varphi^{-1}(\bar{v}_*)|}{2}\right)\right)\right.\\\nonumber
&&\times {h}_N\left(\varphi\left(\frac{\varphi^{-1}(\bar{v})+\varphi^{-1}(\bar{v}_*)}{2}+\sigma\frac{|\varphi^{-1}(\bar{v})-\varphi^{-1}(\bar{v}_*)|}{2}\right)\right)\\\nonumber
&&\left.\left.-{h}_N(\bar{v}){h}_N(\bar{v}_*)\right]d\sigma d\bar{v}_*\right\}\kappa(\bar{v})d\bar{v}.
\end{eqnarray}
Define 
$$
f_N(v)=h_N(\varphi(v))(1+|v|)^{-4},$$
then $(\ref{Assump1Eq1})$ is transformed into
\begin{eqnarray*}
&&\frac{d}{dt}\int_{\mathbb{R}^3}(1+|v|^2)f_Ndv\\
&=&\int_{\mathbb{R}^6\times\mathbb{S}^2}B(|v-v_*|,\sigma)[f_{N*}'f_{N}'-f_{N*}f_N]\varkappa(v)d\sigma dv_*dv\\
&=&\frac{1}{2}\int_{\mathbb{R}^6\times\mathbb{S}^2}B(|v-v_*|,\sigma)f_{N*}f_N[\varkappa(v_*')+\varkappa(v')-\varkappa(v_*)-\varkappa(v)]d\sigma dv_*dv\\
&\leq &0,
\end{eqnarray*}
if the assumption $\ref{Assumption1}$ is satisfied. Notice that we only need $(\ref{AssumpConcret1})$ on $\left(-\frac{\zeta_N}{1-\zeta_N},\frac{\zeta_N}{1-\zeta_N}\right)^6$ since if $(v,v_*)$ lies outside this interval, $f_{N*}f_N$=0. A consequence of this inequality is that the energy of the approximate solution is decreasing
$$\int_{\mathbb{R}^3}(1+|v|^2)f_N(t)dv\leq \int_{\mathbb{R}^3}(1+|v|^2)f_N(0)dv.$$
Later, we will prove that the mass, momentum and energy of the approximate solution converge to the mass, momentum and energy of the exact solution.
\\ {\it We now point out an example which satisfies our assumption $\ref{Assumption1}$.} Let us recall the simplest scaling function: Haar function (see \cite{Daubechies:1992:TLW})
\begin{equation}\label{HaarFunction}
\phi(y)=
\left\{ \begin{array}{ll}1 \mbox{ for } -\frac{1}{2}\leq y\leq\frac{1}{2},\vspace{.1in}\\
0 \mbox{ otherwise }.\end{array}\right.
\end{equation}
The corresponding $\phi_{j,k}^{per}$ are
\begin{eqnarray}\label{HaarPeriodizedBasis}
\phi^{per}_{j,k}(y)=
\left\{ \begin{array}{ll}2^{|j|-1}\chi_{(2^{1-|j|}(k-1/2),2^{1-|j|}(k+1/2)))} \mbox{ for } 0\leq k\leq 2^{|j|-1},\vspace{.1in}\\
2^{|j|-1}\chi_{(-2+2^{1-|j|}(k-1/2),-2+2^{1-|j|}(k+1/2)))} \mbox{ for }2^{|j|-1} <k\leq 2^{|j|} ,\vspace{.1in}\\
2^{|j|-1}\chi_{(-1,-1+2^{-|j|})\cup(1-2^{-|j|},1)} \mbox{ for } k=2^{|j|-1}.\end{array}\right.
\end{eqnarray}
We rearrange the indices of $\{\phi_{j,k}^{per}\}$ to get $\{\bar\phi_{j,k}^{per}\}$
\begin{eqnarray}\label{HaarPeriodizedBasisRearrange}
\left\{ \begin{array}{ll}
\bar\phi^{per}_{j,k}(y)=2^{|j|-1}\chi_{(2^{-|j|}(2k-1),2^{-|j|}(2k+1))} \mbox{ for }k=0,\pm 1,\dots,\pm (2^{|j|-1}-1) ,\vspace{.1in}\\
\bar\phi^{per}_{j,2^{|j|-1}}(y)=2^{|j|-1}\chi_{(-1,-1+2^{-|j|})\cup(1-2^{-|j|},1)}. \end{array}\right.
\end{eqnarray}
For the sake of simplicity, we still denote $\bar\phi^{per}_{j,k}$ by $\phi^{per}_{j,k}$. Then
$$\Phi_{|j|,{k}}(\bar{y})=\Phi^{per}_{j,{k}}(\bar{y})=\phi^{per}_{j,k_1}(\bar{y}_1)\phi^{per}_{j,k_2}(\bar{y}_2)\phi^{per}_{j,k_3}(\bar{y}_3),$$
where ${k}=(k_1,k_2,k_3)\in\{-2^{|j|-1}+1,\dots ,2^{|j|-1}\}^3$, $j\in-\mathbb{N}$,
and
$$\mathcal{V}_{|j|}=\{\Phi_{|j|,{k}}(\bar{y}), {k}=(k_1,k_2,k_3)\in\{-2^{|j|-1}+1,\dots ,2^{|j|-1}\}^3\}.$$
Let $\hat{k}_{|j|}$ be an integer in $\{0,\dots,2^{|j|-1}-1\}$. Let $\varsigma$ be any function in $\mathcal{V}_{|j|}$, $j\in-\mathbb{N}$ and 
$$\varsigma=\sum_{k=(-2^{|j|-1}+1,-2^{|j|-1}+1,-2^{|j|-1}+1)}^{(2^{|j|-1},2^{|j|-1},2^{|j|-1})}\varsigma_{|j|,k}\Phi_{|j|,k}=:\sum_{k=-2^{|j|-1}+1}^{2^{|j|-1}}\varsigma_{|j|,k}\Phi_{|j|,k},$$
where $$\varsigma_{|j|,k}=\int_{(-1,1)^3}\varsigma\Phi_{|j|,k}d\bar{v},$$
define the filter
\begin{equation}\label{HaarFilter}
F_{|j|}\varsigma=\sum_{k=(-\hat{k}_{|j|},-\hat{k}_{|j|},-\hat{k}_{|j|})}^{(\hat{k}_{|j|},\hat{k}_{|j|},\hat{k}_{|j|})}\varsigma_{|j|,k}\Phi_{|j|,k}=:\sum_{k=-\hat{k}_{|j|}}^{\hat{k}_{|j|}}\varsigma_{|j|,k}\Phi_{|j|,k}.
\end{equation}
In other words, the filter $F_{|j|}$ eliminates all of the components with indices ${k}=(k_1,k_2,k_3)$ where $\max\{|k_1|,|k_2|,|k_3|\}>\hat{k}_{|j|}$.
\begin{remark} In this paper, usually we choose the Haar multiresolution analysis as an example to illustrate our theory since it is the simplest multiresolution analysis. However, all of the results in our paper should remain the same if we choose $\phi$ to be a function supported in $(-1/2,1/2)$ and bounded from above and below by $C_1$, $C_2>0$.   Let us mention again that we choose the basis created by $\phi$ but we can choose the basis created by $\psi$ and the analysis would remain the same.
\end{remark}
\begin{proposition}\label{AssumptionPropo1} Let $N$ be a positive integer. Suppose that we take the Haar function $(\ref{HaarFunction})$ as the scaling function for the multiresolution analysis, $\{\Phi_{N,k}\}$ is a basis for $\mathcal{V}_N$ and $F_N$ is the filter defined by $(\ref{HaarFilter})$. Then $\mathcal{P}_N=F_NP_N$ satisfies assumption $\ref{Assumption1}.$
\end{proposition}
\begin{proof}
First, we can see directly that
$$P_N\chi_{(-1,1)^3}=\chi_{(-1,1)^3},$$
and
$$\mathcal{P}_N\chi_{(-1,1)^3}=\chi_{(-2^{-N}(2\hat{k}_N+1),2^{-N}(2\hat{k}_N+1))^3},$$
which implies
$$\kappa(\bar{v})=\eta(\bar{v})^{-1}\chi_{(-2^{-N}(2\hat{k}_N+1),2^{-N}(2\hat{k}_N+1))^3}.$$
Notice that if $$|\bar{v}|=\left|\frac{v}{1+|v|}\right|\leq 2^{-N}(2\hat{k}_N+1),$$
then 
$$|v|\leq \frac{2\hat{k}_N+1}{2^N-2\hat{k}_N-1}.$$
This leads to 
$$\varkappa(v)=\kappa(\varphi(v))=(1+|v|^2)\chi_{\left(-\frac{2\hat{k}_N+1}{2^N-2\hat{k}_N-1},\frac{2\hat{k}_N+1}{2^N-2\hat{k}_N-1}\right)^3},$$
where we recall that $|v|=|(v_1,v_2,v_3)|=\max\{|v_1|,|v_2|,|v_3|\}$. Inequality $(\ref{AssumpConcret1})$ follows directly from the above formula for $\varkappa(v)$.
\end{proof}
\subsubsection{Coercivity preserving property}
\begin{assumption}\label{Assumption2}
Let $N$ be a positive integer and $\vartheta$, $\vartheta'$ be two positive functions in $L^2((-1,1)^3)$. Define
$$\vartheta_N=\mathcal{P}_N\vartheta, \mbox{ and } \vartheta_N=\mathcal{P}_N\vartheta'.$$
Let $s$ be a  constant. We impose the following assumption on the multiresolution analysis and the filter $F_N$: There exist  constants $N_{0}$, $\mathcal{K}_1$, $\mathcal{K}_2$, $\mathcal{K}_3$, $\mathcal{K}_4$ not depending on $
\vartheta$, $\vartheta'$ such that
\begin{eqnarray}\label{AssumpConcret2}\nonumber
\forall N>N_0,&&\mathcal{K}_1(1-|\bar{v}|)^s \geq \mathcal{P}_N((1-|\bar{v}|)^s)\geq\mathcal{K}_2(1-|\bar{v}|)^s \mbox{ on } [-\zeta_N,\zeta_N]^3,\\
\mbox{ and }&&\mathcal{K}_3\vartheta_N\vartheta'_N\geq \mathcal{P}_N(\vartheta_N\vartheta')\geq\mathcal{K}_4\vartheta_N\vartheta'_N.
\end{eqnarray}
\end{assumption}
We now explain the meaning of this assumption. Suppose we take ${h}_N(t,\bar{v})(1-|\bar{v}|)^s$ as a test function for $(\ref{SpectralMethod2})$
\begin{eqnarray*}\nonumber
&&\int_{(-1,1)^3}\eta(\bar{v})^{-1}(1-|\bar{v}|)^s\partial_t{h}_N{h}_N\\
&=&\int_{(-1,1)^3}\eta(\bar{v})^{-1}(1-|\bar{v}|)^s{h}_N\mathbb{P}_N\left\{\int_{(-1,1)^3}\int_{\mathbb{S}^2}\mathcal{B}(\bar{v},\bar{v}_*,\sigma)\right.\\
&&\times\left[\mathcal{C}(\bar{v},\bar{v}_*,\sigma){h}_N\left(\varphi\left(\frac{\varphi^{-1}(\bar{v})+\varphi^{-1}(\bar{v}_*)}{2}-\sigma\frac{|\varphi^{-1}(\bar{v})-\varphi^{-1}(\bar{v}_*)|}{2}\right)\right)\right.\\\nonumber
&&\left.\left.\times {h}_N\left(\varphi\left(\frac{\varphi^{-1}(\bar{v})+\varphi^{-1}(\bar{v}_*)}{2}+\sigma\frac{|\varphi^{-1}(\bar{v})-\varphi^{-1}(\bar{v}_*)|}{2}\right)\right)-{h}_N(\bar{v}){h}_N(\bar{v}_*)\right]d\sigma d\bar{v}_*\right\}d\bar{v}\\
&=&\int_{(-1,1)^3}\mathcal{P}_N((1-|\bar{v}|)^s{h}_N)\left\{\eta(\bar{v})^{-1}\int_{(-1,1)^3}\int_{\mathbb{S}^2}\mathcal{B}(\bar{v},\bar{v}_*,\sigma)\right.\\
&&\times\left[\mathcal{C}(\bar{v},\bar{v}_*,\sigma){h}_N\left(\varphi\left(\frac{\varphi^{-1}(\bar{v})+\varphi^{-1}(\bar{v}_*)}{2}-\sigma\frac{|\varphi^{-1}(\bar{v})-\varphi^{-1}(\bar{v}_*)|}{2}\right)\right)\right.\\\nonumber
&&\left.\left.\times {h}_N\left(\varphi\left(\frac{\varphi^{-1}(\bar{v})+\varphi^{-1}(\bar{v}_*)}{2}+\sigma\frac{|\varphi^{-1}(\bar{v})-\varphi^{-1}(\bar{v}_*)|}{2}\right)\right)-{h}_N(\bar{v}){h}_N(\bar{v}_*)\right]d\sigma d\bar{v}_*\right\}d\bar{v}\\
&=&\int_{(-1,1)^3}\mathcal{P}_N((1-|\bar{v}|)^s{h}_N)\left\{\int_{(-1,1)^3}\int_{\mathbb{S}^2}\eta(\bar{v})^{-1}\mathcal{B}(\bar{v},\bar{v}_*,\sigma)\right.\\
&&\times\left[\mathcal{C}(\bar{v},\bar{v}_*,\sigma){h}_N\left(\varphi\left(\frac{\varphi^{-1}(\bar{v})+\varphi^{-1}(\bar{v}_*)}{2}-\sigma\frac{|\varphi^{-1}(\bar{v})-\varphi^{-1}(\bar{v}_*)|}{2}\right)\right)\right.\\\nonumber
&&\left.\left.\times {h}_N\left(\varphi\left(\frac{\varphi^{-1}(\bar{v})+\varphi^{-1}(\bar{v}_*)}{2}+\sigma\frac{|\varphi^{-1}(\bar{v})-\varphi^{-1}(\bar{v}_*)|}{2}\right)\right)\right]\right\}d\sigma d\bar{v}_*d\bar{v}\\
&&-\int_{(-1,1)^6\times\mathbb{S}^2}\eta(\bar{v})^{-1}\mathcal{B}(\bar{v},\bar{v}_*,\sigma){h}_N(\bar{v}){h}_N(\bar{v}_*)\mathcal{P}_N((1-|\bar{v}|)^s{h}_N)d\sigma d\bar{v}_*d\bar{v}.
\end{eqnarray*}
By assumption $\ref{Assumption2}$, the last term of the above equation could be bounded in the following way
\begin{eqnarray*}
&&\mathcal{K}_2\int_{(-1,1)^6\times\mathbb{S}^2}\eta(\bar{v})^{-1}\mathcal{B}(\bar{v},\bar{v}_*,\sigma){h}_N^2(\bar{v}){h}_N(\bar{v}_*)(1-|\bar{v}|)^sd\sigma d\bar{v}_*d\bar{v}\\
&\leq&\int_{(-1,1)^6\times\mathbb{S}^2}\eta(\bar{v})^{-1}\mathcal{B}(\bar{v},\bar{v}_*,\sigma){h}_N(\bar{v}){h}_N(\bar{v}_*)\mathcal{P}_N((1-|\bar{v}|)^s{h}_N)d\sigma d\bar{v}_*d\bar{v}\\
&\leq&\mathcal{K}_1\int_{(-1,1)^6\times\mathbb{S}^2}\eta(\bar{v})^{-1}\mathcal{B}(\bar{v},\bar{v}_*,\sigma){h}_N^2(\bar{v}){h}_N(\bar{v}_*)(1-|\bar{v}|)^sd\sigma d\bar{v}_*d\bar{v}.
\end{eqnarray*}
Define $f_N$ as in $(\ref{FNfromhN})$, we can transform the above equation into
\begin{eqnarray*}
&&\mathcal{K}_2\int_{\mathbb{R}^6}\int_{\mathbb{S}^2}B(|v-{v}_*|,\sigma)f_N^2f_{N*}(1+|v|^2)(1+|{v}|)^{4-s}d\sigma dv_*d{v}\\
&\leq&\int_{(-1,1)^6\times\mathbb{S}^2}\eta(\bar{v})^{-1}\mathcal{B}(\bar{v},\bar{v}_*,\sigma){h}_N(\bar{v}){h}_N(\bar{v}_*)\mathcal{P}_N((1-|\bar{v}|)^s{h}_N)d\sigma d\bar{v}_*d\bar{v}\\
&\leq&\mathcal{K}_1\int_{\mathbb{R}^6}\int_{\mathbb{S}^2}B(|v-{v}_*|,\sigma)f_N^2f_{N*}(1+|v|^2)(1+|{v}|)^{4-s}d\sigma dv_*d{v}.
\end{eqnarray*}
This estimate is crucial in our $\mathcal{L}^2$ estimate for $h_N$, since it preserves the following property of the Boltzmann equation
$$\int_{\mathbb{R}^6}\int_{\mathbb{S}^2}B(|v-{v}_*|,\sigma)f_N^2f_{N*}(1+|v|^2)(1+|{v}|)^{4-s}d\sigma dv_*d{v}$$
$$\geq C\int_{\mathbb{R}^3}f_N^2(1+|v|^2)(1+|{v}|)^{4+\gamma-s}d{v}.$$
We will discuss about this in more details in the convergence theory of the algorithm.
\begin{proposition}\label{AssumptionPropo2} Let $N$ be a positive integer. Suppose that we take the Haar function $(\ref{HaarFunction})$ as the scaling function for the multiresolution analysis, $\{\Phi_{N,k}\}$ is a basis for $\mathcal{V}_N$ and $F_N$ is the filter defined by $(\ref{HaarFilter})$. Then $\mathcal{P}_N=F_NP_N$ satisfies assumption $\ref{Assumption2}.$
\end{proposition}
\begin{remark}In both propositions \ref{AssumptionPropo1} and \ref{AssumptionPropo2}, we can always take $\hat{k}_N$ to be $2^{N-1}-1$, and the filter $F_N$ only removes the components containing $\phi_{-N,2^{N-1}}^{per}$. 
\end{remark}
\begin{proof} Since the supports of $\Phi_{N,k}$ are disjoint and $\Phi_{N,k}\Phi_{N,k}=\Phi_{N,k}$,then
$$\vartheta_N\vartheta'_N=\mathcal{P}_N(\vartheta_N\vartheta').$$
Equation $(\ref{AssumpConcret2})$ is now equivalent with
\begin{equation}\label{Assump2Eq2}
\mathcal{K}_1(1-|\bar{v}|)^s\geq \mathcal{P}_N((1-|\bar{v}|)^s)\geq\mathcal{K}_2(1-|\bar{v}|)^s \mbox{ on } [-2^{-N}(2\hat{k}_N+1),2^{-N}(2\hat{k}_N+1)]^3.
\end{equation}
Set
$$\mathcal{P}_N[(1-|\bar{v}|)^s]=\sum_{k=-\hat{k}_N}^{\hat{k}_N}d_k\Phi_{N,k},$$
where
$$d_k=\int_{(-1,1)^3}(1-|\bar{v}|)^s\Phi_{N,k}d\bar{v},$$
we consider the coefficient $d_k\ne 0$ of $\mathcal{P}_N[(1-|\bar{v}|)^s]$. Suppose that 
$$\Phi_{N,k}(\bar{v})={\phi}^{per}_{-N,k_1}(\bar{v_1}){\phi}^{per}_{-N,k_2}(\bar{v_2}){\phi}^{per}_{-N,k_3}(\bar{v_3}),$$
with $|k_1|\geq |k_2|\geq |k_3|$. Hence, $|\bar{v}|=\max\{|\bar{v}_1|,|\bar{v}_2|,|\bar{v}_3|\}\in[2^{-N}(2|k_1|-1),2^{-N}(2|k_1|+1)]$ if $k_1\ne2^{N-1}$ and $|\bar{v}|\in[0,2^{-N}]$ if $k_1=2^{N-1}$. Therefore $1-|\bar{v}|\in[1-2^{-N}(2|k_1|+1),1-2^{-N}(2|k_1|-1)]$ if $k_1\ne2^{N-1}$ and $1-|\bar{v}|\in[1-2^{-N},1]$ if $k_1=2^{N-1}$. 
\\ If $k_1\ne2^{N-1}$ and $(1-|\bar{v}|)\in[(1-2^{-N}(2|k_1|+1)),(1-2^{-N}(2|k_1|-1))]$. 
\begin{eqnarray}\label{Assump2Eq3}\nonumber
1&\leq& \frac{\max_{|\bar{v}|\in[2^{-N}(2|k_1|-1),2^{-N}(2|k_1|+1)]}(1-|\bar{v}|)^s}{\min_{|\bar{v}|\in[2^{-N}(2|k_1|-1),2^{-N}(2|k_1|+1)]}(1-|\bar{v}|)^s}\\
&\leq & \left(\frac{2^N-2|k_1|+1}{2^N-2|k_1|-1}\right)^{|s|}\leq  3^{|s|}.
\end{eqnarray}
 If $k_1=2^{N-1}$ and $|\bar{v}|\in[1-2^{-N},1]$. 
\begin{eqnarray}\label{Assump2Eq4}\nonumber
1&\leq& \frac{\max_{|\bar{v}|\in[1-2^{-N},1]}(1-|\bar{v}|)^s}{\min_{|\bar{v}|\in[1-2^{-N},1]}(1-|\bar{v}|)^s}\\
&\leq & \left(\frac{2^N}{2^N-1}\right)^{|s|}\leq  2^{|s|}.
\end{eqnarray}
We still denote $[1-2^{-N},1]$ by $[1-2^{-N}(2|k_1|+1),1-2^{-N}(2|k_1|-1)]$ for $k_1=2^{N-1}$. Inequalities $(\ref{Assump2Eq3})$ and $(\ref{Assump2Eq4})$ imply
\begin{eqnarray*}
&&\int_{(-1,1)^3}(1-|\bar{v}|)^s{\phi}^{per}_{-N,k_1}(\bar{v}_1){\phi}^{per}_{-N,k_2}(\bar{v}_2){\phi}^{per}_{-N,k_3}(\bar{v}_3)d\bar{v}\\
&\geq&\int_{(-1,1)^3}\frac{1}{3^{|s|}}\max_{|\bar{v}|\in[1-2^{-N}(2|k_1|+1),1-2^{-N}(2|k_1|-1)] }(1-|\bar{v}|)^s\times\\
&&\times{\phi}^{per}_{-N,k_1}(\bar{v}_1){\phi}^{per}_{-N,k_2}(\bar{v}_2){\phi}^{per}_{-N,k_3}(\bar{v}_3)d\bar{v}\\
& \geq&\frac{1}{3^{|s|}}\max_{|\bar{v}|\in[1-2^{-N}(2|k_1|+1),1-2^{-N}(2|k_1|-1)] }(1-|\bar{v}|)^s\geq\frac{1}{3^{|s|}}(1-|\bar{v}|)^s,
\end{eqnarray*}
for all $\bar{v}$ in the support of $\Phi_{N,k}$. We deduce from this inequality that
\begin{equation}\label{Assump2Eq5}
d_k\Phi_{N,k}\geq \frac{1}{3^{|s|}}(1-|\bar{v}|)^s,
\end{equation}
for all $\bar{v}$ in the support of $\Phi_{N,k}$. Similarly, we also get
\begin{eqnarray*}
&&\int_{(-1,1)^3}(1-|\bar{v}|)^s{\phi}^{per}_{-N,k_1}(\bar{v}_1){\phi}^{per}_{-N,k_2}(\bar{v}_2){\phi}^{per}_{-N,k_3}(\bar{v}_3)d\bar{v}\\
&\leq&\int_{(-1,1)^3}3^{|s|}\min_{|\bar{v}|\in[1-2^{-N}(2|k_1|+1),1-2^{-N}(2|k_1|-1)] }(1-|\bar{v}|)^s\times\\
&&\times{\phi}^{per}_{-N,k_1}(\bar{v}_1){\phi}^{per}_{-N,k_2}(\bar{v}_2){\phi}^{per}_{-N,k_3}(\bar{v}_3)d\bar{v}\\
& \leq&{3^{|s|}}\min_{|\bar{v}|\in[1-2^{-N}(2|k_1|+1),1-2^{-N}(2|k_1|-1)] }(1-|\bar{v}|)^s \leq{3^{|s|}}(1-|\bar{v}|)^s,
\end{eqnarray*}
for all $\bar{v}$ in the support of $\Phi_{N,k}$, and 
\begin{equation}\label{Assump2Eq6}
d_k\Phi_{N,k}\leq{3^{|s|}}(1-|\bar{v}|)^s,
\end{equation}
for all $\bar{v}$ in the support of $\Phi_{-N,k}$.
We deduce from inequality $(\ref{Assump2Eq5})$ and $(\ref{Assump2Eq6})$ that
$${3^{|s|}}(1-|\bar{v}|)^s\geq \mathcal{P}_N((1-|\bar{v}|)^s)\geq\frac{1}{3^{|s|}}(1-|\bar{v}|)^s.$$
\end{proof}
\section{Convergence theory of the adaptive spectral method}
Consider again equation $(\ref{SpectralMethod2})$ 
\begin{eqnarray}\label{SpectralMethod3}\nonumber
&&\partial_t {h}_N(t,\bar{v})=\mathbb{P}_N\left\{\int_{(-1,1)^3}\int_{\mathbb{S}^2}\mathcal{B}(\bar{v},\bar{v}_*,\sigma)\right.\\
&&\times\left[\mathcal{C}(\bar{v},\bar{v}_*,\sigma){h}_N\left(\varphi\left(\frac{\varphi^{-1}(\bar{v})+\varphi^{-1}(\bar{v}_*)}{2}-\sigma\frac{|\varphi^{-1}(\bar{v})-\varphi^{-1}(\bar{v}_*)|}{2}\right)\right)\right.\\\nonumber
&&\left.\left.\times {h}_N\left(\varphi\left(\frac{\varphi^{-1}(\bar{v})+\varphi^{-1}(\bar{v}_*)}{2}+\sigma\frac{|\varphi^{-1}(\bar{v})-\varphi^{-1}(\bar{v}_*)|}{2}\right)\right)-{h}_N(\bar{v}){h}_N(\bar{v}_*)\right]d\sigma d\bar{v}_*\right\},
\end{eqnarray}
we will prove  that 
\begin{itemize}
\item The solution $h_N$ of $(\ref{SpectralMethod3})$ is  positive and uniformly bounded with respect to $N$ in  $\mathcal{L}^1_2$ norm.
\item $h_N$ has a Maxellian lower bound: for all $t_0>0$, there exist  $\hat{C}_1$, $\hat{C}_2$ $>0$ independent of $N$, such that for all $\bar{v}$ in the support of $h_{N}$
$$h_{N}(t,\bar{v})\geq\hat{C}_1\exp\left(-\hat{C}_2\left|\frac{|\bar{v}|}{1-|\bar{v}|}\right|^2\right),~~~\forall t>t_0.$$
\item  $h_N$  is uniformly bounded with respect to $N$ in  $\mathcal{L}^2_{-4}$ norm.
\item The approximate solution $h_N$ converges to the solution $h$ of  $(\ref{BoltzmannHomoNewFormulation})$ as $N$ tends to infinity. 
\end{itemize}
In order to prove the positivity and boundedness of $h_N$ in $\mathcal{L}^1_2$ and $\mathcal{L}^2_{-4}$ norms, we consider the approximate Boltzmann  equation with a bounded collision kernel as in \cite{Arkeryd:1972:OBE} and \cite{DiPernaLions:1989:OCP}  
\begin{eqnarray}\label{BoltzmannBoundedKernel}\nonumber
&&\partial_t {h}_{N,\lambda}(t,\bar{v})\\\nonumber
&=&Q_{N,\lambda}({h}_{N,\lambda},{h}_{N,\lambda})=Q_{N,\lambda}^+({h}_{N,\lambda},{h}_{N,\lambda})-Q_{N,\lambda}^-({h}_{N,\lambda},{h}_{N,\lambda})\\
&:=&\mathbb{P}_N\left\{\int_{(-1,1)^3}\int_{\mathbb{S}^2}B_\lambda(|\varphi^{-1}(\bar{v})-\varphi^{-1}(\bar{v}_*)|,\sigma)\right.\\\nonumber
&&\times\left[\mathcal{C}(\bar{v},\bar{v}_*,\sigma){h}_{N,\lambda}\left(\varphi\left(\frac{\varphi^{-1}(\bar{v})+\varphi^{-1}(\bar{v}_*)}{2}-\sigma\frac{|\varphi^{-1}(\bar{v})-\varphi^{-1}(\bar{v}_*)|}{2}\right)\right)\right.\\\nonumber
&&\times {h}_{N,\lambda}\left(\varphi\left(\frac{\varphi^{-1}(\bar{v})+\varphi^{-1}(\bar{v}_*)}{2}+\sigma\frac{|\varphi^{-1}(\bar{v})-\varphi^{-1}(\bar{v}_*)|}{2}\right)\right)\\\nonumber
&&\left.\left.-{h}_{N,\lambda}(\bar{v}){h}_{N,\lambda}(\bar{v}_*)\right]d\sigma d\bar{v}_*\right\},
\end{eqnarray}
with 
$$B_\lambda(|u|,\sigma):=|(u\wedge \lambda)|^\gamma b(\cos\theta)=|\min\{u,\lambda\}|^\gamma b(\cos\theta),$$
where $\lambda$ is a positive constant. For the sake of simplicity, we denote 
$$\mathcal{B}_\lambda(\bar{v},\bar{v}_*,\sigma)= B_\lambda\left(|\varphi^{-1}(\bar{v})-\varphi^{-1}(\bar{v}_*)|,\sigma\right).$$
Since $(\ref{BoltzmannBoundedKernel})$ is a system of ODEs, it admits a unique solution which is continuous in time. In this section we always assume that $N$ and $\lambda$ are sufficiently large. We will prove that $h_{N,\lambda}$ is bounded in $\mathcal{L}^1_2$ and $\mathcal{L}^2_{-4}$ and bounded from below by a Maxwellian uniformly with respect to $N$ and $\lambda$. By Nagumo's criterion, Dunford-Pettis theorem and Smulian theorem (see \cite{Edwards:1965:FUN} and \cite{McShane:1944:Int}), $h_{N}$ is bounded in $\mathcal{L}^1_2$ and $\mathcal{L}^2_{-4}$ and bounded from below by a Maxwellian uniformly with respect to $N$. The convergence of the algorithm then follows after some technical computations.
\subsection{Positivity and $\mathcal{L}^1$ estimate of $h_{N,\lambda}$}
\begin{proposition}\label{PropoL1EstimatehNLambda}The solution $h_{N,\lambda}(t)$ of $(\ref{BoltzmannBoundedKernel})$ is positive  for all time $t$ in $\mathbb{R}_+$, moreover
$$\|h_{N,\lambda}(t)\|_{\mathcal{L}^1(\eta^{-1})}\leq \|h_{0_N}\|_{\mathcal{L}^1(\eta^{-1})},\forall t\in\mathbb{R}_+.$$
\end{proposition}
\begin{proof}
First, equation $(\ref{BoltzmannBoundedKernel})$ implies  
\begin{eqnarray*}\nonumber
&&\int_{(-1,1)^3}\partial_t |{h}_{N,\lambda}|\eta^{-1}d\bar{v}\\
&\leq&\int_{(-1,1)^3}\eta^{-1}\mathbb{P}_N\left\{\int_{(-1,1)^3}\int_{\mathbb{S}^2}\mathcal{B}_\lambda(\bar{v},\bar{v}_*,\sigma)\right.\\
&&\times\left[\mathcal{C}(\bar{v},\bar{v}_*,\sigma)\left|{h}_{N,\lambda}\left(\varphi\left(\frac{\varphi^{-1}(\bar{v})+\varphi^{-1}(\bar{v}_*)}{2}-\sigma\frac{|\varphi^{-1}(\bar{v})-\varphi^{-1}(\bar{v}_*)|}{2}\right)\right)\right|\right.\\\nonumber
&&\times\left| {h}_{N,\lambda}\left(\varphi\left(\frac{\varphi^{-1}(\bar{v})+\varphi^{-1}(\bar{v}_*)}{2}+\sigma\frac{|\varphi^{-1}(\bar{v})-\varphi^{-1}(\bar{v}_*)|}{2}\right)\right)\right|\\
&&\left.\left.+|{h}_{N,\lambda}(\bar{v})||{h}_{N,\lambda}(\bar{v}_*)|\right]d\sigma d\bar{v}_*\right\}d\bar{v}\\
&\leq&\int_{(-1,1)^3}\eta^{-1}\mathbb{P}_N\left\{\int_{(-1,1)^3}\int_{\mathbb{S}^2}\mathcal{B}_\lambda(\bar{v},\bar{v}_*,\sigma)\right.\\
&&\times\left[\mathcal{C}(\bar{v},\bar{v}_*,\sigma)\left|{h}_{N,\lambda}\left(\varphi\left(\frac{\varphi^{-1}(\bar{v})+\varphi^{-1}(\bar{v}_*)}{2}-\sigma\frac{|\varphi^{-1}(\bar{v})-\varphi^{-1}(\bar{v}_*)|}{2}\right)\right)\right|\right.\\\nonumber
&&\times\left| {h}_{N,\lambda}\left(\varphi\left(\frac{\varphi^{-1}(\bar{v})+\varphi^{-1}(\bar{v}_*)}{2}+\sigma\frac{|\varphi^{-1}(\bar{v})-\varphi^{-1}(\bar{v}_*)|}{2}\right)\right)\right|\\
&&\left.\left.-|{h}_{N,\lambda}(\bar{v})||{h}_{N,\lambda}(\bar{v}_*)|\right]d\sigma d\bar{v}_*\right\}d\bar{v}\\
&&+2\int_{(-1,1)^3}\eta^{-1}\mathbb{P}_N\left\{\int_{(-1,1)^3}\int_{\mathbb{S}^2}\mathcal{B}_\lambda(\bar{v},\bar{v}_*,\sigma)|{h}_{N,\lambda}(\bar{v})||{h}_{N,\lambda}(\bar{v}_*)|d\sigma d\bar{v}_*\right\}d\bar{v}\\
&\leq&2\int_{(-1,1)^3}\eta^{-1}\mathbb{P}_N\left\{\int_{(-1,1)^3}\int_{\mathbb{S}^2}\mathcal{B}_\lambda(\bar{v},\bar{v}_*,\sigma)|{h}_{N,\lambda}(\bar{v})||{h}_{N,\lambda}(\bar{v}_*)|d\sigma d\bar{v}_*\right\}d\bar{v},
\end{eqnarray*}
where the last inequality follows from assumption $\ref{Assumption1}$ and $(\ref{Assump1Eq1})$.
\\ We deduce from the above equation and $(\ref{L1Projection})$ that
$$\frac{d}{dt}\int_{(-1,1)^3}|h_{N,\lambda}|\eta^{-1}d\bar{v}\leq C\left[\int_{(-1,1)^3}|h_{N,\lambda}|\eta^{-1}d\bar{v}\right]^2,$$
where $C$ is some positive constant, which implies
\begin{equation}\label{L1EstimatehNlambdaEqM1}
\|h_{N,\lambda}\|_{\mathcal{L}^1(\eta^{-1})}\leq \frac{\|h_{0_N}\|_{\mathcal{L}^1(\eta^{-1})}}{1-C\|h_{0_N}\|_{\mathcal{L}^1(\eta^{-1})}t}.
\end{equation}
Set $$M=2\|h_{0_N}\|_{\mathcal{L}^1(\eta^{-1})},$$
then put
$$\tau<\frac{1}{2C\|h_{0_N}\|_{\mathcal{L}^1(\eta^{-1})}},$$
we get
$$\forall t\in[0,\tau]~~~~~~~~~~\|h_{N,\lambda}\|_{\mathcal{L}^1(\eta^{-1})}\leq M.$$
We now prove that on $[0,\tau]$, $h_{N,\lambda}$ is positive. Split $h_{N,\lambda}$ as $h_{N,\lambda}=h_{N,\lambda,+}-h_{N,\lambda,-}$ where $h_{N,\lambda,+}=\max\{h_{N,\lambda},0\}$ and $h_{N,\lambda,-}=\max\{-h_{N,\lambda},0\}$, we get
\begin{eqnarray}\label{L1EstimatehNlambdaEq0}
 Q_{N,\lambda}^+(h_{N,\lambda},h_{N,\lambda})
&=& Q^+_{N,\lambda}(h_{N,\lambda,+}-h_{N,\lambda,-},h_{N,\lambda,+}-h_{N,\lambda,-})\\\nonumber
&\geq & -Q_{N,\lambda}^+(h_{N,\lambda,+},h_{N,\lambda,-})-Q_{N,\lambda}^+(h_{N,\lambda,-},h_{N,\lambda,+}).
\end{eqnarray}
We consider the term
\begin{eqnarray*}
& &\left\|Q_{N,\lambda}^+(h_{N,\lambda,+},h_{N,\lambda,-})\eta^{-1}\right\|_{\mathcal{L}^\infty_{-4}}\\
&\leq&\left\|(1-|\bar{v}|)^4\mathcal{P}_N\left\{\int_{(-1,1)^3}\int_{\mathbb{S}^2}\mathcal{B}_\lambda(\bar{v},\bar{v}_*,\sigma)\eta^{-1}\right.\right.\\
&&\times\mathcal{C}(\bar{v},\bar{v}_*,\sigma)\left|{h}_{N,\lambda,+}\left(\varphi\left(\frac{\varphi^{-1}(\bar{v})+\varphi^{-1}(\bar{v}_*)}{2}-\sigma\frac{|\varphi^{-1}(\bar{v})-\varphi^{-1}(\bar{v}_*)|}{2}\right)\right)\right|\\\nonumber
&&\times\left.\left. \left|{h}_{N,\lambda,-}\left(\varphi\left(\frac{\varphi^{-1}(\bar{v})+\varphi^{-1}(\bar{v}_*)}{2}+\sigma\frac{|\varphi^{-1}(\bar{v})-\varphi^{-1}(\bar{v}_*)|}{2}\right)\right)\right| d\sigma d\bar{v}_*\right\}\right\|_{\mathcal{L}^\infty}\\\nonumber
&\leq&C\left\|\mathcal{P}_N[(1-|\bar{v}|)^4]\mathcal{P}_N\left\{\int_{(-1,1)^3}\int_{\mathbb{S}^2}\mathcal{B}_\lambda(\bar{v},\bar{v}_*,\sigma)\eta^{-1}\right.\right.\\
&&\times\mathcal{C}(\bar{v},\bar{v}_*,\sigma)\left|{h}_{N,\lambda,+}\left(\varphi\left(\frac{\varphi^{-1}(\bar{v})+\varphi^{-1}(\bar{v}_*)}{2}-\sigma\frac{|\varphi^{-1}(\bar{v})-\varphi^{-1}(\bar{v}_*)|}{2}\right)\right)\right|\\\nonumber
&&\times\left.\left. \left| {h}_{N,\lambda,-}\left(\varphi\left(\frac{\varphi^{-1}(\bar{v})+\varphi^{-1}(\bar{v}_*)}{2}+\sigma\frac{|\varphi^{-1}(\bar{v})-\varphi^{-1}(\bar{v}_*)|}{2}\right)\right)\right| d\sigma d\bar{v}_*
\right\}\right\|_{\mathcal{L}^\infty}\\\nonumber
&\leq&C\left\|\mathcal{P}_N\left\{\mathcal{P}_N[(1-|\bar{v}|)^4]\int_{(-1,1)^3}\int_{\mathbb{S}^2}\mathcal{B}_\lambda(\bar{v},\bar{v}_*,\sigma)\eta^{-1}\right.\right.\\
&&\times\mathcal{C}(\bar{v},\bar{v}_*,\sigma)\left|{h}_{N,\lambda,+}\left(\varphi\left(\frac{\varphi^{-1}(\bar{v})+\varphi^{-1}(\bar{v}_*)}{2}-\sigma\frac{|\varphi^{-1}(\bar{v})-\varphi^{-1}(\bar{v}_*)|}{2}\right)\right)\right|\\\nonumber
&&\times\left. \left. \left|{h}_{N,\lambda,-}\left(\varphi\left(\frac{\varphi^{-1}(\bar{v})+\varphi^{-1}(\bar{v}_*)}{2}+\sigma\frac{|\varphi^{-1}(\bar{v})-\varphi^{-1}(\bar{v}_*)|}{2}\right)\right)\right| d\sigma d\bar{v}_*\right\}\right\|_{\mathcal{L}^\infty}\\\nonumber
&\leq&C\left\|(1-|\bar{v}|)^4\int_{(-1,1)^3}\int_{\mathbb{S}^2}\mathcal{B}_\lambda(\bar{v},\bar{v}_*,\sigma)\eta^{-1}\right.\\
&&\times\mathcal{C}(\bar{v},\bar{v}_*,\sigma)\left|{h}_{N,\lambda,+}\left(\varphi\left(\frac{\varphi^{-1}(\bar{v})+\varphi^{-1}(\bar{v}_*)}{2}-\sigma\frac{|\varphi^{-1}(\bar{v})-\varphi^{-1}(\bar{v}_*)|}{2}\right)\right)\right|\\\nonumber
&&\times \left.\left|{h}_{N,\lambda,-}\left(\varphi\left(\frac{\varphi^{-1}(\bar{v})+\varphi^{-1}(\bar{v}_*)}{2}+\sigma\frac{|\varphi^{-1}(\bar{v})-\varphi^{-1}(\bar{v}_*)|}{2}\right)\right)\right|d\sigma d\bar{v}_*\right\|_{\mathcal{L}^\infty},
\end{eqnarray*}
where we use assumption $\ref{Assumption2}$ and $(\ref{LinftyProjection})$. Notice that the norms $\mathcal{L}^\infty$ and $L^\infty$ are taken on the support of the projection.
\\ Similar as $(\ref{FNfromhN})$, set
$${f}_{N,\lambda,-}(v)={h}_{N,\lambda,-}(\varphi(v))(1+|v|)^{-4},~~~{f}_{N,\lambda,+}(v)={h}_{N,\lambda,+}(\varphi(v))(1+|v|)^{-4},$$
then the above equation is now
\begin{eqnarray}\label{L1EstimatehNlambdaEq1}
& &\left\|Q_{N,\lambda}^+(h_{N,\lambda,+},h_{N,\lambda,-})\eta^{-1}\right\|_{\mathcal{L}^\infty_{-4}}\\\nonumber
&\leq&\left\|\int_{\mathbb{R}^3\times\mathbb{S}^2}B_\lambda(|v-v_*|,\sigma)f_{N,\lambda,+*}'f_{N,\lambda,-}'(1+|v|^2)d\sigma dv_*\right\|_{L^{\infty}}.
\end{eqnarray}
By Remark 3 of Theorem 2.1 \cite{MouhotVillani:2004:RTS}, we have
\begin{eqnarray*}
&&\left\|\int_{\mathbb{R}^3\times\mathbb{S}^2}B_\lambda(|v-v_*|,\sigma)f_{N,\lambda,+*}'f_{N,\lambda,-}'(1+|v|^2)d\sigma dv_*\right\|_{L^{\infty}}\\
&\leq&C \|f_{N,\lambda,+}\|_{L^1_2}\|f_{N,\lambda,-}\|_{L^\infty_2},
\end{eqnarray*}
which is equivalent with
\begin{eqnarray}\label{L1EstimatehNlambdaEq2}
&&\left\|\int_{\mathbb{R}^3\times\mathbb{S}^2}B_\lambda(|v-v_*|,\sigma)f_{N,\lambda,+*}'f_{N,\lambda,-}'(1+|v|^2)d\sigma dv_*\right\|_{L^\infty}\\\nonumber
&\leq&C \|h_{N,\lambda,+}\|_{\mathcal{L}^1(\eta)}\|h_{N,\lambda,-}\|_{\mathcal{L}^\infty_{-2}},
\end{eqnarray}
where $C$ is some positive constant.\\
Inequalities $(\ref{L1EstimatehNlambdaEq1})$ and $(\ref{L1EstimatehNlambdaEq2})$ lead to
\begin{equation}\label{L1EstimatehNlambdaEq3}
\left\|Q_{N,\lambda}^+(h_{N,\lambda,+},h_{N,\lambda,-})\eta^{-1}\right\|_{\mathcal{L}^\infty_{-4}}\leq C \|h_{N,\lambda,+}\|_{\mathcal{L}^1(\eta)}\|h_{N,\lambda,-}\|_{\mathcal{L}^\infty_{-2}}.
\end{equation}
Due to assumption $\ref{Assupmtionb2}$, we can permute $h_{N,\lambda,+}$ and $h_{N,\lambda,-}$ to get
\begin{equation}\label{L1EstimatehNlambdaEq4}
\left\|Q_{N,\lambda}^+(h_{N,\lambda,-},h_{N,\lambda,+})\eta^{-1}\right\|_{\mathcal{L}^\infty_{-4}}\leq C \|h_{N,\lambda,-}\|_{\mathcal{L}^\infty_{-2}}\|h_{N,\lambda,+}\|_{\mathcal{L}^1(\eta)}.
\end{equation}
Inequalities $(\ref{L1EstimatehNlambdaEqM1})$, $(\ref{L1EstimatehNlambdaEq0})$, $(\ref{L1EstimatehNlambdaEq3})$ and $(\ref{L1EstimatehNlambdaEq4})$ lead to
\begin{eqnarray}\label{L1EstimatehNlambdaEq5}\nonumber
& &\left\|Q_{N,\lambda}^+(h_{N,\lambda,-},h_{N,\lambda,+})\eta^{-1}\right\|_{\mathcal{L}^\infty_{-4}}+\left\|Q_{N,\lambda}^+(h_{N,\lambda,+},h_{N,\lambda,-})\eta^{-1}\right\|_{\mathcal{L}^\infty_{-4}}\\
&\leq& C(M)\|h_{N,\lambda,-}\|_{\mathcal{L}^\infty_{-2}} \mbox{                          on } [0,\tau],
\end{eqnarray}
where $C(M)$ is a constant depending on $M$.
\\ Equation $(\ref{L1EstimatehNlambdaEq0})$ implies
\begin{eqnarray*}
-\partial_th_{N,\lambda,-}\eta^{-1}&=&Q^+_{N,\lambda}(h_{N,\lambda},h_{N,\lambda})\eta^{-1}-Q^-_{N,\lambda}(h_{N,\lambda},h_{N,\lambda,-})\eta^{-1}\\
&\geq&-Q^+_{N,\lambda}(h_{N,\lambda,+},h_{N,\lambda,-})\eta^{-1}-Q^+_{N,\lambda}(h_{N,\lambda,-},h_{N,\lambda,+})\eta^{-1}\\
& &-Q^-_{N,\lambda}(h_{N,\lambda},h_{N,\lambda,-})\eta^{-1},
\end{eqnarray*}
which means 
\begin{eqnarray}\label{L1EstimatehNlambdaEq6}\nonumber
\partial_th_{N,\lambda,-}\eta^{-1}&\leq&Q^+_{N,\lambda}(h_{N,\lambda,+},h_{N,\lambda,-})\eta^{-1}+Q^+_{N,\lambda}(h_{N,\lambda,-},h_{N,\lambda,+})\eta^{-1}\\
& &+Q^-_{N,\lambda}(h_{N,\lambda},h_{N,\lambda,-})\eta^{-1}.
\end{eqnarray}
Since $$\left\|Q^-_{N,\lambda}(h_{N,\lambda},h_{N,\lambda,-})\eta^{-1}\right\|_{\mathcal{L}^\infty_{-4}}\leq C(M)\|h_{N,\lambda,-}\|_{\mathcal{L}^\infty_{-2}},$$
where $C(M)$ is some constant depending on $M$. Inequalities $(\ref{L1EstimatehNlambdaEq5})$ and $(\ref{L1EstimatehNlambdaEq6})$ lead to
$$\frac{d}{dt}\|h_{N,\lambda,-}\eta^{-1}\|_{\mathcal{L}^\infty_{-4}}\leq C'(M)\|h_{N,\lambda,-}\|_{\mathcal{L}^\infty_{-2}}\mbox{ on } [0,\tau],$$
which implies
$$\|h_{N,\lambda,-}(t)\eta^{-1}\|_{\mathcal{L}^\infty_{-4}}\leq \exp(C'(M)t)\|h_{0_N,-}\eta^{-1}\|_{\mathcal{L}^\infty_{-4}}=0, \mbox{ on }[0,\tau].$$
Hence $h_{N,\lambda,-}=0$ on $[0,\tau]$, which means $h_{N,\lambda}\geq 0$ on $[0,\tau]$. As a consequence, assumption $\ref{Assumption1}$ and $(\ref{Assump1Eq1})$ imply
$$\int_{(-1,1)^3}h_{N,\lambda}(t)\eta^{-1}d\bar{v}\leq \int_{(-1,1)^3}h_{0_N}\eta^{-1}d\bar{v}\mbox{ on } [0,\tau].$$
By repeating the argument for $[\tau,2\tau]$, $[2\tau,3\tau]$... we conclude that $h_{N,\lambda}$ is positive and $\|h_{N,\lambda}\|_{\mathcal{L}^1(\eta^{-1})}$ is bounded at all time.
\end{proof}
\subsection{Maxwellian lower bound for $h_{N,\lambda}$}
In this section, we establish a Maxwellian lower bound for the solution ${h}_{N,\lambda}$ of $(\ref{BoltzmannBoundedKernel})$ like in \cite{PulvirentiWennberg:1997:MLB} and \cite{Carleman:1933:TEI}. We first formulate some inequalities of Duhamel's type that will be the base of our estimates to obtain a Maxwellian lower bound for  ${h}_{N,\lambda}$. Notice that the results in this subsection still hold for the case $\lambda=\infty$. Consider the equation $(\ref{BoltzmannBoundedKernel})$ on ${h}_{N,\lambda}$, by assumption $\ref{Assumption2}$ we have
\begin{eqnarray}\label{BeforeGettingDuhamel}\nonumber
&&\partial_t {h}_{N,\lambda}(t,\bar{v})\\\nonumber
&=&Q_{N,\lambda}({h}_{N,\lambda},{h}_{N,\lambda})=Q_{N,\lambda}^+({h}_{N,\lambda},{h}_{N,\lambda})-Q_{N,\lambda}^-({h}_{N,\lambda},{h}_{N,\lambda})\\\nonumber
&=&Q_{N,\lambda}^+({h}_{N,\lambda},{h}_{N,\lambda})\\\nonumber
&&-\eta\mathcal{P}_N\left\{\eta^{-1}\int_{(-1,1)^3}\int_{\mathbb{S}^2}\mathcal{B}_\lambda(\bar{v},\bar{v}_*,\sigma){h}_{N,\lambda}(\bar{v}){h}_{N,\lambda}(\bar{v}_*)d\sigma d\bar{v}_*\right\}\\
&\geq&Q_{N,\lambda}^+({h}_{N,\lambda},{h}_{N,\lambda})\\\nonumber
&&-C\eta[\eta^{-1}{h}_{N,\lambda}(\bar{v})]\mathcal{P}_N\left\{\int_{(-1,1)^3}\int_{\mathbb{S}^2}\mathcal{B}_\lambda(\bar{v},\bar{v}_*,\sigma){h}_{N,\lambda}(\bar{v}_*)d\sigma d\bar{v}_*\right\}\\\nonumber
&\geq&Q_{N,\lambda}^+({h}_{N,\lambda},{h}_{N,\lambda})\\\nonumber
&&-C{h}_{N,\lambda}(\bar{v})\mathcal{P}_N\left\{\int_{(-1,1)^3}\int_{\mathbb{S}^2}\mathcal{B}_\lambda(\bar{v},\bar{v}_*,\sigma){h}_{N,\lambda}(\bar{v}_*)d\sigma d\bar{v}_*\right\}.
\end{eqnarray}
Set $$\mathcal{H}(\bar{v})=\mathcal{P}_N\left\{\int_{(-1,1)^3}\int_{\mathbb{S}^2}\mathcal{B}_\lambda(\bar{v},\bar{v}_*,\sigma){h}_{N,\lambda}(\bar{v}_*)d\sigma d\bar{v}_*\right\},$$
and
$$G_{t_1}^{t_2}(\bar{v})=\exp\left(-\int_{t_1}^{t_2}\mathcal{H}(t,\bar{v})dt\right),~~~\forall t_1,t_2>0.$$
Apply Duhamel's representation for inequality $(\ref{BeforeGettingDuhamel})$, we get
\begin{equation}\label{Duhamel}
h_{N,\lambda}(t,\bar{v})\geq h_{0_N}(\bar{v})G_0^t(\bar{v})+\int_0^t G_{\tau}^t(\bar{v})Q_{N,\lambda}^+({h}_{N,\lambda}(\tau,.),{h}_{N,\lambda}(\tau,.))(\bar{v})d\tau.
\end{equation}
In order to come back to the original formulation of the Boltzmann equation, we define 
\begin{equation}\label{hNlamdaComeBackTofNlambda}
f_{N,\lambda}(t,v)=h_{N,\lambda}(t,\varphi(v))(1+|v|)^{-4},
\end{equation}
and accordingly 
$$f_{0_N}(v)=h_{0_N}(\varphi(v))(1+|v|)^{-4}.$$
With this new function, $\mathcal{H}$ becomes
$$\mathcal{P}_N\left\{\int_{\mathbb{R}^3\times\mathbb{S}^2}B_\lambda(|v-v_*|,\sigma){f}_{N,\lambda}({v}_*)d\sigma d{v}_*\right\},$$
where for the sake of simplicity, we still denote by $\mathcal{P}_N$ the orthogonal project from $L^2((-1,1)^3)$ onto $\mathcal{V}_{N}$ but with the new variable $v$. 
\\ By proposition $\ref{PropoL1EstimatehNLambda}$, we can see that 
  $$\int_{(-1,1)^3}\int_{\mathbb{S}^2}\mathcal{B}_\lambda(\bar{v},\bar{v}_*,\sigma){h}_{N,\lambda}(\bar{v}_*)d\sigma d\bar{v}_*\leq C(1-|\bar{v}|)^{-\gamma},$$
  which means
    $$\mathcal{H}=\mathcal{P}_N\left\{\int_{(-1,1)^3}\int_{\mathbb{S}^2}\mathcal{B}_\lambda(\bar{v},\bar{v}_*,\sigma){h}_{N,\lambda}(\bar{v}_*)d\sigma d\bar{v}_*\right\}\leq C_H\mathcal{P}_N\left\{(1-|\bar{v}|)^{-\gamma}\right\},$$
    with the notice that $\mathcal{P}_N$ is a positive projection and $C_H$ is a constant not depending on $N$ and $\lambda$.
\\ Define $$\tilde{G}_{t_1}^{t_2}(\bar{v})=\exp\left(- C_H(t_2-t_1)\mathcal{P}_N[(1-|\bar{v}|)^{-\gamma}]\right),$$
we get 
$$G_{t_1}^{t_2}(\bar{v})\geq\tilde{G}_{t_1}^{t_2}(\bar{v}).$$
Using this inequality in $(\ref{Duhamel})$, we obtain
\begin{equation}\label{DuhamelTilde}
h_{N,\lambda}(t,\bar{v})\geq h_{0_N}(\bar{v})\tilde{G}_0^t(\bar{v})+\int_0^t \tilde{G}_{\tau}^t(\bar{v})Q_{N,\lambda}^+({h}_{N,\lambda}(\tau,.),{h}_{N,\lambda}(\tau,.))(\bar{v})d\tau.
\end{equation}
From $(\ref{DuhamelTilde})$, we deduce the following two inequalities, which will be used several times in the rest of this subsection
\begin{equation}\label{DuhamelTilde1}
h_{N,\lambda}(t,\bar{v})\geq \int_0^t \tilde{G}_{\tau}^t(\bar{v})Q_{N,\lambda}^+(h_{0_N}\tilde{G}_0^\tau,h_{0_N}\tilde{G}_0^\tau)(\bar{v})d\tau,
\end{equation}
and
\begin{eqnarray}\label{DuhamelTilde2}
&&h_{N,\lambda}(t,\bar{v})\\\nonumber
&\geq& \int_0^t \tilde{G}_{\tau}^t(\bar{v})Q_{N,\lambda}^+\left(\int_0^\tau \tilde{G}_{\tau_1}^\tau(\bar{v})Q_{N,\lambda}^+(h_{0_N}\tilde{G}_0^{\tau_1},h_{0_N}\tilde{G}_0^{\tau_1})(\bar{v})d\tau_1,h_{0_N}\tilde{G}_0^{\tau}\right)(\bar{v})d\tau.
\end{eqnarray}
With the notation
$$ \hat{G}_{\tau}^t(v)=\tilde{G}_{\tau}^t(\varphi(v)), $$
where $\varphi$ is defined in $(\ref{ChangeVariableMapping})$ and 
$$\mathcal{Q}_{\lambda}^+(F_2,F_1)(v)=\int_{\mathbb{R}^3}\int_{\mathbb{S}^{2}}B_\lambda(|v-v_*|,\cos\theta){F_1}'{F_2}'_*d\sigma dv_*, 
$$
for all measurable functions $F_1$ and $F_2$, we have
\begin{equation}\label{DuhamelTilde3}
f_{N,\lambda}(t,{v})\geq \int_0^t \hat{G}_{\tau}^t(v)\mathcal{P}_N\mathcal{Q}_{\lambda}^+(f_{0_N}\hat{G}_0^\tau,f_{0_N}\hat{G}_0^\tau)({v})d\tau,
\end{equation}
and
\begin{eqnarray}\label{DuhamelTilde4}
&&f_{N,\lambda}(t,{v})\\\nonumber
&\geq& \int_0^t \hat{G}_{\tau}^t({v})\mathcal{P}_N\mathcal{Q}_{\lambda}^+\left(\int_0^\tau \hat{G}_{\tau_1}^\tau({v})\mathcal{P}_N\mathcal{Q}_{\lambda}^+(f_{0_N}\hat{G}_0^{\tau_1},f_{0_N}\hat{G}_0^{\tau_1})({v})d\tau_1,f_{0_N}\hat{G}_0^\tau\right)({v})d\tau,
\end{eqnarray}
where $(\ref{DuhamelTilde4})$ follows from assumption $\ref{Assumption2}$.
\begin{lemma}\label{LemmaMaxwellian1}Let $t_0$ be any positive constant, there are constants $R$, $\alpha$, $\epsilon_0$ and $\bar{\mathcal{O}}\in\mathbb{R}^3$ independent of $N$ and $\lambda$ such that  for $N$, $\lambda$ sufficiently large and for all $t>t_0$, we have $f_{N,\lambda}(t,v)>\epsilon_0$ for all $|v-\bar{\mathcal{O}}|<\alpha$, $|\bar{\mathcal{O}}|<R$. Moreover, for all $n>0$, there exists a constant $\alpha_n$ such that for all time $t$
\begin{equation}\label{LemmaMaxwellian1Eq0}
\int_{B(0,3R)}f_{N,\lambda}(t,v)|v_*-v|^n dv>\alpha_n(1+|v_*|)^n.
\end{equation}
\end{lemma}
\begin{proof}Let  $R$ to be a positive constant and divide $K_R=(-R,R)^3\subset\left(-\frac{\zeta_N}{1-\zeta_N},\frac{\zeta_N}{1-\zeta_N}\right)^3$ into $\left(\frac{2R}{r}\right)^3$ cubes, centred at $\mathcal{O}_i$ and of length $r$. If $R$ is large enough, we have
$$\int_{K_R}f_0dv>\frac{1}{2}.$$
Since $$\lim_{N\to\infty}\int_{K_R}|f_{0_N}-f_0|dv=0,$$
there exists $N_0$ such that for $N>N_0$
$$\int_{K_R}f_{0_N}dv>\frac{1}{2}.$$
Since $\|f_{0}\|_{L^1_2}$ is bounded, we can infer that for $r$ sufficiently small
$$\int_{K_i}f_{0}dv<\frac{1}{4.3^3},$$
for all $i$. Therefore, no set of $27$ subcubes can contain more than half of the mass contained in $K_R$, which means there exist two subcubes $K_1$, $K_2$ with $|\mathcal{O}_1-\mathcal{O}_2|\geq 2\sqrt3 r$ satisfying
  $$\int_{K_i}{f}_0\geq\frac{1}{4(2R/r)^3}, ~~~i=1,2.$$
Since $$\lim_{N\to\infty}\int_{K_i}|f_{0_N}-f_0|dv=0, \forall i,$$
there exists a constant, still denoted by $N_0$ such that for all $N>N_0$
$$\int_{K_i}f_{0_N}dv<\frac{1}{4.3^3},\mbox{           and                } \int_{K_i}{f}_{0_N}\geq\frac{1}{8(2R/r)^3}, ~~~i=1,2.$$    
We define $\bar{\mathcal{O}}=(\mathcal{O}_1+\mathcal{O}_2)/2$ and $\alpha=(|\mathcal{O}_1-\mathcal{O}_2|-\sqrt{6}r)/(4\sqrt2)$ then $\alpha>(2\sqrt3-\sqrt6)r/(4\sqrt2)$. Let $w_1$, $w_2$ be in $K_1$ and $K_2$ and define $S_{w_1,w_2}$ to be the sphere taking the segment $\overline{w_1,w_2}$ as its diagonal. We can see that the ball with center $\bar{\mathcal{O}}$ and radius $2\alpha$ lies entirely inside $S_{w_1,w_2}$. Define $\chi_1$, $\chi_2$ and $\chi_R$ to be the characteristic functions of $K_1$, $K_2$ and $K_R$. Set 
$$F_1=f_{0_N}\chi_1;~~~ F_2=f_{0_N}\chi_2;~~~F_3=f_{0_N}\chi_R,$$
and use these functions in $(\ref{DuhamelTilde4})$, we get
\begin{eqnarray}\label{LemmaMaxwellian1Eq1}
&&f_{N,\lambda}(t,{v})\\\nonumber
&\geq& \int_0^t \hat{G}_{\tau}^t({v})\mathcal{P}_N\mathcal{Q}_{\lambda}^+\left(\int_0^\tau \hat{G}_{\tau_1}^\tau({v})\mathcal{P}_N\mathcal{Q}_{\lambda}^+(f_{0_N}\hat{G}_0^{\tau_1},f_{0_N}\hat{G}_0^{\tau_1})({v})d\tau_1,f_{0_N}\hat{G}_0^\tau\right)({v})d\tau\\\nonumber
&\geq& \int_0^t \hat{G}_{\tau}^t({v})\mathcal{P}_N\mathcal{Q}_{\lambda}^+\left(\int_0^\tau \hat{G}_{\tau_1}^\tau({v})\mathcal{P}_N\mathcal{Q}_{\lambda}^+(F_2\hat{G}_0^{\tau_1},F_1\hat{G}_0^{\tau_1})({v})d\tau_1,F_3\hat{G}_0^\tau\right)({v})d\tau.
\end{eqnarray}
Since $F_1$, $F_2$, $F_3$ are all supported in $K_R$, then $\hat{G}_{t_1}^{t_2}(v)$ could be considered as being supported in $\{v~~:~~~|v|<2R\}$ and
$$\hat{G}_{t_1}^{t_2}(v)=\exp\left(- C_H(t_2-t_1)\mathcal{P}_N((1+|{v}|)^{\gamma})\right)$$
$$\geq\exp\left(- C_H(t_2-t_1)\mathcal{P}_N((1+2R)^{\gamma})\right),$$
which, together with $(\ref{LemmaMaxwellian1Eq1})$ implies
\begin{eqnarray}\label{LemmaMaxwellian1Eq2}
f_{N,\lambda}(t,{v})&\geq& \int_0^t\int_0^\tau\exp(-C\mathcal{P}_N((1+2 R)^{\gamma})(t+\tau+\tau_1))d\tau_1d\tau \\\nonumber
& &\times\mathcal{P}_N\mathcal{Q}_{\lambda}^+\left(\mathcal{P}_N\mathcal{Q}_{\lambda}^+(F_2,F_1),F_3\right)({v}),
\end{eqnarray}
where $C$ is some constant not depending on $N$ and $\lambda$. 
\\ Similar as in the normal case \cite{PulvirentiWennberg:1997:MLB}, we assume without loss of generality that $b(\cos\theta)$ is bounded from below by a constant $b_0$. By Carleman's representation, 
\begin{eqnarray}\label{LemmaMaxwellian1Eq2b}
& &\mathcal{Q}_{\lambda}^+(\mathcal{P}_N\mathcal{Q}_{\lambda}^+(F_2,F_1),F_3)({v})\\\nonumber
&=&\int_{\mathbb{R}^3}F_3(v')\frac{(|v-v'|\wedge\lambda)^\gamma}{|v-v'|^2}\int_{E_{v,v'}} \mathcal{P}_N\mathcal{Q}_{\lambda}^+(F_2,F_1)(v'_*){b(\cos\theta)}dE(v'_*)dv',
\end{eqnarray}
where $E_{v,v'}$ is the plane containing $v$ and perpendicular to $v'-v$ and $dE(v'_*)$ is the Lebesgue measure on $E_{v,v'}$. Since 
$$F_1(v)=\int_{\mathbb{R}^3}F_1(w)\delta(v-w)dw \mbox{ and } F_2(v)=\int_{\mathbb{R}^3}F_2(w)\delta(v-w)dw,$$
denote $v'$ and $v'_*$ by $u$ and $w$ we have
\begin{eqnarray}\label{LemmaMaxwellian1Eq3}\nonumber
& &\int_{E_{v,v'}} \mathcal{P}_N\mathcal{Q}_{\lambda}^+(F_2,F_1)(v'_*){b(\cos\theta)}dE(v'_*)\\\nonumber
&\geq&\int_{E_{v,u}} \mathcal{P}_N[\mathcal{Q}_{\lambda}^+(F_2,F_1)(w){b_0}]dE(w)\\\nonumber
&\geq &\int_{E_{v,u}}\mathcal{P}_N [\int_{\mathbb{R}^3\times\mathbb{S}^2}(|w-w_*|\wedge\lambda)^\gamma{b_0}^2F_1'{F_2}'_*d\sigma dw_*]dE(w)\\
&\geq &\int_{E_{v,u}}\mathcal{P}_N [\int_{\mathbb{R}^6}F_1(w_1)F_2(w_2)\int_{\mathbb{R}^3\times\mathbb{S}^2}(|w-w_*|\wedge\lambda)^\gamma{b_0}^2\\\nonumber
& &\times\delta(w'-w_1)\delta(w'_*-w_2)d\sigma dw_*dw_1dw_2]dE(w)\\\nonumber
&\geq &\int_{E_{v,u}}\mathcal{P}_N [\int_{\mathbb{R}^6}F_1(w_1)F_2(w_2)\int_{\mathbb{R}^3\times\mathbb{S}^2}(|w-w_*|\wedge\lambda)^\gamma{b_0}^2\\\nonumber
& &\times\delta_1(w')\delta_2(w'_*)d\sigma dw_*dw_1dw_2]dE(w)
\end{eqnarray}
where $$\delta_1(v')=\delta(v'-w_1), \mbox{ and } \delta_2(v'_*)=\delta_2(v'_*-w_2).$$
Let $\chi_\epsilon$ be the characteristic function of $\{w|dist(w,E_{v,u})<\epsilon\}$, then
\begin{eqnarray}\label{LemmaMaxwellian1Eq4}\nonumber
& &\int_{E_{v,u}}\mathcal{P}_N [\int_{\mathbb{R}^6}F_1(w_1)F_2(w_2)\int_{\mathbb{R}^3\times\mathbb{S}^2}(|w-w_*|\wedge\lambda)^\gamma{b_0}^2\\\nonumber
& &\times\delta_1(w')\delta_2(w'_*)d\sigma dw_*dw_1dw_2]dE(w)\\\nonumber
&=&\lim_{\epsilon\to0} \frac{1}{2\epsilon}\int_{\mathbb{R}^3}\mathcal{P}_N [\int_{\mathbb{R}^6}F_1(w_1)F_2(w_2)\int_{\mathbb{R}^3\times\mathbb{S}^2}(|w-w_*|\wedge\lambda)^\gamma{b_0}^2\\\nonumber
& &\times\delta_1(w')\delta_2(w'_*)d\sigma dw_*dw_1dw_2]\chi_\epsilon dw\\
&=&\lim_{\epsilon\to0} \frac{1}{2\epsilon}\int_{\mathbb{R}^3} [\int_{\mathbb{R}^6}F_1(w_1)F_2(w_2)\int_{\mathbb{R}^3\times\mathbb{S}^2}(|w-w_*|\wedge\lambda)^\gamma{b_0}^2\\\nonumber
& &\times\delta_1(w')\delta_2(w'_*)d\sigma dw_*dw_1dw_2]\mathcal{P}_N(\chi_\epsilon(1+|w|)^4)(1+|w|)^{-4} dw\\\nonumber
&\geq&\lim_{\epsilon\to0} C\frac{1}{2\epsilon}\int_{\mathbb{R}^3} [\int_{\mathbb{R}^6}F_1(w_1)F_2(w_2)\int_{\mathbb{R}^3\times\mathbb{S}^2}(|w-w_*|\wedge\lambda)^\gamma{b_0}^2\\\nonumber
& &\times\delta_1(w')\delta_2(w'_*)d\sigma dw_*dw_1dw_2]\mathcal{P}_N\chi_\epsilon dw,
\end{eqnarray}
where the last inequality follows from assumption $\ref{Assumption2}.$
Moreover, we have that
\begin{eqnarray}\label{LemmaMaxwellian1Eq5}\nonumber
&&\lim_{\epsilon\to0} C\frac{1}{2\epsilon}\int_{\mathbb{R}^3} \int_{\mathbb{R}^3\times\mathbb{S}^2}(|w-w_*|\wedge\lambda)^\gamma{b_0}^2\delta_1(w')\delta_2(w_*')\mathcal{P}_N\chi_\epsilon(w) d\sigma dw_* dw\\\nonumber
&=&\lim_{\epsilon\to0} C\frac{1}{2\epsilon}\int_{\mathbb{R}^3} \int_{\mathbb{R}^3\times\mathbb{S}^2}(|w-w_*|\wedge\lambda)^\gamma{b_0}^2\delta_1(w)\delta_2(w_*)\mathcal{P}_N\chi_\epsilon(w') d\sigma dw_* dw\\
&=& \lim_{\epsilon\to0}C\frac{1}{2\epsilon}\frac{(|w_1-w_2|\wedge\lambda)^\gamma}{|w_1-w_2|^2} \int_{S_{w_1,w_2}}\frac{b_0^2}{\cos\theta}\mathcal{P}_N\chi_\epsilon(w')d\tilde{n},
\end{eqnarray}
where the first equality follows from the change of variables $dw_*dw\to dw'_*dw'$ and the second one is Carleman's change of variables (see \cite{Carleman:1957:PMC}, \cite{PulvirentiWennberg:1997:MLB}), $\tilde{n}$ denotes the measure on the surface of the sphere. Since for $\lambda$ large enough 
\begin{eqnarray*}\nonumber
&& \lim_{\epsilon\to0}C\frac{1}{2\epsilon}\frac{(|w_1-w_2|\wedge\lambda)^\gamma}{|w_1-w_2|^2} \int_{S_{w_1,w_2}}\frac{b_0^2}{\cos\theta}\chi_\epsilon(w')d\tilde{n}\\\nonumber
&=& \lim_{\epsilon\to0}C\frac{1}{2\epsilon}|w_1-w_2|^{\gamma-2}\int_{S_{w_1,w_2}}\frac{b_0^2}{\cos\theta}\chi_\epsilon(w')d\tilde{n}\\
&\geq& C\pi|w_1-w_2|^{\gamma-1}b_0^2\\\nonumber.
&\geq&C\pi\min\{(2R)^{\gamma-1},(2r)^{1-\gamma}\}b_0^2,
\end{eqnarray*}
then we can have for  $\lambda$, $N$ sufficiently large
\begin{eqnarray}\label{LemmaMaxwellian1Eq6}\nonumber
&&\lim_{\epsilon\to0} C\frac{1}{2\epsilon}\frac{(|w_1-w_2|\wedge\lambda)^\gamma}{|w_1-w_2|^2} \int_{S_{w_1,w_2}}\frac{b_0^2}{\cos\theta}\mathcal{P}_N\chi_\epsilon(w')d\tilde{n}\\
&\geq&\frac{Cb_0^2\pi\min\{(2R)^{\gamma-1},(2r)^{1-\gamma}\}}{2}=:\bar{C}.
\end{eqnarray}
Combine $(\ref{LemmaMaxwellian1Eq2b})$, $(\ref{LemmaMaxwellian1Eq3})$, $(\ref{LemmaMaxwellian1Eq4})$, $(\ref{LemmaMaxwellian1Eq5})$ and $(\ref{LemmaMaxwellian1Eq6})$, we get
\begin{eqnarray}\label{LemmaMaxwellian1Eq7}
f_{N,\lambda}(t,{v})&\geq& \bar{C}\int_0^t\int_0^\tau\exp(-C\mathcal{P}_N((1+2 R)^{\gamma})(t+\tau+\tau_1))d\tau_1 d\tau \\\nonumber
& &\times\int_{\mathbb{R}^3}F_3(u)\frac{(|v-u|\wedge\lambda)^\gamma}{|v-u|^2}\int_{\mathbb{R}^6}F_1(w_1)F_2(w_2)dw_1dw_2du\\\nonumber
&\geq&\int_0^t\int_0^\tau\exp(-C\mathcal{P}_N((1+2 R)^{\gamma})(t+\tau+\tau_1))d\tau_1d\tau \\\nonumber
& &\times\bar{C}(2R)^{\gamma-2}\frac{1}{2}\left(\frac{1}{8(2R/r)^3}\right)^2,
\end{eqnarray}
for $\lambda$ sufficiently large and $v\in supp(f_{N,\lambda})$. Inequality $(\ref{LemmaMaxwellian1Eq7})$ implies that for any positive constant $t_0$, there are constants $R$, $\alpha$, $\epsilon_0$ and $\bar{\mathcal{O}}\in\mathbb{R}^3$ independent of $N$ and $\lambda$ such that  for $N$ and $\lambda$ sufficiently large $f_{N,\lambda}(t,v)>\epsilon_0$ for all $|v-\bar{\mathcal{O}}|<\alpha$, $|\bar{\mathcal{O}}|<R$, $t>t_0$ and $v\in supp(f_{N,\lambda})$. Therefore 
\begin{equation}\label{LemmaMaxwellian1Eq8}
\int_{K_R}f_{N,\lambda}(t,{v})|v-v_*|^n dv\geq \epsilon_0\int_{B(\bar{\mathcal{O}},\alpha)}|v-v_*|^n dv\geq \alpha_n(1+|v_*|)^n, ~~~\forall t>t_0.
\end{equation}
Due to $(\ref{InitialGreatThanEpsilon})$, there exist a positive constant $\epsilon^*$ and an open bounded set $\mathcal{D}$ with non-zero measure such that $f_0>\epsilon^*$ in $\mathcal{D}$. With the assumption that $R$, $N$ are large, we can suppose that $\mathcal{D}\subset K_R$, then $f_{0_N}>\epsilon^*$ in $\mathcal{D}$. Fix a small time $t_0$, inequality $(\ref{DuhamelTilde})$ implies that for all time $t\leq t_0$
\begin{eqnarray}\label{LemmaMaxwellian1Eq9}
\int_{K_R}f_{N,\lambda}(t,{v})|v-v_*|^ndv &\geq &\int_{K_R}f_{0_N}\hat{G}_0^{t}|v-v_*|^ndv\\\nonumber
&\geq &\exp(-(1+2 R)^\gamma t_0)\int_{K_R}f_{0_N}|v-v_*|^ndv\\\nonumber
&\geq &\frac{\exp(-(1+2 R)^\gamma t_0)}{2}\int_{\mathcal{D}}\epsilon^*|v-v_*|^ndv\\\nonumber
&\geq& \alpha_n(1+|v_*|)^n
\end{eqnarray}
the third inequality is satisfied when $N$ is sufficiently large. Inequalities  $(\ref{LemmaMaxwellian1Eq8})$ and $(\ref{LemmaMaxwellian1Eq9})$ imply the conclusion of the lemma.
\end{proof}
\begin{lemma}\label{LemmaMaxwellian2}Suppose that there is $\bar{\mathcal{O}}\in\mathbb{R}^3$ with $|\bar{\mathcal{O}}|<R$, such that $F(v)>\epsilon$ for $|v-\bar{\mathcal{O}}|<\alpha$, then there exist $C,\varsigma,\epsilon>0$, which do not depend on $\lambda$,  such that
$$\mathcal{Q}_\lambda^+(F,F)(v)>C\alpha^{3+\gamma}\varsigma^{5/2}\epsilon^2,$$
for all $v$, $|v-\bar{\mathcal{O}}|<\alpha\sqrt2(1-\varsigma)$.  
\end{lemma}
\begin{proof} Without loss of generality, we can assume that $\bar{\mathcal{O}}$ is the origin.  According to Carleman's representation, we have the following scaling property
\begin{eqnarray*}
& &\mathcal{Q}_\lambda^+(F,F)(\beta v)\\
&=&\int_{\mathbb{R}^3}F(u)\frac{(|u-\beta v|\wedge\lambda)^\gamma}{|u-\beta v|^2}\int_{E_{\beta v,u}}b(\cos\theta)F(w)dE(w)du\\
&=&\beta^3\int_{\mathbb{R}^3}F(\beta u)\frac{\beta^{\gamma-2}(|u- v|\wedge(\lambda/\beta))^\gamma}{|u- v|^2}\int_{E_{\beta v,\beta u}}b(\cos\theta)F( w)dE(w)du\\
&=&\beta^3\int_{\mathbb{R}^3}F(\beta u)\frac{\beta^{\gamma-2}(|u- v|\wedge(\lambda/\beta))^\gamma}{|u- v|^2}\beta^2\int_{E_{ v, u}}b(\cos\theta)F(\beta w)dE(w)du\\
&=&\beta^{\gamma+3}\int_{\mathbb{R}^3}F(\beta u)\frac{(|u- v|\wedge(\lambda/\beta))^\gamma}{|u- v|^2}\int_{E_{ v, u}}b(\cos\theta)F(\beta w)dE(w)du,
\end{eqnarray*}
where we still use the notation of the previous lemma $u=v'$ and $w=v'_*$. This scaling property means that we can suppose $R$ to be $1$ and since we only consider $\lambda$ sufficiently large, we can still keep $\lambda$ instead of changing it into $\lambda/\beta$. Suppose without loss of generality that $F$ is the characteristic function of the ball $\{v~~~|~~~|v|<1\}$ and assume by symmetry that $v=(0,0,z)$, $1\leq z<\sqrt2$. Use polar coordinates for $u$ with $v$ to be the origin, then $du=r^2\sin\varpi d\varpi dr$ and
\begin{eqnarray*}
\mathcal{Q}_\lambda^+(F,F)(\beta v)&=&\int_{\mathbb{R}^3}F(u)\frac{(|u- v|\wedge\lambda)^\gamma}{|u- v|^2}\int_{E_{ v,u}}b(\cos\theta)F(w)dE(w)du\\
&\geq&b_02\pi\int_0^\pi\int_0^\infty F(u)\frac{(r\wedge \lambda)^\gamma}{r^2}\int_{E_{v,u}}F(w)dE(w)\sin\varpi d\varpi r^2dr.\\
\end{eqnarray*}
Since $F$ is the characteristic function of $\{|v|<1\}$, we can suppose that $|u|\leq 1$ and $|w|\leq 1$, which means  $\arccos(1/z)\leq \varpi \leq \arcsin(1/z)$ and $z\cos\varpi-\sqrt{1-z^2\sin^2\varpi}\leq r\leq z\cos\varpi+\sqrt{1-z^2\sin^2\varpi}$. Applying the change of variables $y=z\cos\varpi$ gives 
\begin{eqnarray}\label{LemmaMaxwellian2Eq1}
\mathcal{Q}_\lambda^+(F,F)( v)&\geq&\frac{2\pi^2b_0}{z}\int_{\sqrt{z^2-1}}^1\int_{y-\sqrt{1-z^2+y^2}}^{y+\sqrt{1-z^2+y^2}}|r\wedge \lambda|^\gamma dr (1-y^2)dy\\\nonumber
&\geq&\frac{2\pi^2b_0}{z}\int_{\sqrt{z^2-1}}^1\int_{y-\sqrt{1-z^2+y^2}}^{y+\sqrt{1-z^2+y^2}}|r|^\gamma dr (1-y^2)dy,
\end{eqnarray}
the last inequality follows when we take $\lambda>3>y+\sqrt{1-z^2+y^2}$. We now need to estimate the integral near $z=\sqrt2$. Put $y'=1-y$ and $z=\sqrt2(1-\varsigma)$, then the integral of $r$ becomes $2\sqrt{4\varsigma-2y'}+O(\varsigma^{3/2})$ and the right hand side of $(\ref{LemmaMaxwellian2Eq1})$ could be bounded from below by
\begin{eqnarray*}
& &8\pi^2b_0(1+O(\varsigma))\int_0^{2\varsigma+O(\varsigma^2)}\left(\sqrt{2\varsigma-y'}+O(\varsigma^{3/2})\right)y'(1+O(\varsigma))dy'\\
&=&8\pi^2b_0(1+O(\varsigma))(2\varsigma)^{5/2}\int_0^{1+O(\varsigma)}\left(\sqrt{1-y'}+O(\varsigma)\right)(y'+O(\varsigma))dy'\\
&=&\frac{64\sqrt2\pi^2}{15}\varsigma^{5/2}+O(\varsigma^{7/2}).
\end{eqnarray*}
\end{proof}
\begin{proposition}\label{PropoMaxwellianLowerBound} There exist positive constants $\tilde{C}_1$, $\tilde{C}_2$ independent of $N$ and $\lambda$, such that for all $\bar{v}$ in the support of $h_{N,\lambda}$
$$h_{N,\lambda}(\bar{v},t)\geq \tilde{C}_1\exp\left(-\tilde{C}_2\left|\frac{|\bar{v}|}{1-|\bar{v}|}\right|^2\right).$$
The constants $\tilde{C}_1$ and $\tilde{C}_2$ depend on $t$, however, they could be chosen uniformly for all $t>t_0$, where $t_0$ is an arbitrary positive time.
\end{proposition}
\begin{proof} We now proceed the proof by a classical iteration process as in the normal case \cite{PulvirentiWennberg:1997:MLB}. By lemma $\ref{LemmaMaxwellian1}$, there exists a ball $|v-\bar{\mathcal{O}}|<\alpha$ such that $f_{N,\lambda}(t_0,v)>\epsilon_0$. By $(\ref{DuhamelTilde3})$
\begin{eqnarray}\label{PropoMaxwellianLowerBoundEq1}
f_{N,\lambda}(t_0+t_1,{v})\geq \int_{t_0}^{t_0+t_1} \hat{G}_{\tau}^{t_0+t_1}(v)\mathcal{P}_N\mathcal{Q}_{\lambda}^+(f_{0_N}\hat{G}_{t_0}^\tau,f_{0_N}\hat{G}_{t_0}^\tau)({v})d\tau.
\end{eqnarray}
Now, for $v$ near the given ball and lies in the support of $f_{N,\lambda}$,
$$\hat{G}_{\tau_1}^{\tau_2}\geq\exp(-(\tau_2-\tau_1)c(1+2|\bar{\mathcal{O}}|^\gamma+2^{1+\gamma/2}\alpha^\gamma)).$$
Plug this inequality into $(\ref{PropoMaxwellianLowerBoundEq1})$ and use lemma $\ref{LemmaMaxwellian2}$
\begin{eqnarray*}
f_{N,\lambda}(t_0+t_1,{v})&\geq &t_1\exp(-t_1C(1+2|\bar{\mathcal{O}}|^\beta+2^{1+\gamma/2}\alpha^\gamma))\alpha^{3+\gamma}\varsigma_1^{5/2}\epsilon_0^2\\
&\geq &t_1\exp(-Ct_12^{1+\gamma/2}\alpha^\gamma)\alpha^{3+\gamma}\varsigma_1^{5/2}\epsilon_0^2,
\end{eqnarray*}
and this holds with $|v-\bar{\mathcal{O}}|<\sqrt2(1-\varsigma_1)\alpha $ and $v\in (-\frac{\zeta_N}{1-\zeta_N},\frac{\zeta_N}{1-\zeta_N})^3$. Now, we take the iteration
\begin{eqnarray*}
f_{N,\lambda}(t_0+t_1+t_2,{v})&\geq &(t_1\exp(-t_1C2^{1+\gamma/2}\alpha^\gamma)\alpha^{3+\gamma}\varsigma_1^{5/2}\epsilon_0^2)^2\\
& &\times t_2\exp(-t_2C2^{1+2\gamma/2}\alpha^\gamma)(2^{1/2}(1-\varsigma_1)\alpha)^{3+\gamma}\varsigma_2^{5/2},
\end{eqnarray*}
for $|v-\bar{\mathcal{O}}|<2(1-\varsigma_1)(1-\varsigma_2)\alpha$ and $v\in (-\frac{\zeta_N}{1-\zeta_N},\frac{\zeta_N}{1-\zeta_N})^3$. At the $n-th$ step
\begin{eqnarray*}
& &f_{N,\lambda}(t_0+t_1+\dots+t_n,{v})\\
&> &\epsilon_0^{2^n}(C\alpha^{3+\gamma})^{2^n-1}(2^{1/2}(1-\varsigma_1))^{(3+\gamma)2^{n-1-1}}\\
& &\dots(2^{k/2}(1-\varsigma_1)\dots(1-\varsigma_k))^{(3+\gamma)2^{n-1-k}}\\
& &\dots(t_1\varsigma_1^{5/2})^{2^{n-1}}\dots(t_k\varsigma_k^{5/2})^{2^{n-k}}\dots(t_n\varsigma_n^{5/2})^{2^{0}}\\
& &\exp(-Ct_1\alpha^{\gamma}2^{1+\gamma/2}2^{n-1})\dots\exp(-Ct_k\alpha^{\gamma}2^{1+\gamma k/2}2^{n-k})\dots\exp(-Ct_n\alpha^{\gamma}2^{1+\gamma n/2}2^{n-n}),
\end{eqnarray*}for $|v-\bar{\mathcal{O}}|<2^{n/2}(1-\varsigma_1)\dots(1-\varsigma_n)\alpha$ and $v\in (-\frac{\zeta_N}{1-\zeta_N},\frac{\zeta_N}{1-\zeta_N})^3$, which leads to the conclusion of the proposition. 
\end{proof}
\subsection{$\mathcal{L}^2_{-4}$ estimate for $h_{N,\lambda}$}
Define 
\begin{equation}\label{LowerBoundPolynomialEq1}
\Upsilon_{\lambda}(v)=[1+(|v|\wedge \lambda)]^\gamma,
\end{equation}
we now prove a technical lemma on
$$\mathcal{Q}_{\lambda}^+(F,F)(v)=\int_{\mathbb{R}^3}\int_{\mathbb{S}^{2}}B_\lambda(|v-v_*|,\cos\theta)F'_*F'd\sigma dv_*,
$$
 before going to  the $\mathcal{L}^2_{-4}$ estimate for $h_{N,\lambda}$.
\begin{lemma}\label{LemmaQBound}Let $\nu$, $\delta$ and $k$ be three constants satisfying $\nu,\delta\geq -\gamma$ and $k>\gamma$. There exist positive constants $C$ and $\iota$, such that the following estimate holds for all $\epsilon>0$ and all measurable function $F$
\begin{eqnarray}\label{QBoundinLemma1}
&&\left|\int_{\mathbb{R}^3}\mathcal{Q}_{\lambda}^+(F,F)Fdv\right|\leq C\epsilon^{-\iota}\|F\|_{L^
{10/7}_\delta}\|F\|_{L^2_{-\delta}}\|F\|_{L^1_{2|\delta|}}\\\nonumber
&&+\epsilon\|F\|_{L^2(\Upsilon_\lambda^{1+\nu/\gamma})}\|F\|_{L^2(\Upsilon_\lambda^{-\nu/\gamma})}\|F\|_{L^1_{|k+\nu|+|\nu|}}.
 \end{eqnarray}
 In particular, if we take $\delta=0$ and $\nu=-\gamma/2$
 \begin{eqnarray}\label{QBoundinLemma1}
&&\left|\int_{\mathbb{R}^3}\mathcal{Q}_{\lambda}^+(F,F)Fdv\right|\\\nonumber
&\leq &C\epsilon^{-\iota}\|F\|_{L^2}\|F\|_{L^{10/7}}\|F\|_{L^1}+\epsilon\|F\|_{L^2(\Upsilon_\lambda^{1/2})}^2\|F\|_{L^1_{|k|}}.
 \end{eqnarray}
\end{lemma}
\begin{remark} Notice that the lemma is still valid for the case $\lambda=\infty$.
\end{remark}
\begin{proof}
By similar arguments as in \cite{MouhotVillani:2004:RTS}, we can suppose that $b\in C_c^\infty(-1,1)$. Let $\Theta:\mathbb{R}^3\to\mathbb{R}$ be a radial $C^\infty$ function such that $supp\Theta\subset B(0,1)$ and $\int_{\mathbb{R}^3}\Theta=1$. Let $\mu$ be a constant smaller than $\lambda$ and define the regularizing function
$$\Theta_\mu(x)=\mu^3\Theta(\mu x)~~~~(x\in\mathbb{R}^3).$$
Define 
$$\Phi_{S}=\Phi *(\Theta 1_{\mathbb{A}_\mu}), ~~~\Phi_R=\Phi-\Phi_S,$$
where $\mathbb{A}_\mu$ is the annulus $\mathbb{A}_\mu=\{x\in\mathbb{R}^3; \frac{2}{\mu}\leq |x|\leq \mu\}$.
\\ Set
$$B_\lambda(|v|,\sigma)=B^S(|v|,\sigma)+B^R(|v|,\sigma),$$
where
$$B^S(|v|,\sigma):=\Phi_{S}(v)b(\cos\theta).$$
Set $$\mathcal{Q}_{\lambda}^+=Q^{+}_S+Q^{+}_R,$$
with
$$Q^{+}_S(F,F)(v)=\int_{\mathbb{R}^3}\int_{\mathbb{S}^{2}}B^S(|v-v_*|,\cos\theta)F'_*F'd\sigma dv_*.
$$
By Corollary 3.2 \cite{MouhotVillani:2004:RTS}, the first term $Q^{+}_S$ could be bounded in the following way
\begin{eqnarray}\label{QBoundEq1}
\left|\int_{\mathbb{R}^3}Q^+_S(F,F)Fdv\right|
\leq C(\delta,b)\|F\|_{L^{10/7}_\delta}\|F\|_{L^2_{-\delta}}\|F\|_{L^1_{2|\delta|}}.
\end{eqnarray}
Now, we will estimate the second term $Q^{+}_R$. For all test function $\varrho$ the following equality holds
$$\int_{\mathbb{R}^3}Q^+_R(F,F)\varrho dv=\int_{\mathbb{R}^6}F(v_*)F(v)\left[\int_{\mathbb{S}^2}B^R(|v-v_*|,\sigma)\varrho(v')d\sigma\right]dv_*dv.$$
By defining
$$\mathcal{S}\varrho(v)=\int_{\mathbb{S}^2}B^R(|v|,\sigma)\varrho\left(\frac{v+|v|\sigma}{2}\right)d\sigma,$$
we have
$$\int_{\mathbb{R}^3}Q^+_R(F,F)\varrho dv=\int_{\mathbb{R}^3}F(v_*)\left(\int_{\mathbb{R}^3}F(v)(T_{v_*}\mathcal{S}(T_{-v_*}\varrho))(v)dv\right)dv_*.$$
Let $\xi_1$, $\xi_2$ be two non-negative constants. Consider the weighted $L^\infty$ norm of $\mathcal{S}\varrho$
$$\|\mathcal{S}\varrho\|_{L^{\infty}(\Upsilon_\lambda^{-\xi_1-\xi_2-1})}\leq C\|b\|_{L^1(\mathbb{S}^2)}\|\varrho\|_{L^\infty({\Upsilon_\lambda^{-\xi_2}})}\|\Phi_R\|_{L^\infty({\Upsilon_\lambda^{-\xi_1-1}})},$$ 
where $C$ is some positive constant.
\\ Now, consider the  weighted $L^1$ norm of $\mathcal{S}_2\varrho$
\begin{eqnarray*}
\|\mathcal{S}\varrho\|_{L^1(\Upsilon_\lambda^{-\xi_1-\xi_2-1})}
 &\leq &\int_{\mathbb{R}^3}\int_{\mathbb{S}^2}\Phi_R(v)\Upsilon_\lambda^{-\xi_1-\xi_2-1}(v)b(\cos\theta)\left|\varrho\left(\frac{v+|v|\sigma}{2}\right)\right|d\sigma dv\\
  &\leq &\|\Phi_R\|_{L^\infty({\Upsilon_\lambda^{-\xi_1-1}})}\int_{\mathbb{R}^3}\int_{\mathbb{S}^2} \Upsilon_\lambda^{-\xi_2}(v)b(\cos\theta)\left|\varrho\left(\frac{v+|v|\sigma}{2}\right)\right|d\sigma dv\\
 &\leq &C\|\Phi_R\|_{L^\infty({\Upsilon_\lambda^{-\xi_1-1}})}\int_{\mathbb{R}^3}\int_{\mathbb{S}^2}b(\cos\theta) \Upsilon_\lambda^{-\xi_2}(v^+)|\varrho(v^+)|d\sigma dv.
\end{eqnarray*}
The last inequality follows from the fact that $|v^+|\leq |v|$ and $\xi_2\geq0$.\\
Apply the change of variables $v\to v^+$, we obtain
\begin{eqnarray*}
\|\mathcal{S}\varrho\|_{L^1(\Upsilon_\lambda^{-\xi_1-\xi_2-1})}
 &\leq &C(\theta_b,\xi)\|\Phi_R\|_{L^\infty({\Upsilon_\lambda^{-\xi_1-1}})}\times\\
 & &\int_{\mathbb{R}^3}\int_{\mathbb{S}^2}\frac{4b(\cos\theta)}{\cos^2(\theta/2)}\Upsilon_\lambda^{-\xi_2}(v^+)|\varrho(v^+)|d\sigma dv^+\\
&\leq &C(\theta_b,\xi)\|\Phi_R\|_{L^\infty({\Upsilon_\lambda^{-\xi_1-1}})}\|b\|_{L^1(\mathbb{S}^2)}\|\varrho\|_{L^1(\Upsilon_\lambda^{-\xi_2})}.
\end{eqnarray*}
By the Riesz-Thorin interpolation theorem, the above estimates on the weighted $L^1$ and $L^\infty$ norms of $\mathcal{S}\varrho$ lead to
\begin{eqnarray*}
\|\mathcal{S}\varrho\|_{L^2(\Upsilon_\lambda^{-\xi_1-\xi_2-1})}
\leq C(\theta_b,\xi)\|\Phi_R\|_{L^\infty({\Upsilon_\lambda^{-\xi_1-1}})}\|b\|_{L^1(\mathbb{S}^2)}\|\varrho\|_{L^2(\Upsilon_\lambda^{-\xi_2})}.
\end{eqnarray*}
Now, we will estimate the term $$\int_{\mathbb{R}^3}Q_R^+(F,F)Fdv,$$
by using the above bound on $\|\mathcal{S}_1\varrho\|_{L^2(\Upsilon_\lambda^{-\xi_1-\xi_2-1})}$. In order to do this, we separate $F$ into large and small velocities:
$$F=F_r+F_r^c,\mbox{ with } r<\lambda,$$
$$F_r=F\chi_{\{|v|\leq r\}}\mbox{                and           }F_r^c=F\chi_{\{|v|>r\}},$$
where $\chi_{\{|v|\leq r\}}$ and $\chi_{\{|v|>r\}}$ are the characteristic functions of the sets $\{|v|\leq r\}$ and $\{|v|>r\}$. Let $\nu$ be a positive constant. We make the following separation
\begin{equation}\label{QBoundEq2}
\int_{\mathbb{R}^3}Q_R^+(F,F)Fdv=\int_{\mathbb{R}^3}Q_R^+(F,F_r^c)Fdv+\int_{\mathbb{R}^3}Q_R^+(F,F_r)Fdv.
\end{equation}
Estimating the first term on the right hand side of $(\ref{QBoundEq2})$, we get
\begin{eqnarray}\label{QBoundEq3}\nonumber
&&\int_{\mathbb{R}^3}Q_R^+(F,F_r^c)Fdv\\\nonumber
&\leq&\int_{\mathbb{R}^3}|F_r^c(v_*)|\int_{\mathbb{R}^3}|F(v)||T_{-v_*}\mathcal{S}(T_{v_*}F)(v)|dvdv_*\\\nonumber
&\leq&\int_{\mathbb{R}^3}|F_r^c(v_*)|\|F\|_{L^2(\Upsilon_\lambda^{1+\nu/\gamma})}\|T_{-v_*}\mathcal{S}(T_{v_*}F)\|_{L^2(\Upsilon_\lambda^{-1-\nu/\gamma})}dv_*\\\nonumber
&\leq&\int_{\mathbb{R}^3}|F_r^c(v_*)|\|F\|_{L^2(\Upsilon_\lambda^{1+\nu/\gamma})}<v_*>^{|\gamma+\nu|}\|\mathcal{S}(T_{v_*}F)\|_{L^2(\Upsilon_\lambda^{-1-\nu/\gamma})}dv_*\\
&\leq&C\int_{\mathbb{R}^3}|F_r^c(v_*)|\|F\|_{L^2(\Upsilon_\lambda^{1+\nu/\gamma})}<v_*>^{|\gamma+\nu|}\|T_{v_*}F\|_{L^2(\Upsilon_\lambda^{-\nu/\gamma})}dv_*\\\nonumber
&\leq&C\int_{\mathbb{R}^3}|F_r^c(v_*)|\|F\|_{L^2(\Upsilon_\lambda^{1+\nu/\gamma})}<v_*>^{|\gamma+\nu|+|\nu|}\|F\|_{L^2(\Upsilon_\lambda^{-\nu/\gamma})}dv_*\\\nonumber
&\leq&Cr^{\gamma-k}\|F\|_{L^2(\Upsilon_\lambda^{1+\nu/\gamma})}\|F\|_{L^2(\Upsilon_\lambda^{-\nu/\gamma})}\|F\|_{L^1((1+|v|)^{|k+\nu|+|\nu|})},
\end{eqnarray}
with $k>\gamma$.
\\ We estimate the second term on the right hand side of $(\ref{QBoundEq2})$
\begin{eqnarray}\label{QBoundEq4}\nonumber
&&\int_{\mathbb{R}^3}Q_R^+(F,F_r)Fdv\\\nonumber
&\leq&\int_{\mathbb{R}^3}|F(v_*)|\int_{\mathbb{R}^3}|F_r(v)||T_{-v_*}\mathcal{S}(T_{v_*}F)(v)|dvdv_*\\\nonumber
&\leq&\int_{\mathbb{R}^3}|F(v_*)|\|F_r\|_{L^2(\Upsilon_\lambda^{(k+\nu)/\gamma})}\|T_{-v_*}\mathcal{S}(T_{v_*}F)\|_{L^2(\Upsilon_\lambda^{-(k+\nu)/\gamma})}dv_*\\\nonumber
&\leq&\int_{\mathbb{R}^3}|F(v_*)|\|F_r\|_{L^2(\Upsilon_\lambda^{(k+\nu)/\gamma})}<v_*>^{|k+\nu|}\|\mathcal{S}(T_{v_*}F)\|_{L^2(\Upsilon_\lambda^{-(k+\nu)/\gamma})}dv_*\\\nonumber
&\leq&C\int_{\mathbb{R}^3}|F(v_*)|\|F_r\|_{L^2(\Upsilon_\lambda^{(k+\nu)/\gamma})}<v_*>^{|k+\nu|}\|T_{v_*}F\|_{L^2(\Upsilon_\lambda^{-\nu/\gamma})}dv_*\\
& &\times\|\Phi_R\|_{L^{\infty}(\Upsilon_\lambda^{-k/\gamma})}\\\nonumber
&\leq&C\left(\frac{1}{\mu}\right)^{\min\{\gamma,k-\gamma\}}\int_{\mathbb{R}^3}|F(v_*)|\|F_r\|_{L^2(\Upsilon_\lambda^{(k+\nu)/\gamma})}<v_*>^{|k+\nu|}\|T_{v_*}F\|_{L^2(\Upsilon_\lambda^{-\nu/\gamma})}dv_*\\\nonumber
&\leq&C\left(\frac{1}{\mu}\right)^{\min\{\gamma,k-\gamma\}}\|F_r\|_{L^2(\Upsilon_\lambda^{(k+\nu)/\gamma})}\|F\|_{L^2(\Upsilon_\lambda^{-\nu/\gamma})}\|F\|_{L^1_{|k+\nu|+|\nu|}}\\\nonumber
&\leq&C{r}^{k-\gamma}\left(\frac{1}{\mu}\right)^{\min\{\gamma,k-\gamma\}}\|F\|_{L^2(\Upsilon_\lambda^{(\gamma+\nu)/\gamma})}\|F\|_{L^2(\Upsilon_\lambda^{-\nu/\gamma})}\|F\|_{L^1_{|k+\nu|+|\nu|}},
\end{eqnarray}
with $k>\gamma$.
Combine $(\ref{QBoundEq2})$, $(\ref{QBoundEq3})$ and $(\ref{QBoundEq4})$, we get
\begin{eqnarray}\label{QBoundEq5}
 & &\int_{\mathbb{R}^3}Q_R^+(F,F)Fdv\\\nonumber
 &\leq& \left(r^{\gamma-k}+C{r}^{k-\gamma}\left(\frac{1}{\mu}\right)^{\min\{\gamma,k-\gamma\}}\right)\|F\|_{L^2(\Upsilon_\lambda^{(\gamma+\nu)/\gamma})}\|F\|_{L^2(\Upsilon_\lambda^{-\nu/\gamma})}\|F\|_{L^1_{|k+\nu|+|\nu|}}.
\end{eqnarray}
We deduce from $(\ref{QBoundEq1})$ and $(\ref{QBoundEq5})$ that
\begin{eqnarray}\label{QBoundEq6}\nonumber
&&\left|\int_{\mathbb{R}^3}\mathcal{Q}_{\lambda}^+(F,F)Fdv\right|\\
&\leq &C(\delta,b)\|F\|_{L^{10/7}_\delta}\|F\|_{L^2_{-\delta}}\|F\|_{L^1_{2|\delta|}}\\\nonumber
 && +\left(r^{\gamma-k}+C{r}^{k-\gamma}\left(\frac{1}{\mu}\right)^{\min\{\gamma,k-\gamma\}}\right)\|F\|_{L^2(\Upsilon_\lambda^{(\gamma+\nu)/\gamma})}\|F\|_{L^2(\Upsilon_\lambda^{-\nu/\gamma})}\|F\|_{L^1_{|k+\nu|+|\nu|}}.
 \end{eqnarray}
For suitable choices of $r$ and $\mu$, we have the conclusions of the lemma .
 \end{proof}
 \begin{proposition}\label{PropL2EstimateofhLambdaN}
For all $t_0>0$, there exist constants $C$, $N_0$, $\lambda_0$  such that the solution $h_{N,\lambda}$ of $(\ref{BoltzmannBoundedKernel})$ is globally bounded in the following sense
  \begin{equation}\label{PropL2EstimateofhLambdaNEqinProp}
 \forall N\in\mathbb{N}, N>N_0, \forall \lambda>\lambda_0~~~ \sup_{t\geq t_0}\|h_{N,\lambda}\|_{\mathcal{L}^2_{-4}}<C.
\end{equation}
Moreover, if $h_{0_N}\in {\mathcal{L}^2_{-4}}$, then there exist constants $C'$, $\lambda_0$ such that
$$\forall \lambda>\lambda_0~~~ \sup_{t\geq 0}\|h_{N,\lambda}\|_{\mathcal{L}^2_{-4}}<C'.$$
 \end{proposition}
 \begin{proof}
 Use $(1-|\bar{v}|)^6\eta^{-1}{h}_{N,\lambda}$ as a test function for $(\ref{BoltzmannBoundedKernel})$, we get
 \begin{eqnarray}\label{PropL2EstimateofhLambdaNEq1}\nonumber
&&\int_{(-1,1)^3}(1-|\bar{v}|)^6\eta^{-1}\partial_t {h}_{N,\lambda}{h}_{N,\lambda}d\bar{v}\\
&=&\int_{(-1,1)^3}\mathcal{P}_N\left\{\int_{(-1,1)^3}\int_{\mathbb{S}^2}\mathcal{B}_\lambda(\bar{v},\bar{v}_*,\sigma)\eta^{-1}\right.\\\nonumber
&&\times\left[\mathcal{C}(\bar{v},\bar{v}_*,\sigma){h}_{N,\lambda}\left(\varphi\left(\frac{\varphi^{-1}(\bar{v})+\varphi^{-1}(\bar{v}_*)}{2}-\sigma\frac{|\varphi^{-1}(\bar{v})-\varphi^{-1}(\bar{v}_*)|}{2}\right)\right)\right.\\\nonumber
&&\times {h}_{N,\lambda}\left(\varphi\left(\frac{\varphi^{-1}(\bar{v})+\varphi^{-1}(\bar{v}_*)}{2}+\sigma\frac{|\varphi^{-1}(\bar{v})-\varphi^{-1}(\bar{v}_*)|}{2}\right)\right)\\\nonumber
&&\left.\left.-{h}_{N,\lambda}(\bar{v}){h}_{N,\lambda}(\bar{v}_*)\right]d\sigma d\bar{v}_*\right\}(1-|\bar{v}|)^6{h}_{N,\lambda}d\bar{v}.
\end{eqnarray}
Define as in $(\ref{hNlamdaComeBackTofNlambda})$
$$f_{N,\lambda}(v)=h_{N,\lambda}(\varphi(v))(1+|v|)^{-4},~~~v\in\mathbb{R}^3,$$
the left hand side of $(\ref{PropL2EstimateofhLambdaNEq1})$ becomes
 \begin{eqnarray}\label{PropL2EstimateofhLambdaNEq2}\nonumber
\int_{(-1,1)^3}(1-|\bar{v}|)^6\eta^{-1}\partial_t {h}_{N,\lambda}{h}_{N,\lambda}d\bar{v}
&=&\int_{\mathbb{R}^3}\partial_t f_{N,\lambda}(1+|v|^2) f_{N,\lambda}(1+|v|)^{-2}dv\\
&=&\frac{1}{2}\frac{d}{dt}\int_{\mathbb{R}^3}| f_{N,\lambda}|^2(1+|v|^2)(1+|v|)^{-2}dv.
\end{eqnarray}
Consider  the right hand side of $(\ref{PropL2EstimateofhLambdaNEq1})$
\begin{eqnarray}\label{PropL2EstimateofhLambdaNEq3}\nonumber
&&\int_{(-1,1)^3}\mathcal{P}_N\left\{\int_{(-1,1)^3}\int_{\mathbb{S}^2}\mathcal{B}_\lambda(\bar{v},\bar{v}_*,\sigma)\eta^{-1}\right.\\\nonumber
&&\times\left[\mathcal{C}(\bar{v},\bar{v}_*,\sigma){h}_{N,\lambda}\left(\varphi\left(\frac{\varphi^{-1}(\bar{v})+\varphi^{-1}(\bar{v}_*)}{2}-\sigma\frac{|\varphi^{-1}(\bar{v})-\varphi^{-1}(\bar{v}_*)|}{2}\right)\right)\right.\\\nonumber
&&\times {h}_{N,\lambda}\left(\varphi\left(\frac{\varphi^{-1}(\bar{v})+\varphi^{-1}(\bar{v}_*)}{2}+\sigma\frac{|\varphi^{-1}(\bar{v})-\varphi^{-1}(\bar{v}_*)|}{2}\right)\right)\\\nonumber
&&\left.\left.-{h}_{N,\lambda}(\bar{v}){h}_{N,\lambda}(\bar{v}_*)\right]d\sigma d\bar{v}_*\right\}(1-|\bar{v}|)^6{h}_{N,\lambda}(\bar{v})d\bar{v}\\
&=&\int_{(-1,1)^3}\left\{\int_{(-1,1)^3}\int_{\mathbb{S}^2}\mathcal{B}_\lambda(\bar{v},\bar{v}_*,\sigma)\eta^{-1}\right.\\\nonumber
&&\times\left[\mathcal{C}(\bar{v},\bar{v}_*,\sigma){h}_{N,\lambda}\left(\varphi\left(\frac{\varphi^{-1}(\bar{v})+\varphi^{-1}(\bar{v}_*)}{2}-\sigma\frac{|\varphi^{-1}(\bar{v})-\varphi^{-1}(\bar{v}_*)|}{2}\right)\right)\right.\\\nonumber
&&\times {h}_{N,\lambda}\left(\varphi\left(\frac{\varphi^{-1}(\bar{v})+\varphi^{-1}(\bar{v}_*)}{2}+\sigma\frac{|\varphi^{-1}(\bar{v})-\varphi^{-1}(\bar{v}_*)|}{2}\right)\right)\\\nonumber
&&\left.\left.-{h}_{N,\lambda}(\bar{v}){h}_{N,\lambda}(\bar{v}_*)\right]d\sigma d\bar{v}_*\right\}\mathcal{P}_N[(1-|\bar{v}|)^6{h}_{N,\lambda}(\bar{v})]d\bar{v}\\\nonumber
&\leq&C_1\int_{(-1,1)^6\times\mathbb{S}^2}\mathcal{B}_\lambda(\bar{v},\bar{v}_*,\sigma)\mathcal{C}(\bar{v},\bar{v}_*,\sigma)\eta^{-1}(1-|\bar{v}|)^6{h}_{N,\lambda}(\bar{v})\\\nonumber
&&\times{h}_{N,\lambda}\left(\varphi\left(\frac{\varphi^{-1}(\bar{v})+\varphi^{-1}(\bar{v}_*)}{2}-\sigma\frac{|\varphi^{-1}(\bar{v})-\varphi^{-1}(\bar{v}_*)|}{2}\right)\right)\\\nonumber
&&\times {h}_{N,\lambda}\left(\varphi\left(\frac{\varphi^{-1}(\bar{v})+\varphi^{-1}(\bar{v}_*)}{2}+\sigma\frac{|\varphi^{-1}(\bar{v})-\varphi^{-1}(\bar{v}_*)|}{2}\right)\right)d\sigma d\bar{v}_*d\bar{v}\\\nonumber
&&-C_2\int_{(-1,1)^6\times\mathbb{S}^2}\mathcal{B}_\lambda(\bar{v},\bar{v}_*,\sigma)\eta^{-1}(1-|\bar{v}|)^6|{h}_{N,\lambda}(\bar{v})|^2{h}_{N,\lambda}(\bar{v}_*)d\sigma d\bar{v}_*d\bar{v},
\end{eqnarray}
where the last inequality follows from assumption $\ref{Assumption2}$ and $C_1$, $C_2$ are some positive constants.
\\ We deduce from  $(\ref{PropL2EstimateofhLambdaNEq3})$ that
\begin{eqnarray}\label{PropL2EstimateofhLambdaNEq4}\nonumber
&&\int_{(-1,1)^3}\mathcal{P}_N\left\{\int_{(-1,1)^3}\int_{\mathbb{S}^2}\mathcal{B}_\lambda(\bar{v},\bar{v}_*,\sigma)\eta^{-1}\right.\\\nonumber
&&\times\left[\mathcal{C}(\bar{v},\bar{v}_*,\sigma){h}_{N,\lambda}\left(\varphi\left(\frac{\varphi^{-1}(\bar{v})+\varphi^{-1}(\bar{v}_*)}{2}-\sigma\frac{|\varphi^{-1}(\bar{v})-\varphi^{-1}(\bar{v}_*)|}{2}\right)\right)\right.\\
&&\times {h}_{N,\lambda}\left(\varphi\left(\frac{\varphi^{-1}(\bar{v})+\varphi^{-1}(\bar{v}_*)}{2}+\sigma\frac{|\varphi^{-1}(\bar{v})-\varphi^{-1}(\bar{v}_*)|}{2}\right)\right)\\\nonumber
&&\left.\left.-{h}_{N,\lambda}(\bar{v}){h}_{N,\lambda}(\bar{v}_*)\right]d\sigma dv_*\right\}(1-|\bar{v}|)^6{h}_{N,\lambda}(\bar{v})d\bar{v}\\\nonumber
&\leq& C_1\int_{\mathbb{R}^6\times \mathbb{S}^2}B_\lambda(|v-v_*|,\sigma)f_{N,\lambda}(v'_*)f_{N,\lambda}(v')f_{N,\lambda}(v)(1+|v|^2)(1+|v|)^{-2}d\sigma dv_*dv\\\nonumber
&& -C_2\int_{\mathbb{R}^6\times \mathbb{S}^2}B_\lambda(|v-v_*|,\sigma)f_{N,\lambda}(v_*)|f_{N,\lambda}(v)|^2(1+|v|^2)(1+|v|)^{-2}d\sigma dv_*dv.
\end{eqnarray}
 Combine $(\ref{PropL2EstimateofhLambdaNEq2})$ and $(\ref{PropL2EstimateofhLambdaNEq4})$, we get
 \begin{eqnarray}\label{PropL2EstimateofhLambdaNEq5}
&&\frac{1}{2}\frac{d}{dt}\int_{\mathbb{R}^3}| f_{N,\lambda}|^2(1+|v|^2)(1+|v|)^{-2}dv\\\nonumber
&\leq& C_1\int_{\mathbb{R}^6\times \mathbb{S}^2}B_\lambda(|v-v_*|,\sigma)f_{N,\lambda}(v'_*)f_{N,\lambda}(v')f_{N,\lambda}(v)(1+|v|^2)(1+|v|)^{-2}d\sigma dv_*dv\\\nonumber
&& -C_2\int_{\mathbb{R}^6\times \mathbb{S}^2}B_\lambda(|v-v_*|,\sigma)f_{N,\lambda}(v_*)|f_{N,\lambda}(v)|^2(1+|v|^2)(1+|v|)^{-2}d\sigma dv_*dv.
\end{eqnarray}
According to lemma $\ref{LemmaQBound}$, the first term on the right hand side of $(\ref{PropL2EstimateofhLambdaNEq5})$ could be bounded by
\begin{eqnarray}\label{PropL2EstimateofhLambdaNEq6}\nonumber
& & C_1\int_{\mathbb{R}^6\times \mathbb{S}^2}B_\lambda(|v-v_*|,\sigma)f_{N,\lambda}(v'_*)f_{N,\lambda}(v')f_{N,\lambda}(v)(1+|v|^2)(1+|v|)^{-2}d\sigma dv_*dv\\
&\leq & C\epsilon^{-\iota}\|f_{N,\lambda}\|_{L^2}\|f_{N,\lambda}\|_{L^{10/7}}\|f_{N,\lambda}\|_{L^1}+\epsilon\|f_{N,\lambda}\|_{L^2(\Upsilon_\lambda^{1/2})}^2\|f_{N,\lambda}\|_{L^1_{|k|}},
\end{eqnarray}
where $C$ is some positive constant.
\\ By the inequality
$$(|v-v_*|\wedge \lambda)^\gamma\geq\frac{1}{4}(|v|\wedge \lambda)^\gamma-|v_*|^\gamma,$$ 
we have
\begin{eqnarray}\label{PropL2EstimateofhLambdaNEq7a}\nonumber
&& \int_{\mathbb{R}^3\times \mathbb{S}^2}B_\lambda(|v-v_*|,\sigma)f_{N,\lambda}(v_*)d\sigma dv_*\\
&\geq  & C\int_{\mathbb{R}^3\times \mathbb{S}^2}(|v-v_*|\wedge \lambda)^\gamma b(\cos(\theta))f_{N,\lambda}(v_*)d\sigma dv_*\\\nonumber
&\geq  & C\int_{\mathbb{R}^3\times \mathbb{S}^2}b(\cos(\theta))\left(\frac{1}{4}(|v|\wedge \lambda)^\gamma-|v_*|^\gamma\right) f_{N,\lambda}(v_*)d\sigma dv_*\\\nonumber
&\geq&C(|v|\wedge \lambda)^\gamma -C\|f_{N,\lambda}\|_{L^1_\gamma}\geq C(|v|\wedge \lambda)^\gamma -C\|f_{0}\|_{L^1_2},
\end{eqnarray}
where the last inequality follows from the $L^1_2$ boundedness of $f_{N,\lambda}$, lemma $\ref{LemmaMaxwellian1}$ and $C$ is some positive constant varying from lines to lines. 
\\ Due to lemma $\ref{LemmaMaxwellian1}$, for a fix time $t_0$, there exist $B(\bar{\mathcal{O}},\alpha)\subset B(\bar{\mathcal{O}},2R)$ and $\epsilon>0$, such that $f_{N,\lambda}(t,v)>\epsilon, ~~~\forall t>t_0$ and $\forall v\in B(\bar{\mathcal{O}},\alpha)$. We have that for all $t>t_0$
\begin{eqnarray*}\nonumber
\int_{\mathbb{R}^3\times \mathbb{S}^2}B_\lambda(|v-v_*|,\sigma)f_{N,\lambda}(t,v_*)d\sigma dv_*
&\geq  & C\int_{\mathbb{R}^3}(|v-v_*|\wedge \lambda)^\gamma f_{N,\lambda}(t,v_*)dv_*\\
&\geq  & C\epsilon\int_{B(\bar{\mathcal{O}},\alpha)}(|v-v_*|\wedge \lambda)^\gamma dv_*.
\end{eqnarray*}
Let $\delta<\lambda$ be a constant not depending on $v$ such that
$$\int_{B(\bar{\mathcal{O}},\alpha)\cap\{|v-v_*|<\delta\}}1dv_*<\frac{1}{2}\int_{B(\bar{\mathcal{O}},\alpha)}1dv_*,$$
then
\begin{eqnarray}\label{PropL2EstimateofhLambdaNEq7b}\nonumber
\int_{\mathbb{R}^3\times \mathbb{S}^2}B_\lambda(|v-v_*|,\sigma)f_{N,\lambda}(t,v_*)d\sigma dv_*
&\geq  & C\epsilon\int_{B(\bar{\mathcal{O}},\alpha)\cap\{|v-v_*|\geq\delta\}}\delta^\gamma dv_*\geq C.
\end{eqnarray}
According to lemma $\ref{LemmaMaxwellian1}$ again, for $0\leq t\leq t_0$
\begin{eqnarray*}\label{PropL2EstimateofhLambdaNEq7c}
\int_{\mathbb{R}^3\times \mathbb{S}^2}B_\lambda(|v-v_*|,\sigma)f_{N,\lambda}(t,v_*)d\sigma dv_*
&\geq  & C\exp(-Ct_0)\int_{\mathcal{D}}(|v-v_*|\wedge \lambda)^\gamma \epsilon^*dv_*.
\end{eqnarray*}
Let $\delta'<\lambda$ be a constant not depending on $v$ such that
$$\int_{\mathcal{D}\cap\{|v-v_*|<\delta'\}}1dv_*<\frac{1}{2}\int_{\mathcal{D}}1dv_*,$$
then
\begin{eqnarray}\label{PropL2EstimateofhLambdaNEq7c}
\int_{\mathbb{R}^3\times \mathbb{S}^2}B_\lambda(|v-v_*|,\sigma)f_{N,\lambda}(t,v_*)d\sigma dv_*
\geq   C\exp(-Ct_0)\int_{\mathcal{D}\cap\{|v-v_*|\geq\delta'\}}\delta'^\gamma \epsilon^*dv_*\geq  C,
\end{eqnarray}
with the notice that the constant $C$ varies from lines to lines and the last inequality follows from the same argument as lemma $\ref{LemmaMaxwellian1}$.\\
Inequalities $(\ref{PropL2EstimateofhLambdaNEq7a})$, $(\ref{PropL2EstimateofhLambdaNEq7b})$ and $(\ref{PropL2EstimateofhLambdaNEq7c})$ imply
\begin{eqnarray}\label{PropL2EstimateofhLambdaNEq7}\nonumber
&& \int_{\mathbb{R}^6\times \mathbb{S}^2}B_\lambda(|v-v_*|,\sigma)f_{N,\lambda}(v_*)|f_{N,\lambda}(v)|^2(1+|v|^2)(1+|v|)^{-2}d\sigma dv_*dv\\
&\geq&C\int_{\mathbb{R}^3}(1+|v|\wedge \lambda)^\gamma |f_{N,\lambda}(v)|^2dv.
\end{eqnarray}
\\ Combine $(\ref{PropL2EstimateofhLambdaNEq5})$, $(\ref{PropL2EstimateofhLambdaNEq6})$ and $(\ref{PropL2EstimateofhLambdaNEq7})$, and choose $\epsilon$ small enough, we get
\begin{eqnarray}\label{PropL2EstimateofhLambdaNEq8}
\frac{d}{dt}\|f_{N,\lambda}\|_{L^2}^2\leq C\epsilon^{-\iota}\|f_{N,\lambda}\|_{L^2}\|f_{N,\lambda}\|_{L^{10/7}}\|f_{N,\lambda}\|_{L^1}-(C-\epsilon)\|f_{N,\lambda}\|_{L^2(\Upsilon_\lambda^{1/2})}^2,
\end{eqnarray}
where $C$ is some positive constant varying from lines to lines. By a classical argument as \cite{MouhotVillani:2004:RTS}, there exist constants $C$, $N_0$ such that
$$\forall s\geq 0, \forall t_0>0, \forall N>N_0~~~~ \sup_{t\geq t_0}\|f_{N,\lambda}\|_{L^2}<C.
$$
\end{proof}
\subsection{The convergence analysis}
\begin{theorem}\label{ThmBondh_N}
Suppose that assumptions \ref{Assumption1} and \ref{Assumption2} are satisfied. The solution $h_N$ of $(\ref{SpectralMethod2})$ is positive and uniformly bounded with respect to $N$ in $\mathcal{L}^1_{2}$ and $\mathcal{L}^2_{-4}$ norms, i.e. for all $t_0>0$ there exist constants $C$, $N_0$  such that
$$,\forall N\in\mathbb{N},N>N_0~~~ \sup_{t\geq t_0}\|h_{N}\|_{\mathcal{L}^1_2}<C,
$$
and
$$ \forall t_0>0, \forall N\in\mathbb{N},N>N_0~~~ \sup_{t\geq t_0}\|h_{N}\|_{\mathcal{L}^2_{-4}}<C.
$$
Moreover there are positive constants $\hat{C}_1$, $\hat{C}_2$, $N_0$, such that for all $\bar{v}$ in the support of $h_{N}$
$$h_{N}(\bar{v},t)\geq \hat{C}_1\exp\left(-\hat{C}_2\left|\frac{|\bar{v}|}{1-|\bar{v}|}\right|^2\right)~~~\forall N\in\mathbb{N}, N>N_0.$$
The constants $\hat{C}_1$ and $\hat{C}_2$ depend on $t$, however, they could be chosen uniformly for all $t>t_0$, where $t_0$ is an arbitrary positive time.
\end{theorem}
\begin{proof}
 Since the sequence $\{h_{N,\lambda}\}$ is uniformly bounded with respect to $N$ and $\lambda$ in $\mathcal{L}^1_{2}$ and $\mathcal{L}^2_{-4}$ norms, the proof is direct and similar to the  proofs of classical cases (for example theorem 3.2 \cite{Arkeryd:1972:OBE}). First we observe that since $h_{0_N}$ is a sum of finite compactly supported wavelets not containing the extreme points of $-1$ and $1$, then $h_{0_N}$ belongs to $\mathcal{L}^2_{-4}$. Hence  the sequence $\{h_{N,\lambda}\}$ is uniformly bounded with respect to $\lambda$ (but not $N$) in $\mathcal{L}^1_{2}$ and $\mathcal{L}^2_{-4}$ norms for all time. By Nagumo's criterion, Dunford-Pettis theorem and Smulian theorem (see \cite{Edwards:1965:FUN} and \cite{McShane:1944:Int}) there exists a subsequence $\{h_{N,\lambda_j}\}_{j=1}^\infty$ converging weakly to a positive function $\check{h}_N$ in $\mathcal{L}^1$, which is a solution of $(\ref{SpectralMethod2})$ and bounded from below by a Maxwellian. According to proposition $\ref{PropoExistenceSystemODE}$, the linear ODEs $(\ref{SpectralMethod2})$ has a unique solution, then $ {h}_N\equiv\check{h}_N\geq0.$ Since the proofs of propositions $\ref{PropL2EstimateofhLambdaN}$ and $\ref{PropoL1EstimatehNLambda}$ are still valid when $\lambda=+\infty$, we infer that $h_N$ is uniformly bounded with respect to $N$ in $\mathcal{L}^1_{2}$ and $\mathcal{L}^2_{-4}$ norms and it is also bounded from below by a Maxwellian truncated in its support. 
\end{proof}
\begin{theorem}\label{ThmConvergence}Suppose that assumptions \ref{Assumption1} and \ref{Assumption2} are satisfied. If $f_0\in L^1_{2+\gamma}$, the solution of $(\ref{SpectralMethod2})$ tends to the solution of $(\ref{BoltzmannNewFormulation})$ in the energy sense
$$\sup_{t\in[0,T]}\lim_{N\to\infty}\|h_N(t)-h(t)\|_{\mathcal{L}_2^1}=0, \forall T\in\mathbb{R},$$
which implies the limits of the mass and momentum
$$\sup_{t\in[0,T]}\lim_{N\to\infty}\|h_N(t)-h(t)\|_{\mathcal{L}^1}=0, \forall T\in\mathbb{R},$$
$$\sup_{t\in[0,T]}\lim_{N\to\infty}\|h_N(t)-h(t)\|_{\mathcal{L}_1^1}=0, \forall T\in\mathbb{R}.$$
The algorithm also has a spectral accuracy. 
\end{theorem}
\begin{proof} Take the difference between $(\ref{SpectralMethod2})$ and $(\ref{BoltzmannNewFormulation})$, multiply both sides with $\eta^{-1}$, we get
\begin{eqnarray}\label{ThmConvergenceEq1}
&&\partial_t ({h}_N-\mathcal{P}_Nh)\eta^{-1}\\\nonumber
&=&\mathcal{P}_N\left\{\int_{(-1,1)^3}\int_{\mathbb{S}^2}\mathcal{B}(\bar{v},\bar{v}_*,\sigma)\mathcal{C}(\bar{v},\bar{v}_*,\sigma)\eta^{-1}\right.\\\nonumber
&&\times\left.{h}_N\left(\varphi\left(\frac{\varphi^{-1}(\bar{v})+\varphi^{-1}(\bar{v}_*)}{2}-\sigma\frac{|\varphi^{-1}(\bar{v})-\varphi^{-1}(\bar{v}_*)|}{2}\right)\right)\right.\\\nonumber
&&\left.\times {h}_N\left(\varphi\left(\frac{\varphi^{-1}(\bar{v})+\varphi^{-1}(\bar{v}_*)}{2}+\sigma\frac{|\varphi^{-1}(\bar{v})-\varphi^{-1}(\bar{v}_*)|}{2}\right)\right)d\sigma d\bar{v}_*\right.
\\\nonumber
&&-\left.\int_{(-1,1)^3}\int_{\mathbb{S}^2}\mathcal{B}(\bar{v},\bar{v}_*,\sigma)\eta^{-1}h_N(\bar{v})h_N(\bar{v}_*)d\sigma d\bar{v}_*\right.\\\nonumber
&&-\left.\int_{(-1,1)^3}\int_{\mathbb{S}^2}\mathcal{B}(\bar{v},\bar{v}_*,\sigma)\mathcal{C}(\bar{v},\bar{v}_*,\sigma)\eta^{-1}\right.\\\nonumber
&&\times\left.{h}\left(\varphi\left(\frac{\varphi^{-1}(\bar{v})+\varphi^{-1}(\bar{v}_*)}{2}-\sigma\frac{|\varphi^{-1}(\bar{v})-\varphi^{-1}(\bar{v}_*)|}{2}\right)\right)\right.\\\nonumber
&&\left.\times {h}\left(\varphi\left(\frac{\varphi^{-1}(\bar{v})+\varphi^{-1}(\bar{v}_*)}{2}+\sigma\frac{|\varphi^{-1}(\bar{v})-\varphi^{-1}(\bar{v}_*)|}{2}\right)\right)d\sigma d\bar{v}_*\right.
\\\nonumber
&&+\left.\int_{(-1,1)^3}\int_{\mathbb{S}^2}\mathcal{B}(\bar{v},\bar{v}_*,\sigma)\eta^{-1}h(\bar{v})h(\bar{v}_*)d\sigma d\bar{v}_*\right\}
\end{eqnarray}
We deduce from $(\ref{ThmConvergenceEq1})$ and $(\ref{L1Projection})$ that
\begin{eqnarray}\label{ThmConvergenceEq2}\nonumber
&&\frac{d}{dt}\int_{(-1,1)^3}|{h}_N-\mathcal{P}_Nh|\eta^{-1}d\bar{v}\\
&=&\left\{\int_{(-1,1)^3}\int_{\mathbb{S}^2}\mathcal{B}(\bar{v},\bar{v}_*,\sigma)\mathcal{C}(\bar{v},\bar{v}_*,\sigma)\eta^{-1}\right.\\\nonumber
&&\times\left.{h}_N\left(\varphi\left(\frac{\varphi^{-1}(\bar{v})+\varphi^{-1}(\bar{v}_*)}{2}-\sigma\frac{|\varphi^{-1}(\bar{v})-\varphi^{-1}(\bar{v}_*)|}{2}\right)\right)\right.\\\nonumber
&&\left.\times {h}_N\left(\varphi\left(\frac{\varphi^{-1}(\bar{v})+\varphi^{-1}(\bar{v}_*)}{2}+\sigma\frac{|\varphi^{-1}(\bar{v})-\varphi^{-1}(\bar{v}_*)|}{2}\right)\right)d\sigma d\bar{v}_*\right.
\\\nonumber
&&-\left.\int_{(-1,1)^3}\int_{\mathbb{S}^2}\mathcal{B}(\bar{v},\bar{v}_*,\sigma)\eta^{-1}h_N(\bar{v})h_N(\bar{v}_*)d\sigma d\bar{v}_*\right.\\\nonumber
&&-\left.\int_{(-1,1)^3}\int_{\mathbb{S}^2}\mathcal{B}(\bar{v},\bar{v}_*,\sigma)\mathcal{C}(\bar{v},\bar{v}_*,\sigma)\eta^{-1}\right.\\\nonumber
&&\times\left.{h}\left(\varphi\left(\frac{\varphi^{-1}(\bar{v})+\varphi^{-1}(\bar{v}_*)}{2}-\sigma\frac{|\varphi^{-1}(\bar{v})-\varphi^{-1}(\bar{v}_*)|}{2}\right)\right)\right.\\\nonumber
&&\left.\times {h}\left(\varphi\left(\frac{\varphi^{-1}(\bar{v})+\varphi^{-1}(\bar{v}_*)}{2}+\sigma\frac{|\varphi^{-1}(\bar{v})-\varphi^{-1}(\bar{v}_*)|}{2}\right)\right)d\sigma d\bar{v}_*\right.
\\\nonumber
&&+\left.\int_{(-1,1)^3}\int_{\mathbb{S}^2}\mathcal{B}(\bar{v},\bar{v}_*,\sigma)\eta^{-1}h(\bar{v})h(\bar{v}_*)d\sigma d\bar{v}_*\right\}\mathcal{P}_N[sign(h_N-\mathcal{P}_Nh)]d\bar{v},
\end{eqnarray}
where $sign(h_N-\mathcal{P}_Nh)=1$ if $h_N-\mathcal{P}_Nh>0$, $sign(h_N-\mathcal{P}_Nh)=-1$ if $h_N-\mathcal{P}_Nh<0$ and $sign(h_N-\mathcal{P}_Nh)=0$ if $h_N-\mathcal{P}_Nh=0$. Set
$$f_N(v)=h_N(\varphi(v))(1+|v|)^{-4},$$
$$\vartheta(v)=(1+|v|^2)\mathcal{P}_N(sign(f_N-\tilde{f}_N)(v))=(1+|v|^2)\mathcal{P}_N(sign(f_N-\mathcal{P}_N{f})(v)),$$
 we transform inequality $(\ref{ThmConvergenceEq2})$ into
\begin{eqnarray}\label{ThmConvergenceEq3}\nonumber
&&\frac{d}{dt}\int_{\mathbb{R}^3}|{f}_N-\tilde{f}_N|(1+|v|^2)d{v}\\
&= &\int_{\mathbb{R}^6\times\mathbb{S}^2}B(|v-v_*|,\sigma)[f_{N_*}'f_{N}'-f_{N_*}f_{N}]\vartheta(v)d\sigma dv_*dv\\\nonumber
& &-\int_{\mathbb{R}^6\times\mathbb{S}^2}B(|v-v_*|,\sigma)[f_{*}'f'-f_{*}f]\vartheta(v)d\sigma dv_*dv\\\nonumber
&= &\frac{1}{2}\int_{\mathbb{R}^6\times\mathbb{S}^2}B(|v-v_*|,\sigma)[f_{N_*}f_{N}-f_{*}f][\vartheta(v'_*)+\vartheta(v')-\vartheta(v_*)-\vartheta(v)]d\sigma dv_*dv\\\nonumber
&= &\frac{1}{2}\int_{\mathbb{R}^6\times\mathbb{S}^2}B(|v-v_*|,\sigma)[f_{N_*}f_{N}-\tilde{f}_{N_*}\tilde{f}_{N}][\vartheta(v'_*)+\vartheta(v')-\vartheta(v_*)-\vartheta(v)]d\sigma dv_*dv\\\nonumber
& &+\frac{1}{2}\int_{\mathbb{R}^6\times\mathbb{S}^2}B(|v-v_*|,\sigma)[\tilde{f}_{N_*}\tilde{f}_{N}-f_{*}f][\vartheta(v'_*)+\vartheta(v')-\vartheta(v_*)-\vartheta(v)]d\sigma dv_*dv\\\nonumber
&= &\frac{1}{2}\int_{\{(f_N-\tilde{f}_N)(f_{N_*}-\tilde{f}_{N_*})> 0\}}B(|v-v_*|,\sigma)[f_{N_*}f_{N}-\tilde{f}_{N_*}\tilde{f}_{N}]\\\nonumber
&&\times[\vartheta(v'_*)+\vartheta(v')-\vartheta(v_*)-\vartheta(v)]d\sigma dv_*dv\\\nonumber
& &+\int_{\{(f_N-\tilde{f}_N)> 0>(f_{N_*}-\tilde{f}_{N_*})\}}B(|v-v_*|,\sigma)[f_{N_*}f_{N}-\tilde{f}_{N_*}\tilde{f}_{N}]\\\nonumber
&&\times[\vartheta(v'_*)+\vartheta(v')-\vartheta(v_*)-\vartheta(v)]d\sigma dv_*dv\\\nonumber
& &+\int_{\{(f_N-\tilde{f}_N)> 0=(f_{N_*}-\tilde{f}_{N_*})\}}B(|v-v_*|,\sigma)[f_{N_*}f_{N}-\tilde{f}_{N_*}\tilde{f}_{N}]\\\nonumber
&&\times[\vartheta(v'_*)+\vartheta(v')-\vartheta(v_*)-\vartheta(v)]d\sigma dv_*dv\\\nonumber
& &+\int_{\{(f_N-\tilde{f}_N)= 0>(f_{N_*}-\tilde{f}_{N_*})\}}B(|v-v_*|,\sigma)[f_{N_*}f_{N}-\tilde{f}_{N_*}\tilde{f}_{N}]\\\nonumber
&&\times[\vartheta(v'_*)+\vartheta(v')-\vartheta(v_*)-\vartheta(v)]d\sigma dv_*dv\\\nonumber
& &+\frac{1}{2}\int_{\mathbb{R}^6\times\mathbb{S}^2}B(|v-v_*|,\sigma)[\tilde{f}_{N_*}\tilde{f}_{N}-f_{*}f][\vartheta(v'_*)+\vartheta(v')-\vartheta(v_*)-\vartheta(v)]d\sigma dv_*dv.
\end{eqnarray}
On the set $\mathcal{I}_1=\{(f_N-\tilde{f}_{N})(f_{N_*}-\tilde{f}_{N_*})> 0\}$, we can see that  $$[\vartheta(v'_*)+\vartheta(v')-\vartheta(v_*)-\vartheta(v)]sign(f_N-\tilde{f}_{N})\leq 0\mbox{ on }  \mathcal{I}_1.$$
Therefore  
\begin{eqnarray}\label{ThmConvergenceEq3a}
& &\int_{\{(f_N-\tilde{f}_{N})(f_{N_*}-\tilde{f}_{N_*})> 0\}}B(|v-v_*|,\sigma)\\\nonumber
& &\times[f_{N_*}f_{N}-\tilde{f}_{N_*}\tilde{f}_{N}][\vartheta(v'_*)+\vartheta(v')-\vartheta(v_*)-\vartheta(v)]d\sigma dv_*dv\leq 0.
\end{eqnarray}
On the set $\mathcal{I}_2=\{(f_N-\tilde{f}_{N})>0>(f_{N_*}-\tilde{f}_{N_*})\}$, $\vartheta(v_*)=-(1+|v_*|^2)$ and $\vartheta(v)=(1+|v|^2)$. Hence
$$-2(1+|v|^2)\leq\vartheta(v'_*)+\vartheta(v')-\vartheta(v_*)-\vartheta(v)\leq 2(1+|v_*|^2)\mbox{ on }  \mathcal{I}_2,$$
which leads to
\begin{eqnarray}\label{ThmConvergenceEq3b}
& &\int_{\{(f_N-\tilde{f}_{N})>0>(f_{N_*}-\tilde{f}_{N_*})\}}B(|v-v_*|,\sigma)[f_{N_*}f_{N}-\tilde{f}_{N_*}\tilde{f}_{N}]\\\nonumber
&&\times[\vartheta(v'_*)+\vartheta(v')-\vartheta(v_*)-\vartheta(v)]d\sigma dv_*dv\\\nonumber
&=&\int_{\{(f_N-\tilde{f}_{N})>0>(f_{N_*}-\tilde{f}_{N_*})\}}B(|v-v_*|,\sigma)[f_{N}-\tilde{f}_{N}]f_{N_*}\\\nonumber
&&\times[\vartheta(v'_*)+\vartheta(v')-\vartheta(v_*)-\vartheta(v)]d\sigma dv_*dv\\\nonumber
&&+\int_{\{(f_N-\tilde{f}_{N})>0>(f_{N_*}-\tilde{f}_{N_*})\}}B(|v-v_*|,\sigma)[f_{N_*}-\tilde{f}_{N_*}]\tilde{f}_{N}\\\nonumber
&&\times[\vartheta(v'_*)+\vartheta(v')-\vartheta(v_*)-\vartheta(v)]d\sigma dv_*dv\\\nonumber
&\leq&C\int_{\mathbb{R}^6\times\mathbb{S}^2}B(|v-v_*|,\sigma)|f_{N}-\tilde{f}_{N}|\tilde{f}_{N_*}(1+|v_*|)^2d\sigma dv_*dv\\\nonumber
&&+C\int_{\mathbb{R}^6\times\mathbb{S}^2}B(|v-v_*|,\sigma)|f_{N_*}-\tilde{f}_{N_*}|\tilde{f}_{N}(1+|v|)^2d\sigma dv_*dv\\\nonumber
&\leq&C\|f_N-\tilde{f}_{N}\|_{L^1_2}.
\end{eqnarray}
On the set $\mathcal{I}_3=\{(f_N-\tilde{f}_{N})>0=(f_{N_*}-\tilde{f}_{N_*})\}$, $\vartheta(v_*)=0$ and $\vartheta(v)=(1+|v|^2)$. Hence
$$ \vartheta(v'_*)+\vartheta(v')-\vartheta(v_*)-\vartheta(v)\leq (1+|v_*|^2)\mbox{ on } \mathcal{I}_3,$$
which leads to
\begin{eqnarray}\label{ThmConvergenceEq3c}
& &\int_{\{(f_N-\tilde{f}_{N})>0=(f_{N_*}-\tilde{f}_{N_*})\}}B(|v-v_*|,\sigma)[f_{N_*}f_{N}-\tilde{f}_{N_*}\tilde{f}_{N}]\\\nonumber
&&\times[\vartheta(v'_*)+\vartheta(v')-\vartheta(v_*)-\vartheta(v)]d\sigma dv_*dv\\\nonumber
&=&\int_{\{(f_N-\tilde{f}_{N})>0=(f_{N_*}-\tilde{f}_{N_*})\}}B(|v-v_*|,\sigma)[f_{N}-\tilde{f}_{N}]\tilde{f}_{N_*}\\\nonumber
&&\times[\vartheta(v'_*)+\vartheta(v')-\vartheta(v_*)-\vartheta(v)]d\sigma dv_*dv\\\nonumber
&=&\int_{\{(f_N-\tilde{f}_{N})>0=(f_{N_*}-\tilde{f}_{N_*})\}}B(|v-v_*|,\sigma)[f_{N}-\tilde{f}_{N}]\tilde{f}_{N_*}(1+|v_*|^2)d\sigma dv_*dv\\\nonumber
&\leq&C\|f_N-\tilde{f}_{N}\|_{L^1_2}.
\end{eqnarray}
On the set $\mathcal{I}_4=\{(f_N-\tilde{f}_{N})=0>(f_{N_*}-\tilde{f}_{N_*})\}$, $\vartheta(v_*)=-(1+|v_*|^2)$ and $\vartheta(v)=0$. Hence
$$-(1+|v|^2)\leq \vartheta(v'_*)+\vartheta(v')-\vartheta(v_*)-\vartheta(v)\mbox{ on } \mathcal{I}_4,$$
which leads to
\begin{eqnarray}\label{ThmConvergenceEq3d}
& &\int_{\{(f_N-\tilde{f}_{N})=0>(f_{N_*}-\tilde{f}_{N_*})\}}B(|v-v_*|,\sigma)[f_{N_*}f_{N}-\tilde{f}_{N_*}\tilde{f}_{N}]\\\nonumber
&&\times[\vartheta(v'_*)+\vartheta(v')-\vartheta(v_*)-\vartheta(v)]d\sigma dv_*dv\\\nonumber
&=&\int_{\{(f_N-\tilde{f}_{N})=0>(f_{N_*}-\tilde{f}_{N_*})\}}B(|v-v_*|,\sigma)[f_{N_*}-\tilde{f}_{N_*}]\tilde{f}_{N}\\\nonumber
&&\times[\vartheta(v'_*)+\vartheta(v')-\vartheta(v_*)-\vartheta(v)]d\sigma dv_*dv\\\nonumber
&\leq&C\int_{\mathbb{R}^6\times\mathbb{S}^2}B(|v-v_*|,\sigma)|f_{N_*}-\tilde{f}_{N_*}|\tilde{f}_{N}(1+|v|)^2d\sigma dv_*dv\\\nonumber
&\leq&C\|f_N-\tilde{f}_{N}\|_{L^1_2}.
\end{eqnarray}
Therefore $(\ref{ThmConvergenceEq3})$,  $(\ref{ThmConvergenceEq3a})$,  $(\ref{ThmConvergenceEq3b})$,  $(\ref{ThmConvergenceEq3c})$,  $(\ref{ThmConvergenceEq3d})$     imply
\begin{eqnarray}\label{ThmConvergenceEq4}
& &\frac{d}{dt}\|f_N-\tilde{f}_{N}\|_{L^1_2}\\\nonumber
&\leq &C\|f_N-\tilde{f}_{N}\|_{L^1_2}\\\nonumber
& &+\frac{1}{2}\int_{\mathbb{R}^6\times\mathbb{S}^2}B(|v-v_*|,\sigma)[\tilde{f}_{N_*}\tilde{f}_{N}-f_{*}f][\vartheta(v'_*)+\vartheta(v')-\vartheta(v_*)-\vartheta(v)]d\sigma dv_*dv.,\end{eqnarray}
where $C$ is a constant varying from lines to lines.
\\ Apply Gronwall's inequality to $(\ref{ThmConvergenceEq4})$, we get
\begin{eqnarray}\label{ThmConvergenceEq5}
& &\|f_N(T)-\tilde{f}_{N}(T)\|_{L^1_2}\\\nonumber
& \leq &\int_0^T\int_{\mathbb{R}^6\times\mathbb{S}^2} \frac{e^{C(T-t)}}{2}B(|v-v_*|,\sigma)[\tilde{f}_{N_*}\tilde{f}_{N}-f_{*}f]\times\\\nonumber
& &[\vartheta(v'_*)+\vartheta(v')-\vartheta(v_*)-\vartheta(v)]d\sigma dv_*dvdt+e^{CT}\|f_N(0)-\tilde{f}_{N}(0)\|_{L^1_2}.
\end{eqnarray}
 Inequality $(\ref{ThmConvergenceEq5})$  implies that the accuracy of the method is indeed the accuracy of the orthogonal projection onto the subspaces created by the wavelets, in other words the method has a spectral accuracy.
\end{proof}
\section{Convergence to equilibrium}
In \cite{Mouhot:2006:RCE}, Mouhot proved that the solution of the Boltzmann equation converges to the equilibrium $\mathcal{M}$ with the rate $O(\exp(-ct))$.  Suppose that the solution $f$ of the homogeneous equation $(\ref{BoltzmannHomogeneous})$ converges to the equilibrium $\mathcal{M}$ with the velocity $O(\exp(-ct))$
$$\|f(t)-\mathcal{M}\|_{L^2}=O(\exp(-ct)).$$
Define
$$\mathcal{E}(t)=\|f(t)-\mathcal{M}\|_{L^2}=O(\exp(-ct)).$$
Fix a time $t$, according to theorem $\ref{ThmConvergence}$, there exists $N_0$ such that  
$$\|f_N(t)-f(t)\|_{L^2}\leq \frac{\mathcal{E}(t)}{2}, \forall N>N_0,$$
which means that 
\begin{equation}\label{ThmConvergencetoEquilibriumEq1}
\|f_N(t)-\mathcal{M}\|_{L^2}\leq \frac{3\mathcal{E}(t)}{2}, \forall N>N_0.
\end{equation}
Therefore by the same argument of $\cite{FilbetMouhot:2011:ASM}$, the distance between $f_N(t)$ and $\mathcal{M}$ is  $O(\exp(-ct))$. We have:
\begin{theorem}\label{ThmConvergencetoEquilibrium}Suppose that the solution $f$ of the homogeneous equation $(\ref{BoltzmannHomogeneous})$ converges to the equilibrium $\mathcal{M}$ with the velocity $O(\exp(-ct))$. Then   $f_N(t)$ converges to the equilibrium $\mathcal{M}$ with the velocity $O(\exp(-ct))$, in the sense of $(\ref{ThmConvergencetoEquilibriumEq1})$.\end{theorem}
\section{Propagation of polynomial moments}\label{secpolynomialmonemts}
In \cite{Desvillettes:1993:SAM}, \cite{Wennberg:1997:EDM}, it is proved that the solution $f$ of $(\ref{BoltzmannHomogeneous})$ satisfies the following property 
$$\forall s>0,\forall t_0>0,\sup_{t\geq t_0}\int_{\mathbb{R}^3}f(t,v)(1+|v|^s)dv<+\infty,$$
or equivalently
$$\forall s>0,\forall t_0>0,\sup_{t\geq t_0}\int_{(-1,1)^3}h(t,\bar{v})(1-|\bar{v}|)^{-s}d\bar{v}<+\infty.$$
We will establish some conditions on the filter $F_N$ such that the above property is satisfied with the solution $h_N$ of the approximate problem $(\ref{SpectralMethod2})$. The idea of constructing $F_N$ is, again, to remove some components of the wavelet representation which are close to the extreme points of $(-1,1)^3$, or in other words, to restrict $h_N$ onto $[-\zeta_N,\zeta_N]^3$ with $0<\zeta_N<1$.
\subsection{Assumption}\label{SubsecAssumptionPolynomial}
First, we establish some properties on the filter $F_N$.
\begin{assumption}\label{AssumptionPolynomialMoments} Let $n$ be a positive integer, we suppose the following assumption on the multiresolution analysis and the filter $F_N$: There exists a constant $\epsilon(N)$ such that 
$$\lim_{N\to\infty}\epsilon(N)=0,$$
and
$$\|\mathcal{P}_N^C(\eta^{-n}(\bar{v}))\|_{L^\infty([-\zeta_N,\zeta_N]^3)}<\epsilon(N),$$
where  
$$\mathcal{P}_N^C(\eta^{-n}(\bar{v})):=\eta^{-n}(\bar{v})-\mathcal{P}_N(\eta^{-n}(\bar{v})).$$
\end{assumption}
We now point out an example which satisfies this assumption. Consider again the Haar function in  $(\ref{HaarFunction})$, $(\ref{HaarPeriodizedBasis})$, $(\ref{HaarPeriodizedBasisRearrange})$ and $(\ref{HaarFilter})$. According to the definition of the filter $F_N$, the approximate function $\mathcal{P}_N(\eta^{-n}(\bar{v}))$ is supported in $[-2^{-N}(2\hat{k}_N+1),2^{-N}(2\hat{k}_N+1)]^3$. 
\begin{proposition}\label{PropoHaarPolynomialMoments} Let $\Delta$ be some constant in $(0,1)$ and suppose that $$\hat{k}_N=\left[\frac{\Delta 2^N-1}{2}\right],$$
where $[\frac{\Delta 2^N-1}{2}]$ denotes the largest integer smaller than $\frac{\Delta 2^N-1}{2}.$\\
There exists a constant $\epsilon(N)$ such that 
$$\lim_{N\to\infty}\epsilon(N)=0,$$
and
$$\|\mathcal{P}_N^C(\eta^{-n}(\bar{v}))\|_{L^\infty([-2^{-N}(2\hat{k}_N+1),2^{-N}(2\hat{k}_N+1)]^3)}<\epsilon(N).$$
\end{proposition}
\begin{remark} This technique of wavelets filtering is quite similar to the Fourier filtering technique introduced in \cite{Zuazua:2005:POC}, \cite{Zuazua:CNA:2006}, \cite{MaricaZuazua:2010:LSF}, \cite{MaricaZuazua:2010:LSS}. In order to preserve the propagation, observation and control of waves, Zuazua introduced a new Fourier filter: Suppose that the solution $u$ defined on $(0,1)$ could be written under the form of Fourier series
$$u(x)=\sum_{-\infty}^\infty a_m\exp(-2\pi m i),$$
and its approximation is 
$$u_N(x)=\sum_{-N}^N a_m\exp(-2\pi mi).$$
 Zuazua's Fourier filter is defined by removing all of the indices $m$ such that $|m|>[\Delta(N+1)]$ where $\Delta$ is a constant in $(0,1)$
$$F_Nu_N(x)=\sum_{-[\Delta(N+1)]}^{[\Delta(N+1)]} a_m\exp(-2\pi m i).$$
\end{remark}
\begin{proof}
Set
$$\mathcal{P}_N\left[\left(\frac{|\bar{v}|^2}{(1-|\bar{v}|)^2}\right)^n\right]=\sum_{k=-\hat{k}_N}^{\hat{k}_N}d_k\Phi_{N,k},$$
where
$$d_k=\int_{(-1,1)^3}\left(\frac{|\bar{v}|^2}{(1-|\bar{v}|)^2}\right)^n\Phi_{N,k}d\bar{v}.$$
Suppose that 
$$\Phi_{N,k}(\bar{v})={\phi}^{per}_{-N,k_1}(\bar{v}_1){\phi}^{per}_{-N,k_2}(\bar{v}_2){\phi}^{per}_{-N,k_3}(\bar{v}_3),$$
with $|k_1|\geq |k_2|\geq |k_3|$. Hence, $|\bar{v}|=\max\{|\bar{v}_1|,|\bar{v}_2|,|\bar{v}_3|\}\in[2^{-N}(2|k_1|-1),2^{-N}(2|k_1|+1)]$ if $k_1\ne2^{N-1}$ and $|\bar{v}|\in[0,2^{-N}]$ if $k_1=2^{N-1}$.  
\\ If $k_1\ne2^{N-1}$ and $|\bar{v}|=\max\{|\bar{v}_1|,|\bar{v}_2|,|\bar{v}_3|\}\in[2^{-N}(2|k_1|-1),2^{-N}(2|k_1|+1)]$. 
\begin{eqnarray}\label{PropoHaarPolynomialMomentsEq1}
 & &\left|\mathcal{P}_N^C\left[\left(\frac{|\bar{v}|^2}{(1-|\bar{v}|)^2}\right)^n\right]\right|\\\nonumber
 &=&\left|\left(\frac{|\bar{v}|^2}{(1-|\bar{v}|)^2}\right)^n-d_k\Phi_{N,k}\right|\\\nonumber
 & \leq& \left|\left|\frac{2^{-N}(2|k_1|-1)}{1-2^{-N}(2|k_1|-1)}\right|^{2n}-\left|\frac{2^{-N}(2|k_1|+1)}{1-2^{-N}(2|k_1|+1)}\right|^{2n}\right|\\\nonumber
 & \leq& \left|\frac{2^{-N}(2|k_1|-1)}{1-2^{-N}(2|k_1|-1)}-\frac{2^{-N}(2|k_1|+1)}{1-2^{-N}(2|k_1|+1)}\right|\\\nonumber
 & &\times\sum_{i=0}^{2n-1}\left|\frac{2^{-N}(2|k_1|-1)}{1-2^{-N}(2|k_1|-1)}\right|^{i}\left|\frac{2^{-N}(2|k_1|+1)}{1-2^{-N}(2|k_1|+1)}\right|^{2n-1-i}\\\nonumber
 & \leq& \frac{{2^{1-N}}}{(1-2^{-N}(2|k_1|-1))(1-2^{-N}(2|k_1|+1))}\\\nonumber
 & &\times\sum_{i=0}^{2n-1}\left|\frac{2^{-N}(2|k_1|-1)}{1-2^{-N}(2|k_1|-1)}\right|^{i}\left|\frac{2^{-N}(2|k_1|+1)}{1-2^{-N}(2|k_1|+1)}\right|^{2n-1-i}\\\nonumber
 & \leq& 2^{1-N}\sum_{i=0}^{2n-1}\frac{|2^{-N}(2|k_1|-1)|^i}{|1-2^{-N}(2|k_1|-1)|^{i+1}}\frac{|2^{-N}(2|k_1|+1)|^{2n-1-i}}{|1-2^{-N}(2|k_1|+1)|^{2n-i}}.
\end{eqnarray}
Now consider the function 
$$\varrho(y)=\frac{y^{j}}{(1-y)^{j+1}}, ~~~y\in(0,1),$$
for any positive integer $j$. Since
$$\varrho'(y)=\frac{jy^{j-1}(1-y)^{j+1}+(j+1)y^{j}(1-y)^{j}}{(1-y)^{2j+2}}>0,$$
the function $\varrho$ is increasing. Apply this into $(\ref{PropoHaarPolynomialMomentsEq1})$, we get
\begin{eqnarray}\label{PropoHaarPolynomialMomentsEq2}\nonumber
 & &\left|\mathcal{P}_N^C\left[\left(\frac{|\bar{v}|^2}{(1-|\bar{v}|)^2}\right)^n\right]\right|\\\nonumber
 & \leq& 2^{1-N}\sum_{i=0}^{2n-1}\frac{|2^{-N}(2\hat{k}_N+1)|^i}{|1-2^{-N}(2\hat{k}_N+1)|^{i+1}}\frac{|2^{-N}(2\hat{k}_N+1)|^{2n-1-i}}{|1-2^{-N}(2\hat{k}_N+1)|^{2n-i}}\\
 & \leq& 2^{1-N}2n\frac{|2^{-N}(2\hat{k}_N+1)|^{2n-1}}{|1-2^{-N}(2\hat{k}_N+1)|^{2n+1}}\\\nonumber
  & \leq& 2^{1-N}2n\frac{\Delta^{2n-1}}{|1-\Delta|^{2n+1}}.
\end{eqnarray}
If $k_1=2^{N-1}$ and $|\bar{v}|\in[0,2^{-N}]$.
\begin{eqnarray}\label{PropoHaarPolynomialMomentsEq3}\nonumber
 & &\left|\mathcal{P}_N^C\left[\left(\frac{|\bar{v}|^2}{(1-|\bar{v}|)^2}\right)^n\right]\right|\\
 &=&\left|\left(\frac{|\bar{v}|^2}{(1-|\bar{v}|)^2}\right)^n-\mathcal{P}_N\left[\left(\frac{|\bar{v}|^2}{(1-|\bar{v}|)^2}\right)^n\right]\right|\\\nonumber
 & \leq& \left|\frac{2^{-N}}{1-2^{-N}}\right|^{2n}.
\end{eqnarray}
The conclusion of the proposition follows from $(\ref{PropoHaarPolynomialMomentsEq2})$ and $(\ref{PropoHaarPolynomialMomentsEq3})$.
\end{proof}
\subsection{Propagation of polynomial moments}
\begin{theorem}\label{ThmPropaPolyMoments} Assuming assumptions $\ref{Assumption1}$, $\ref{Assumption2}$,  $\ref{AssumptionPolynomialMoments}$ on the multiresolution analysis and the filter, then we get the following propagation of polynomial moments property
\begin{equation}\label{ThmPropaPolyMomentsEq0}\forall s>0,\forall t_0>0,\exists N_0, \mbox{ such that }\sup_{t\geq t_0,N>N_0}\int_{(-1,1)^3}h_N(t,\bar{v})(1-|\bar{v}|)^{-s}d\bar{v}<+\infty.
\end{equation}
If $$\int_{(-1,1)^3}h_0(t,\bar{v})(1-|\bar{v}|)^{-s}d\bar{v}<+\infty,$$
then $t_0$ could be chosen to be $0$.
\end{theorem}
\begin{remark}\label{RemarkLs1Convergence} Using theorem $\ref{ThmPropaPolyMoments}$, by the same argument as theorem $\ref{ThmBondh_N}$, we can have the following property 
\begin{equation}\label{ThmPropaPolyMomentsEq00}\forall s>0,\forall t_0>0,\exists N_0, \mbox{ s.t. }\sup_{t\geq t_0,N>N_0}\int_{(-1,1)^3}|h_N(t,\bar{v})|^2(1-|\bar{v}|)^{-s+4}d\bar{v}<+\infty.
\end{equation}
Moreover, by theorem $\ref{ThmPropaPolyMoments}$, we can expect that 
$$\sup_{t\in[0,T]}\lim_{N\to\infty}\|h_N(t)-h(t)\|_{\mathcal{L}_s^1}=0, \forall T\in\mathbb{R},$$
if $h_0\in\mathcal{L}_s^1$.
\end{remark}
\begin{proof} We only prove the theorem for integer values of $s$, $s>1$, the non-integer cases could be deduced directly from the integer cases by classical interpolation arguments. First, we observe that $f_{0_N}$ are uniformly bounded with respect to $N$ in ${L}^1_s$ if $f_0$ belongs to ${L}^1_s$. We prove the theorem in two steps.
{\\\bf Step 1:} Transforming $(\ref{SpectralMethod1})$.
\\We recall $(\ref{SpectralMethod1})$
\begin{eqnarray}\label{ThmPropaPolyMomentsEq1}\nonumber
&&\partial_t {h}_N(t,\bar{v})\\\nonumber
&=&\eta\mathcal{P}_N\left\{\int_{(-1,1)^3}\int_{\mathbb{S}^2}\mathcal{B}(\bar{v},\bar{v}_*,\sigma)
\eta^{-1}\right.\\
&&\times\left[\mathcal{C}(\bar{v},\bar{v}_*,\sigma){h}_N\left(\varphi\left(\frac{\varphi^{-1}(\bar{v})+\varphi^{-1}(\bar{v}_*)}{2}-\sigma\frac{|\varphi^{-1}(\bar{v})-\varphi^{-1}(\bar{v}_*)|}{2}\right)\right)\right.\\\nonumber
&&\left.\left.\times {h}_N\left(\varphi\left(\frac{\varphi^{-1}(\bar{v})+\varphi^{-1}(\bar{v}_*)}{2}+\sigma\frac{|\varphi^{-1}(\bar{v})-\varphi^{-1}(\bar{v}_*)|}{2}\right)\right)-{h}_N(\bar{v}){h}_N(\bar{v}_*)\right]d\sigma d\bar{v}_*\right\},
\end{eqnarray}
and use $\eta^{-s}$ $(s\in\mathbb{N})$ as a test function for $(\ref{ThmPropaPolyMomentsEq1})$ to obtain
\begin{eqnarray}\label{ThmPropaPolyMomentsEq2}\nonumber
&&\int_{(-1,1)^3}\partial_t {h}_N(t,\bar{v})\eta^{-s}d\bar{v}\\\nonumber
&=&\int_{(-1,1)^3}\eta^{1-s}\mathcal{P}_N\left\{\int_{(-1,1)^3}\int_{\mathbb{S}^2}\mathcal{B}(\bar{v},\bar{v}_*,\sigma)
\eta^{-1}\right.\\\nonumber
&&\times\left[\mathcal{C}(\bar{v},\bar{v}_*,\sigma){h}_N\left(\varphi\left(\frac{\varphi^{-1}(\bar{v})+\varphi^{-1}(\bar{v}_*)}{2}-\sigma\frac{|\varphi^{-1}(\bar{v})-\varphi^{-1}(\bar{v}_*)|}{2}\right)\right)\right.\\\nonumber
&&\left.\left.\times {h}_N\left(\varphi\left(\frac{\varphi^{-1}(\bar{v})+\varphi^{-1}(\bar{v}_*)}{2}+\sigma\frac{|\varphi^{-1}(\bar{v})-\varphi^{-1}(\bar{v}_*)|}{2}\right)\right)-{h}_N(\bar{v}){h}_N(\bar{v}_*)\right]d\sigma d\bar{v}_*\right\}d\bar{v}\\\nonumber
&=&\int_{(-1,1)^3}\mathcal{P}_N[\eta^{1-s}]\left\{\int_{(-1,1)^3}\int_{\mathbb{S}^2}\mathcal{B}(\bar{v},\bar{v}_*,\sigma)
\eta^{-1}\right.\\\nonumber
&&\times\left[\mathcal{C}(\bar{v},\bar{v}_*,\sigma){h}_N\left(\varphi\left(\frac{\varphi^{-1}(\bar{v})+\varphi^{-1}(\bar{v}_*)}{2}-\sigma\frac{|\varphi^{-1}(\bar{v})-\varphi^{-1}(\bar{v}_*)|}{2}\right)\right)\right.\\\nonumber
&&\left.\left.\times {h}_N\left(\varphi\left(\frac{\varphi^{-1}(\bar{v})+\varphi^{-1}(\bar{v}_*)}{2}+\sigma\frac{|\varphi^{-1}(\bar{v})-\varphi^{-1}(\bar{v}_*)|}{2}\right)\right)-{h}_N(\bar{v}){h}_N(\bar{v}_*)\right]d\sigma d\bar{v}_*\right\}d\bar{v}\\\nonumber
&=&\int_{(-1,1)^6\times\mathbb{S}^2}\mathcal{B}(\bar{v},\bar{v}_*,\sigma)
\eta^{-s}\\\nonumber
&&\times\left[\mathcal{C}(\bar{v},\bar{v}_*,\sigma){h}_N\left(\varphi\left(\frac{\varphi^{-1}(\bar{v})+\varphi^{-1}(\bar{v}_*)}{2}-\sigma\frac{|\varphi^{-1}(\bar{v})-\varphi^{-1}(\bar{v}_*)|}{2}\right)\right)\right.\\\nonumber
&&\left.\times {h}_N\left(\varphi\left(\frac{\varphi^{-1}(\bar{v})+\varphi^{-1}(\bar{v}_*)}{2}+\sigma\frac{|\varphi^{-1}(\bar{v})-\varphi^{-1}(\bar{v}_*)|}{2}\right)\right)-{h}_N(\bar{v}){h}_N(\bar{v}_*)\right]d\sigma d\bar{v}_*d\bar{v}\\
&&-\int_{(-1,1)^3}\mathcal{P}^C_N[\eta^{1-s}]\left\{\int_{(-1,1)^3}\int_{\mathbb{S}^2}\mathcal{B}(\bar{v},\bar{v}_*,\sigma)
\eta^{-1}\right.\\\nonumber
&&\times\left[\mathcal{C}(\bar{v},\bar{v}_*,\sigma){h}_N\left(\varphi\left(\frac{\varphi^{-1}(\bar{v})+\varphi^{-1}(\bar{v}_*)}{2}-\sigma\frac{|\varphi^{-1}(\bar{v})-\varphi^{-1}(\bar{v}_*)|}{2}\right)\right)\right.\\\nonumber
&&\left.\left.\times {h}_N\left(\varphi\left(\frac{\varphi^{-1}(\bar{v})+\varphi^{-1}(\bar{v}_*)}{2}+\sigma\frac{|\varphi^{-1}(\bar{v})-\varphi^{-1}(\bar{v}_*)|}{2}\right)\right)\right]d\sigma d\bar{v}_*\right\}d\bar{v}\\\nonumber
& &+\int_{(-1,1)^3}\mathcal{P}^C_N[\eta^{1-s}]\left\{\int_{(-1,1)^3}\int_{\mathbb{S}^2}\mathcal{B}(\bar{v},\bar{v}_*,\sigma)
\eta^{-1}{h}_N(\bar{v}){h}_N(\bar{v}_*)d\sigma d\bar{v}_*\right\}d\bar{v}.
\end{eqnarray}
Now, consider the second term on the right hand side of $(\ref{ThmPropaPolyMomentsEq2})$, we have 
\begin{eqnarray}\label{ThmPropaPolyMomentsEq3}\nonumber
&&-\int_{(-1,1)^3}\mathcal{P}^C_N[\eta^{1-s}]\left\{\int_{(-1,1)^3}\int_{\mathbb{S}^2}\mathcal{B}(\bar{v},\bar{v}_*,\sigma)
\eta^{-1}\right.\\\nonumber
&&\times\left[\mathcal{C}(\bar{v},\bar{v}_*,\sigma){h}_N\left(\varphi\left(\frac{\varphi^{-1}(\bar{v})+\varphi^{-1}(\bar{v}_*)}{2}-\sigma\frac{|\varphi^{-1}(\bar{v})-\varphi^{-1}(\bar{v}_*)|}{2}\right)\right)\right.\\\nonumber
&&\left.\left.\times {h}_N\left(\varphi\left(\frac{\varphi^{-1}(\bar{v})+\varphi^{-1}(\bar{v}_*)}{2}+\sigma\frac{|\varphi^{-1}(\bar{v})-\varphi^{-1}(\bar{v}_*)|}{2}\right)\right)\right]d\sigma d\bar{v}_*\right\}d\bar{v}\\\nonumber
&=&-\int_{(-1,1)^3}\mathcal{P}^C_N[\eta^{1-s}]\chi_{(-\zeta_N,\zeta_N)^3}\left\{\int_{(-1,1)^3}\int_{\mathbb{S}^2}\mathcal{B}(\bar{v},\bar{v}_*,\sigma)
\eta^{-1}\right.\\\nonumber
&&\times\left[\mathcal{C}(\bar{v},\bar{v}_*,\sigma){h}_N\left(\varphi\left(\frac{\varphi^{-1}(\bar{v})+\varphi^{-1}(\bar{v}_*)}{2}-\sigma\frac{|\varphi^{-1}(\bar{v})-\varphi^{-1}(\bar{v}_*)|}{2}\right)\right)\right.\\\nonumber
&&\left.\left.\times {h}_N\left(\varphi\left(\frac{\varphi^{-1}(\bar{v})+\varphi^{-1}(\bar{v}_*)}{2}+\sigma\frac{|\varphi^{-1}(\bar{v})-\varphi^{-1}(\bar{v}_*)|}{2}\right)\right)\right]d\sigma d\bar{v}_*\right\}d\bar{v}\\
&&-\int_{(-1,1)^3}\eta^{1-s}\chi_{\mathbb{R}^3\backslash(-\zeta_N,\zeta_N)^3}\left\{\int_{(-1,1)^3}\int_{\mathbb{S}^2}\mathcal{B}(\bar{v},\bar{v}_*,\sigma)
\eta^{-1}\right.\\\nonumber
&&\times\left[\mathcal{C}(\bar{v},\bar{v}_*,\sigma){h}_N\left(\varphi\left(\frac{\varphi^{-1}(\bar{v})+\varphi^{-1}(\bar{v}_*)}{2}-\sigma\frac{|\varphi^{-1}(\bar{v})-\varphi^{-1}(\bar{v}_*)|}{2}\right)\right)\right.\\\nonumber
&&\left.\left.\times {h}_N\left(\varphi\left(\frac{\varphi^{-1}(\bar{v})+\varphi^{-1}(\bar{v}_*)}{2}+\sigma\frac{|\varphi^{-1}(\bar{v})-\varphi^{-1}(\bar{v}_*)|}{2}\right)\right)\right]d\sigma d\bar{v}_*\right\}d\bar{v}\\\nonumber
&\leq&-\int_{(-1,1)^3}\mathcal{P}^C_N[\eta^{1-s}]\chi_{(-\zeta_N,\zeta_N)^3}\left\{\int_{(-1,1)^3}\int_{\mathbb{S}^2}\mathcal{B}(\bar{v},\bar{v}_*,\sigma)
\eta^{-1}\right.\\\nonumber
&&\times\left[\mathcal{C}(\bar{v},\bar{v}_*,\sigma){h}_N\left(\varphi\left(\frac{\varphi^{-1}(\bar{v})+\varphi^{-1}(\bar{v}_*)}{2}-\sigma\frac{|\varphi^{-1}(\bar{v})-\varphi^{-1}(\bar{v}_*)|}{2}\right)\right)\right.\\\nonumber
&&\left.\left.\times {h}_N\left(\varphi\left(\frac{\varphi^{-1}(\bar{v})+\varphi^{-1}(\bar{v}_*)}{2}+\sigma\frac{|\varphi^{-1}(\bar{v})-\varphi^{-1}(\bar{v}_*)|}{2}\right)\right)\right]d\sigma d\bar{v}_*\right\}d\bar{v}\\\nonumber
&\leq&\epsilon(N)\int_{(-1,1)^6\times \mathbb{S}^2}\mathcal{B}(\bar{v},\bar{v}_*,\sigma)
\eta^{-1}\\\nonumber
&&\times\left[\mathcal{C}(\bar{v},\bar{v}_*,\sigma){h}_N\left(\varphi\left(\frac{\varphi^{-1}(\bar{v})+\varphi^{-1}(\bar{v}_*)}{2}-\sigma\frac{|\varphi^{-1}(\bar{v})-\varphi^{-1}(\bar{v}_*)|}{2}\right)\right)\right.\\\nonumber
&&\left.\times {h}_N\left(\varphi\left(\frac{\varphi^{-1}(\bar{v})+\varphi^{-1}(\bar{v}_*)}{2}+\sigma\frac{|\varphi^{-1}(\bar{v})-\varphi^{-1}(\bar{v}_*)|}{2}\right)\right)\right]d\sigma d\bar{v}_*d\bar{v},
\end{eqnarray}
where we use that fact 
$$\mathcal{P}^C_N[\eta^{1-s}]\chi_{\mathbb{R}^3\backslash(-\zeta_N,\zeta_N)^3}=(Id-\mathcal{P}_N)[\eta^{1-s}]\chi_{\mathbb{R}^3\backslash(-\zeta_N,\zeta_N)^3}=\eta^{1-s}\chi_{\mathbb{R}^3\backslash(-\zeta_N,\zeta_N)^3},$$
since $\mathcal{P}_N[\eta^{1-s}]$ is supported in $(-\zeta_N,\zeta_N)^3$, after that assumption $\ref{AssumptionPolynomialMoments}$ is applied to get the final inequality.
\\ According to assumption $\ref{AssumptionPolynomialMoments}$, the third term on the right hand side of $(\ref{ThmPropaPolyMomentsEq2})$ could be bounded in the following way
\begin{eqnarray}\label{ThmPropaPolyMomentsEq4a}\nonumber
& &\int_{(-1,1)^3}\mathcal{P}^C_N[\eta^{1-s}]\left\{\int_{(-1,1)^3}\int_{\mathbb{S}^2}\mathcal{B}(\bar{v},\bar{v}_*,\sigma)
\eta^{-1}{h}_N(\bar{v}){h}_N(\bar{v}_*)d\sigma d\bar{v}_*\right\}d\bar{v}\\
& \leq&\int_{(-1,1)^6\times\mathbb{S}^2}\epsilon(N)\eta^{-1}\mathcal{B}(\bar{v},\bar{v}_*,\sigma)
{h}_N(\bar{v}){h}_N(\bar{v}_*)d\sigma d\bar{v}_*d\bar{v},
\end{eqnarray}
with the notice that since $h_N(\bar{v})$ is supported in $(-\zeta_N,\zeta_N)^3$, we can suppose that $\eta^{1-s}(\bar{v})$ is supported in $(-\zeta_N,\zeta_N)^3$ as well and hence assumption $\ref{AssumptionPolynomialMoments}$ is applicable. 
\\ We use again the change of variables mapping $\varphi$ to define
$$f_N(v)=h_N(\varphi(v))(1+|v|)^{-4}.$$
Inequalities $(\ref{ThmPropaPolyMomentsEq2})$, $(\ref{ThmPropaPolyMomentsEq3})$ and $(\ref{ThmPropaPolyMomentsEq4a})$ lead to
\begin{eqnarray}\label{ThmPropaPolyMomentsEq4}
& &\int_{\mathbb{R}^3}\partial_t f_N|v|^{2s}dv\\\nonumber
& \leq &\int_{\mathbb{R}^6\times\mathbb{S}^2}B(|v-v_*|,\sigma)[f_{N*}'f_N'-f_{N*}f_N]|v|^{2s}d\sigma dv_*dv\\\nonumber
& &+\epsilon(N)\int_{\mathbb{R}^6\times\mathbb{S}^2}B(|v-v_*|,\sigma)[f_{N*}'f_N'+f_{N*}f_N]|v|^{2}d\sigma dv_*dv\\\nonumber
& \leq &\int_{\mathbb{R}^6\times\mathbb{S}^2}B(|v-v_*|,\sigma)[f_{N*}'f_N'-f_{N*}f_N](|v|^{2s}+\epsilon(N)|v|^{2})d\sigma dv_*dv\\\nonumber
& &+2\epsilon(N)\int_{\mathbb{R}^6\times\mathbb{S}^2}B(|v-v_*|,\sigma)f_{N*}f_N|v|^{2}d\sigma dv_*dv\\\nonumber
& \leq &\frac{1}{2}\int_{\mathbb{R}^6\times\mathbb{S}^2}B(|v-v_*|,\sigma)f_{N*}f_N(|v'_*|^{2s}+|v'|^{2s}-|v_*|^{2s}-|v|^{2s})d\sigma dv_*dv\\\nonumber
& &+2\epsilon(N)\int_{\mathbb{R}^6\times\mathbb{S}^2}B(|v-v_*|,\sigma)f_{N*}f_N|v|^{2}d\sigma dv_*dv,
\end{eqnarray}
where the last inequality follows from the usual change of variables $(v,v_*)\to(v',v'_*)$.
{\\\bf Step 2:} Using Povzner's inequality. 
\\ By Povzner's inequality (Theorem 4.1 \cite{Wennberg:1997:EDM}), we get from $(\ref{ThmPropaPolyMomentsEq4})$ that
\begin{eqnarray}\label{ThmPropaPolyMomentsEq5}
\int_{\mathbb{R}^3}\partial_t f_N|v|^{2s}dv& \leq &C\int_{\mathbb{R}^6}f_{N*}f_N|v_*|^{2s-1}|v||v-v_*|^\gamma dv_*dv\\\nonumber
& &-C\int_{\mathbb{R}^6}f_{N*}f_N(|v_*|^{2s}+|v|^{2s})|v-v_*|^\gamma  dv_*dv\\\nonumber
& &+C\epsilon(N)\int_{\mathbb{R}^6}f_{N*}f_N|v|^{2} |v-v_*|^\gamma dv_*dv.
\end{eqnarray}
Since
$$|v_*|^{2s-1}|v||v-v_*|^\gamma\leq |v_*|^{2s-1}|v|(1+|v|+|v_*|)\leq (1+|v_*|^{2s})(1+|v|^2),$$
the first term on the right hand side of $(\ref{ThmPropaPolyMomentsEq5})$ could be bounded by
\begin{equation}\label{ThmPropaPolyMomentsEq6}
C\left(1+\int_{\mathbb{R}^3}f_N|v|^{2s}dv\right).
\end{equation}
We estimate the second term on the right hand side of $(\ref{ThmPropaPolyMomentsEq5})$
\begin{eqnarray}\label{ThmPropaPolyMomentsEq7}\nonumber
&  &\int_{\mathbb{R}^6\times\mathbb{S}^2}f_{N*}f_N|v_*|^{2s}|v-v_*|^\gamma d\sigma dv_*dv\\
& \geq & C\int_{\mathbb{R}^3}f_N|v|^{2s+\gamma}\geq C\left(f_N|v|^{2s}\right)^{\frac{2s+\gamma}{2s}},
\end{eqnarray}
with the notice that the results of lemma $\ref{LemmaMaxwellian1}$ still hold with $\lambda=\infty$.\\
We now estimate the third term on the right hand side of $(\ref{ThmPropaPolyMomentsEq5})$
\begin{eqnarray}\label{ThmPropaPolyMomentsEq8}
C\epsilon(N)\int_{\mathbb{R}^6}f_{N*}f_N|v|^{2} |v-v_*|^\gamma dv_*dv\leq C\epsilon(N)\int_{\mathbb{R}^3}f_N(C_\epsilon+\epsilon |v|^{2s+\gamma})dv,
\end{eqnarray}
where the  inequality follows from the fact that the energy of $f_N$ is uniformly bounded with respect to $N$ and  Young's inequality with a small positive constant $\epsilon$.
Combine $(\ref{ThmPropaPolyMomentsEq5})$, $(\ref{ThmPropaPolyMomentsEq6})$, $(\ref{ThmPropaPolyMomentsEq7})$ and  $(\ref{ThmPropaPolyMomentsEq8})$ with the assumption that $N$ is sufficiently large or $\epsilon(N)$ is sufficiently small
we get 
\begin{eqnarray}\label{ThmPropaPolyMomentsEq9}
\int_{\mathbb{R}^3}\partial_t f_N|v|^{2s}dv\leq  C\left(\int_{\mathbb{R}^3} f_N|v|^{2s}dv+1\right)-C\left(\int_{\mathbb{R}^3} f_N|v|^{2s}dv\right)^{1+\frac{\gamma}{2s}}.
\end{eqnarray}
Define 
$$Y(t)=\int_{\mathbb{R}^3} f_N|v|^{2s}dv,$$
inequality $(\ref{ThmPropaPolyMomentsEq9})$ becomes
$$\frac{d Y}{dt}\leq C(Y+1)-CY^{1+\frac{\gamma}{2s}}.$$
Proceed similarly as the classical case \cite{Wennberg:1997:EDM}, we get the conclusion of the theorem.
\end{proof}
\section{Propagation of exponential moments}\label{secexponentialmonemts}
In \cite{ACGM:2013:NAC} and \cite{GambaPanferovVillani:2009:UMS}, it is proved that the solution of the homogeneous Boltzmann equation is bounded from above by a Maxwellian. Let us recall theorem 2 \cite{ACGM:2013:NAC}.
{\\\bf Theorem 2 \cite{ACGM:2013:NAC}} (Propagation of exponential moments)
{\\\it Let $f$ be an energy-conserving solution to the homogeneous Boltzmann equation $(\ref{BoltzmannHomogeneous})$ on $[0,+\infty)$ with initial data $f_0\in L^1_2$. Assume moreover that the initial data satisfies for some $s\in[\gamma,2]$
\begin{equation}\label{IntialDataExponentialAssumpition}
\int_{\mathbb{R}^3}f_0(v)(a_0|v|^s)dv\leq C_0.
\end{equation}
Then there are some constants $C,a>0$ which depend only on $b,\gamma$ and the initial mass, energy and $a_0$, $C_0$ in $(\ref{IntialDataExponentialAssumpition})$ such that
\begin{equation}\label{GambaUpperMaxwellianBounds}
\int_{\mathbb{R}^3}f(t,v)\exp(a|v|^s)dv<C.
\end{equation}
}
Our task in this section is to preserve this property at the numerical level
\begin{equation}\label{UpperMaxwellianBoundsh_N}
\int_{(-1,1)^3}h_N(t,\bar{v})\exp\left(a\left(\frac{|\bar{v}|}{1-|\bar{v}|}\right)^s\right)d\bar{v}<C,
\end{equation}
or if we use the $f_N$ formulation 
$$f_N(v)=h_N(\varphi(v))(1+|v|)^{-4},$$
we should have
\begin{equation}\label{UpperMaxwellianBoundsh_N}
\int_{\mathbb{R}^3}f_N(t,v)\exp(a|v|^s)dv<C.
\end{equation}
\subsection{Assumption}
Since 
$$\int_{\mathbb{R}^3}f_0\exp(a_0|v|^s)dv\leq C_0,$$
we have 
$$\int_{(-1,1)^3}h_0\exp\left(a\left(\frac{|\bar{v}|}{1-|\bar{v}|}\right)^s\right)d\bar{v}<C_0.$$
Therefore, we need the following property as well for each initial datum of the approximate equation $(\ref{SpectralMethod2})$
$$\int_{(-1,1)^3}\eta \mathcal{P}_N[h_0\eta^{-1}]\exp\left(a\left(\frac{|\bar{v}|}{1-|\bar{v}|}\right)^s\right)d\bar{v}<C,$$
or equivalently
\begin{equation}\label{BeforeAssumptionExponentialMoments}
\int_{(-1,1)^3} h_0\eta^{-1}\mathcal{P}_N\left[\eta\exp\left(a\left(\frac{|\bar{v}|}{1-|\bar{v}|}\right)^s\right)\right]d\bar{v}<C.
\end{equation}
In order to have $(\ref{BeforeAssumptionExponentialMoments})$ we establish some further properties on the multiresolution analysis and the filter $F_N$ (notice that we always assume assumptions  $\ref{Assumption1}$, $\ref{Assumption2}$, $\ref{AssumptionPolynomialMoments}$ are satisfied).
\begin{assumption}\label{AssumptionExponentialMoments} Let $q$, $a$ be positive  constants. We impose the following assumption on the multiresolution analysis and the filter $F_N$: There exist  constants $N_{0}$, $\bar{\mathcal{K}}$, such that
\begin{eqnarray}\label{AssumptionExponentialMomentsEq0}
\forall N>N_0,\mathcal{P}_N[\eta\exp(a\eta^q)]\leq \bar{\mathcal{K}}\eta\exp(a\eta^q).
\end{eqnarray}
\end{assumption}
A consequence of this assumption is that $(\ref{BeforeAssumptionExponentialMoments})$ follows directly from $(\ref{IntialDataExponentialAssumpition})$. We will point out an example which satisfies this assumption. Consider again the Haar function in  $(\ref{HaarFunction})$, $(\ref{HaarPeriodizedBasis})$, $(\ref{HaarPeriodizedBasisRearrange})$ and $(\ref{HaarFilter})$. According to the definition of the filter $F_N$, the approximate function $\mathcal{P}_N[\eta\exp(a\eta^q)]$ is supported in $[-2^{-N}(2\hat{k}_N+1),2^{-N}(2\hat{k}_N+1)]^3$. 
\begin{proposition}\label{PropoHaarExponentialMoments} Let $\Delta$ be some constant in $(0,1)$ and suppose that $$\hat{k}_N=\left[\frac{\Delta 2^N-1}{2}\right],$$
where $[\frac{\Delta 2^N-1}{2}]$ denotes the largest integer smaller than $\frac{\Delta 2^N-1}{2}.$\\
There exist  constants $N_{0}$, $\bar{\mathcal{K}}$, such that
\begin{eqnarray*}
\forall N>N_0,\mathcal{P}_N[\eta\exp(a\eta^q)]\leq \bar{\mathcal{K}}\eta\exp(a\eta^q).
\end{eqnarray*}
\end{proposition}
\begin{proof}
Set
$$\mathcal{P}_N\left[\eta\exp(a\eta^q)\right]=\sum_{k=-\hat{k}_N}^{\hat{k}_N}d_k\Phi_{N,k},$$
where
$$d_k=\int_{(-1,1)^3}\eta\exp(a\eta^q)\Phi_{N,k}d\bar{v}.$$
Suppose that 
$$\Phi_{N,k}(\bar{v})={\phi}^{per}_{-N,k_1}(\bar{v}_1){\phi}^{per}_{-N,k_2}(\bar{v}_2){\phi}^{per}_{-N,k_3}(\bar{v}_3),$$
with $k=(k_1,k_2,k_3)$ and $|k_1|\geq |k_2|\geq |k_3|$. Hence, $|\bar{v}|=\max\{|\bar{v}_1|,|\bar{v}_2|,|\bar{v}_3|\}\in[2^{-N}(2|k_1|-1),2^{-N}(2|k_1|+1)]$ if $k_1\ne2^{N-1}$ and $|\bar{v}|\in[0,2^{-N}]$ if $k_1=2^{N-1}$.  
\\ If $k_1\ne2^{N-1}$ and $|\bar{v}|=\max\{|\bar{v}_1|,|\bar{v}_2|,|\bar{v}_3|\}\in[2^{-N}(2|k_1|-1),2^{-N}(2|k_1|+1)]$.  
\begin{eqnarray}\label{PropoHaarExponentialMomentsEq1}
& &\frac{d_k\Phi_{N,k}(\bar{v})}{\left(1+\frac{|\bar{v}|^2}{(1-|\bar{v}|)^2}\right)^{-1}\exp\left(a\left(\frac{|\bar{v}|^2}{(1-|\bar{v}|)^2}\right)^q\right)}\\\nonumber
&\leq&\frac{1+\frac{(2^{-N}(2|k_1|+1))^2}{(1-2^{-N}(2|k_1|+1))^2}}{1+\frac{(2^{-N}(2|k_1|-1))^2}{(1-2^{-N}(2|k_1|-1))^2}}\left[\exp\left(a\left(\frac{(2^{-N}(2|k_1|+1))^2}{(1-2^{-N}(2|k_1|+1))^2}\right)^q\right)\right.\\\nonumber
& &\left.-\exp\left(a\left(\frac{(2^{-N}(2|k_1|-1))^2}{(1-2^{-N}(2|k_1|-1))^2}\right)^q\right)\right]\\\nonumber
&\leq &\frac{1+\frac{(2^{-N}(2|k_1|+1))^2}{(1-2^{-N}(2|k_1|+1))^2}}{1+\frac{(2^{-N}(2|k_1|-1))^2}{(1-2^{-N}(2|k_1|-1))^2}}\left[2\exp\left(a\left(\frac{\Delta}{1-\Delta}\right)^{2q}\right)\right]\leq C,
\end{eqnarray}
where the second inequality follows from a similar argument as proposition $\ref{PropoHaarPolynomialMoments}$.
\\ If $k_1=2^{N-1}$ and $|\bar{v}|\in[0,2^{-N}]$. 
\begin{eqnarray}\label{PropoHaarExponentialMomentsEq2}
& &\frac{d_k\Phi_{N,k}(\bar{v})}{\left(1+\frac{|\bar{v}|^2}{(1-|\bar{v}|)^2}\right)^{-1}\exp\left(a\left(\frac{|\bar{v}|^2}{(1-|\bar{v}|)^2}\right)^q\right)}\\\nonumber
&\leq&{\left(1+\frac{|2^{-N}|^2}{(1-2^{-N})^2}\right)\exp\left(a\left(\frac{|2^{-N}|^2}{(1-2^{-N})^2}\right)^q\right)}\leq C.
\end{eqnarray}
The two inequalities $(\ref{PropoHaarExponentialMomentsEq1})$ and $(\ref{PropoHaarExponentialMomentsEq2})$ imply the conclusion of the proposition.
\end{proof}
\subsection{Propagation of exponential moments}
\begin{theorem}\label{ThmUpperMaxwellBounds}Assume that the assumptions $\ref{Assumption1}$, $\ref{Assumption2}$, $\ref{AssumptionPolynomialMoments}$, $\ref{AssumptionExponentialMoments}$ are all satisfied. Assume moreover that the initial data satisfies for some $s\in[\gamma,1]$
\begin{equation}\label{ThmUpperMaxwellBoundsEq0a}
\int_{\mathbb{R}^3}f_0(v)(a_0|v|^s)dv\leq C_0.
\end{equation}
Then there are some constants $C,a,N_0>0$ which depend only on the equation, the initial mass, momentum energy and $a_0$, $C_0$ in $(\ref{IntialDataExponentialAssumpition})$ such that
\begin{equation}\label{ThmUpperMaxwellBoundsEq0b}
\int_{(-1,1)^3}h_N(t,\bar{v})\exp\left(a\left(\frac{|\bar{v}|}{1-|\bar{v}|}\right)^s\right)d\bar{v}<C,~~~\forall N>N_0.
\end{equation}
\end{theorem}
\begin{proof} We define
\begin{equation}\label{ThmUpperMaxwellBoundsEq1}
m_p^N(t)=\int_{\mathbb{R}^3}f_N(t,v)|v|^pdv,~~~p\in\mathbb{R}_+.
\end{equation}
We now prove the theorem in two steps.
{\\\bf Step 1:} Estimate $m_{sp}^N$, with $0<s\leq 1$ and $p\geq 2/s$.
\\ A similar argument as theorem $\ref{ThmPropaPolyMoments}$ gives
\begin{eqnarray}\label{ThmUpperMaxwellBoundsEq2}\nonumber
& &\int_{\mathbb{R}^3}\partial_tf_N|v|^{sp}\\\nonumber
&\leq&\int_{\mathbb{R}^6\times\mathbb{S}^2}|v-v_*|^\gamma b(\cos\theta)f_{N*}f_N[|v'_*|^{sp}+|v'|^{sp}-|v_*|^{sp}-|v|^{sp}]d\sigma dv_*dv\\
& &+2\epsilon(N)\int_{\mathbb{R}^6}|v-v_*|^\gamma b(\cos\theta)f_{N*}f_N|v|^2d\sigma dv_*dv.
\end{eqnarray}
We recall the sharp Povzner Lemma (Lemma 3 \cite{ACGM:2013:NAC}) for $q\geq 1$
\begin{equation}\label{SharpPovznerLemma}
\int_{\mathbb{S}^2}(|v'|^{2q}+|v'_*|^{2q})b(\cos\theta)d\sigma\leq \gamma_q(|v|^2+|v_*|^2)^q,
\end{equation}
where $\gamma_q$ are positive constants satisfying $q\to\gamma_q$ is strictly decreasing and tends to $0$ as $q\to\infty$.
\\ Apply $(\ref{SharpPovznerLemma})$ to $(\ref{ThmUpperMaxwellBoundsEq2})$, we get  
\begin{eqnarray}\label{ThmUpperMaxwellBoundsEq3}\nonumber
\frac{d}{dt}m_{sp}^N&\leq& \gamma_{\frac{sp}{2}}\int_{\mathbb{R}^6}f_Nf_{N*}\left[(|v|^2+|v_*|^2)^{\frac{sp}{2}}-|v|^{sp}-|v_*|^{sp}\right]|v-v_*|^\gamma dv_*dv\\
 & &-2\left(1-\gamma_{\frac{sp}{2}}\right)\int_{\mathbb{R}^6}f_Nf_{N*}|v|^{sp}|v-v_*|^\gamma dv_*dv\\\nonumber
 & &+2\epsilon(N)\int_{\mathbb{R}^6\times\mathbb{S}^2}|v-v_*|^\gamma f_{N*}f_N|v|^2dv_*dv,
\end{eqnarray}
with the normalized assumption
$$\int_{\mathbb{S}^2}b(\cos\theta)d\sigma=1.$$
By the inequalities
$$(|v|^2+|v_*|^2)^{\frac{sp}{2}}\leq (|v|^s+|v_*|^s)^{p},\mbox{ for }0<{s}\leq 1,$$
and 
$$\sum_{k=1}^{\left[\frac{p+1}{2}\right]-1}C^k_p(a^kb^{p-k}+a^{p-k}b^k)\leq (a+b)^p-a^p-b^p\leq \sum_{k=1}^{\left[\frac{p+1}{2}\right]}C^k_p(a^kb^{p-k}+a^{p-k}b^k),$$
(see Lemma 2 in \cite{BobylevGambaPanferov:2004:MIH}), we can bound the first term on the right hand side of $(\ref{ThmUpperMaxwellBoundsEq3})$
\begin{eqnarray}\label{ThmUpperMaxwellBoundsEq4}\nonumber
&& \gamma_{\frac{sp}{2}}\int_{\mathbb{R}^6}f_Nf_{N*}\left[(|v|^2+|v_*|^2)^{\frac{sp}{2}}-|v|^{sp}-|v_*|^{sp}\right]|v-v_*|^\gamma dv_*dv\\
 & \leq &2\gamma_{\frac{sp}{2}}\sum_{k=1}^{\left[\frac{p+1}{2}\right]}C^k_p(m_{sk+\gamma}^Nm_{s(p-k)}^N+m_{sk}^Nm_{s(p-k)+\gamma}^N).
\end{eqnarray}
We then define
$$S_{s,p}:=\sum_{k=1}^{\left[\frac{p+1}{2}\right]}C^k_p(m_{sk+\gamma}^Nm_{s(p-k)}^N+m_{sk}^Nm_{s(p-k)+\gamma}^N).$$
The second term could be bounded from below by 
\begin{eqnarray}\label{ThmUpperMaxwellBoundsEq5}
2\left(1-\gamma_{\frac{sp}{2}}\right)\int_{\mathbb{R}^6}f_Nf_{N*}|v|^{sp}|v-v_*|^\gamma\geq2\left(1-\gamma_{\frac{sp}{2}}\right)\bar{C}_\gamma [m_{sp+\gamma}^N+C],
\end{eqnarray}
where $\bar{C}_\gamma$ depends on $\gamma$ and the initial data and we use lemma $\ref{LemmaMaxwellian1}$, with the assumption that $N$ is sufficiently large.
\\ We can also estimate the third term 
\begin{eqnarray}\label{ThmUpperMaxwellBoundsEq6}
 2\epsilon(N)\int_{\mathbb{R}^6\times\mathbb{S}^2}|v-v_*|^\gamma b(\cos\theta)f_{N*}f_N|v|^2dv_*dv\leq 2C\epsilon(N) [m_{\gamma+2}^N+C].
\end{eqnarray}
Combine $(\ref{ThmUpperMaxwellBoundsEq3})$, $(\ref{ThmUpperMaxwellBoundsEq4})$, $(\ref{ThmUpperMaxwellBoundsEq5})$ and $(\ref{ThmUpperMaxwellBoundsEq6})$, with the assumption that $\epsilon(N)$ is small enough, we get
\begin{eqnarray}\label{ThmUpperMaxwellBoundsEq7}
\frac{d}{dt}m_{sp}^N&\leq& 2\gamma_{\frac{sp}{2}}S_{s,p}-2\left(1-\gamma_{\frac{sp}{2}}\right)\bar{C}_\gamma  m_{sp+\gamma}^N+2C\epsilon(N) m_{\gamma+2}^N.
\end{eqnarray}
{\bf Step 2:}  Reduce the problem to the classical case of  \cite{ACGM:2013:NAC}.\\
We define 
$$E_s^m(t,z):=\sum_{p=0}^mm^N_{sp}(t)\frac{z^p}{p!},$$
and
$$I_{s,\gamma}^m(t,z):=\sum_{p=0}^mm^N_{sp+\gamma}(t)\frac{z^p}{p!}.$$
Let $s$ be in $[\gamma,1]$ and $p_0>\frac{2}{s}$. We reuse $(\ref{ThmUpperMaxwellBoundsEq7})$ with $a<\min\{1,a_0\}$ to get
\begin{eqnarray}\label{ThmUpperMaxwellBoundsEq12}\nonumber
\frac{d}{dt}\sum_{p=p_0}^{m}m_{s p}\frac{a^p}{p!}&\leq& \sum_{p=p_0}^{m}\frac{a^p}{p!}(2\gamma_{\frac{\gamma p}{2}}S_{\gamma,p}-2\bar{C}_\gamma \left(1-\gamma_{\frac{sp}{2}}\right) m_{s p+\gamma}^N+2C\epsilon(N) m_{\gamma+2}^N)\\\nonumber
&\leq& \sum_{p=p_0}^{m}\frac{a^p}{p!}2\gamma_{\frac{s p}{2}}S_{s,p}-K_1I_{s,\gamma}^m+K_1\sum_{p=0}^{p_0-1}m_{s p+\gamma}^N\frac{a^p}{p!}\\\nonumber
& &+\left(\sum_{p=p_0}^{m}\frac{a^p}{p!}\right)K_2\epsilon(N)m_{\gamma+2}^N\\
&\leq& \sum_{p=p_0}^{m}\frac{a^p}{p!}2\gamma_{\frac{s p}{2}}S_{s,p}-K_1I_{s,\gamma}^m+K_1\sum_{p=0}^{p_0-1}m_{s p+\gamma}^N\frac{a^p}{p!}\\\nonumber
& &+\left(\sum_{p=p_0}^{m}\frac{a^p}{p!}\right)K_2\epsilon(N)(C(\epsilon) m_{\gamma}^N+\epsilon m_{s p_0+\gamma}^N)\\\nonumber
&\leq& \sum_{p=p_0}^{m}\frac{a^p}{p!}2\gamma_{\frac{s p}{2}}S_{\gamma,p}-K_3I_{s,\gamma}^n+K_4\sum_{p=0}^{p_0-1}m_{s p+\gamma}^N\frac{a^p}{p!},
\end{eqnarray}
where in the third inequality, we use Young's inequality
$$m_{\gamma+2}^N\leq C(\epsilon) m_{\gamma}^N+\epsilon m_{s p_0+\gamma}^N,$$
in the fourth inequality, we suppose that $N$ is sufficiently large, such that $K_1$ could absorb the constants of the last term. Once we have $(\ref{ThmUpperMaxwellBoundsEq12})$ the proof  could be proceeded in exactly the same way as the proof of the classical case (Theorem 2 \cite{ACGM:2013:NAC}).
\end{proof}
\section{Conclusion}
For the last two decades, nonlinear approximation based on wavelets has become one of the most important theories in scientific computing and the theory for elliptic equations has been fully developed (\cite{DeVore:2007:OC},\cite{DeVore:1998:NA},\cite{DaubechiesDeVore:2003:ABF},\cite{Cohen:2002:AMP},\cite{Dahmen:1995:MTS}). This paper is the first bridge between the two important theories: kinetic and nonlinear approximation. The strategy is based on  a new way of constructing an adaptive non-uniform mesh and a new filtering technique. The non-uniform mesh is created by a wavelet "support-stretching" technique: we stretch supports of wavelets defined in a bounded domain to the entire space to get a new "nonlinear basis", which are "the approximants" of our nonlinear approximation and solve the problem on the whole space. In our approximation, the lower-upper Maxwellian bounds play the role of a preconditioning technique. Our wavelet filtering technique designed to preserve the properties of propagation of polynomial and exponential moments is inspired by Zuazua's Fourier filtering technique in Control Theory (\cite{Zuazua:2005:POC},\cite{Zuazua:CNA:2006}). We have provided a complete theory for the method. Our nonlinear approximation solves the equation without having to impose non-physics conditions on the equation and could also be considered as an equivalent strategy with the Absorbing Boundary Conditions in PDEs theory (\cite{EngquistMajda:1977:ABC},\cite{EngquistMajda:1979:RBC}) for kinetic integral equations like Boltzmann equations, coagulation equations...    The reason that some physical properties of the Boltzmann equation could not be preserved through classical Discrete Velocity Models is that the convolution structure is destroyed. One of the reasons that make Fourier basis not an ideal choice for spectral approximations is that it could not preserve the structure of the collision operator, for example the "coercivity" property of the "gain" part of the collision operator, which is due to the non-positivity of the projection of a positive function onto its Fourier components as well as the effect of the Gibbs phenomenon. With suitable wavelet basis, our approach perfectly preserves the structure of the equation, therefore it is quite normal that most physical properties of the solution are reflected in the approximate solutions, while previous strategies could not. Our frame work also gives a unified point of view for the two available methods, Fourier Spectral Methods and Discrete Velocity Models: these strategies could be considered as special cases of our method in the sense that our approximation could produce spectral methods as well as discrete velocity models by using different wavelet basis; moreover, they are nonlinear and adaptive. One of the main advantages of adaptive, nonlinear approximations is that they are much cheaper (\cite{DeVore:2007:OC},\cite{DeVore:1998:NA}) than normal ones. Our method is therefore cheap and quite easy to use.  Moreover, it has a spectral accuracy. In this paper, we only give a theory of this method for cutoff collision kernels, however it would be not difficult to extend it to the non-cutoff and non-homogeneous cases. Numerical tests and accelerations of our adaptive spectral approximation will be provided in a coming paper. 
\\ {\bf Acknowledgements.} The author would like to sincerely thank his advisor, Professor Enrique Zuazua and his second advisor, Professor Miguel Escobedo for their great support, deep inspiration, constant encouragement and kindness. He would also like to express his gratitude to Professor Ronald Devore for the course and fruitful discussions on the Nonlinear Approximation Theory. The author has been supported by Grant MTM2011-29306-C02-00, MICINN, Spain, ERC Advanced Grant FP7-246775 NUMERIWAVES, and Grant PI2010-04 of the Basque Government. 
\bibliographystyle{plain}\bibliography{ApproximatingBoltzmann}
\end{document}